\documentclass{amsart}

\usepackage{fullpage}
\usepackage{amssymb}
\usepackage{amsfonts}
\usepackage{amsthm}
\usepackage{amsmath}
\usepackage{mathtools}
\usepackage{bbm}
\usepackage{bussproofs}
\usepackage{mathrsfs}
\usepackage{tikz, tikz-cd}
\usepackage[all,2cell]{xy}
\usepackage{dutchcal}
\usepackage{ stmaryrd } 
\usepackage{todonotes}
\usepackage{hyperref}
\usepackage[graphicx]{adjustbox}
\usepackage{graphicx}

\tikzset{%
	symbol/.style={%
		draw=none,
		every to/.append style={%
			edge node={node [sloped, allow upside down, auto=false]{$#1$}}}
	}
}

\usetikzlibrary{matrix,arrows}

\newtheorem{Theorem}[equation]{Theorem}

\newtheorem{proposition}[equation]{Proposition}
\newtheorem{lemma}[equation]{Lemma}
\newtheorem{corollary}[equation]{Corollary}

\newtheorem{thmx}{Theorem}
\newtheorem{propx}[thmx]{Proposition}

\newtheorem{corx}[thmx]{Corollary}

\theoremstyle{definition}
\newtheorem{example}[equation]{Example}

\newtheorem{remark}[equation]{Remark}
\newtheorem{definition}[equation]{Definition}

\DeclareMathOperator{\Ascr}{\mathscr{A}}
\DeclareMathOperator{\Bscr}{\mathscr{B}}
\DeclareMathOperator{\Cscr}{\mathscr{C}}
\DeclareMathOperator{\Dscr}{\mathscr{D}}
\DeclareMathOperator{\Escr}{\mathscr{E}}
\DeclareMathOperator{\Fscr}{\mathscr{F}}
\DeclareMathOperator{\Gscr}{\mathscr{G}}
\DeclareMathOperator{\Hscr}{\mathscr{H}}
\DeclareMathOperator{\Iscr}{\mathscr{I}}

\DeclareMathOperator{\Rscr}{\mathscr{R}}

\DeclareMathOperator{\Tscr}{\mathscr{T}}

\DeclareMathOperator{\Abb}{\mathbb{A}}

\DeclareMathOperator{\Fbb}{\mathbb{F}}

\DeclareMathOperator{\Ibb}{\mathbb{I}}

\DeclareMathOperator{\Qbb}{\mathbb{Q}}

\DeclareMathOperator{\Sbb}{\mathbb{S}}
\DeclareMathOperator{\Tbb}{\mathbb{T}}
\DeclareMathOperator{\Ubb}{\mathbb{U}}
\DeclareMathOperator{\Vbb}{\mathbb{V}}

\DeclareMathOperator{\Ccal}{\mathcal{C}}

\DeclareMathOperator{\Ocal}{\mathcal{O}}

\DeclareMathOperator{\pfrak}{\mathfrak{p}}

\DeclareFontFamily{U}{min}{}
\DeclareFontShape{U}{min}{m}{n}{<-> s*[0.95] dmjhira}{}

\DeclareMathSymbol{\mathinvertedexclamationmark}{\mathord}{operators}{'074}
\DeclareMathSymbol{\mathexclamationmark}{\mathord}{operators}{'041}
\makeatletter
\newcommand{\raisedmathinvertedexclamationmark}{%
	\mathord{\mathpalette\raised@mathinvertedexclamationmark\relax}%
}
\newcommand{\raised@mathinvertedexclamationmark}[2]{%
	\raisebox{\depth}{$\m@th#1\mathinvertedexclamationmark$}%
}
\makeatother

\newcommand{\bang}{\mathexclamationmark}
\newcommand{\gnab}{\raisedmathinvertedexclamationmark}

\DeclareMathOperator{\SManPc}{\mathbf{SMan}_{paracompact}}

\DeclareMathOperator{\qsSch}{\mathbf{qsSch}}

\DeclareMathOperator{\sym}{Sym}

\DeclareMathOperator{\fTan}{\mathfrak{Tan}}

\DeclareMathOperator{\id}{id}

\DeclareMathOperator{\Set}{\mathbf{Set}}

\DeclareMathOperator{\R}{\mathbb{R}}

\DeclareMathOperator{\Z}{\mathbb{Z}}
\DeclareMathOperator{\N}{\mathbb{N}}

\DeclareMathOperator{\Crig}{\mathbf{Crig}}
\DeclareMathOperator{\Cring}{\mathbf{Cring}}

\DeclareMathOperator{\Ab}{\mathbf{Ab}}

\DeclareMathOperator{\fCat}{\mathfrak{Cat}}

\DeclareMathOperator{\Adj}{\mathsf{Adj}}

\DeclareMathOperator{\Open}{\mathbf{Open}}

\DeclareMathOperator{\res}{res}

\DeclareMathOperator{\SMan}{\mathbf{SMan}}
\DeclareMathOperator{\op}{op}

\DeclareMathOperator{\Loc}{\mathsf{Loc}}
\DeclareMathOperator{\Locfp}{\mathsf{Loc}_{f.p.}}

\DeclareMathOperator{\MMod}{\mathcal{Mod}}
\DeclareMathOperator{\CCalg}{\mathcal{CAlg}}
\DeclareMathOperator{\und}{Und}
\DeclareMathOperator{\Spec}{Spec}
\DeclareMathOperator{\rSch}{\mathbf{redSch}}
\DeclareMathOperator{\red}{red}

\newcommand{\CAlg}[1]{\mathbf{CAlg}_{#1}}

\newcommand{\Par}[2]{\mathsf{Par}({#1},{#2})}

\newcommand{\Und}[1]{\und_{#1}}

\newcommand{\Sym}[1]{\sym_{#1}}

\newcommand{\RelSpec}[1]{\underline{\Spec}_{#1}}

\newcommand{\RelSym}[1]{\underline{\sym}_{{#1}}}

\newcommand{\RelUnd}[1]{\underline{\und}_{{#1}}}

\let\emptyset\varnothing
\let\epsilon\varepsilon

\DeclareMathOperator{\Sch}{\mathbf{Sch}}
\DeclareMathOperator{\QCoh}{\mathbf{QCoh}}
\DeclareMathOperator{\QQCoh}{\mathcal{QCoh}}
\DeclareMathOperator{\QQCAlg}{\mathcal{QCohCAlg}}
\DeclareMathOperator{\Bicat}{\mathsf{Bicat}}

\DeclareMathOperator{\cnst}{cnst}

\DeclareMathOperator{\pr}{pr}
\DeclareMathOperator{\DBun}{\mathbf{DBun}}

\let\emptyset\varnothing
\let\epsilon\varepsilon

\newcommand{\Mod}[1]{{#1}\textnormal{-}\mathbf{Mod}}
\newcommand{\Kah}[2]{\Omega_{{#1}/{#2}}}

\UseTwocells

\numberwithin{equation}{subsection}

\title{A Deep Dive Into the Tangent Category of Schemes}
\author{Geoff Vooys}
\date{\today}

\begin{document}

\begin{abstract}
In this largely expository paper we provide a deep and explicit exploration and exposition of the tangent structure on the category of schemes $\mathbf{Sch}_{/S}$ whose tangent functor $T(X) = T_{X/S}$ is the relative tangent scheme of Grothendieck described in \emph{{\'E}l{\'e}ments de G{\'e}om{\'e}trie Alg{\'e}brique} 4. In particular we provide explicit descriptions of the ways that the bifibration of quasicoherent sheaves and bifbration of quesicoherent sheaves of algebras over schemes may be built from the ways in which the bifibrations of modules and commutative algebras over commutative rings interact. We also show the ways in which these interactions give rise to an explicit description of the standard tangent structure on the category of schemes in terms of sheaves of K{\"a}hler differentials, properties of the relative spectrum functor, and more. Finally, we show that quasi-coherent sheaves can be reconstructed from their category of differential bundles by showing that for quasi-separated schemes $X$ and $Y$, there is an isomorphism $X \cong Y$ if and only if there is an equivalence of category $\DBun(X) \simeq \DBun(Y)$.
\end{abstract}


\maketitle
\tableofcontents

\section{Introduction}

As a trio of mathematical subjects, algebraic geometry, differential geometry, and category theory sit as a three-body system with porous and non-rigid boundaries between them. Ideas and techniques from any of the three areas often get abstracted, translated, and used in the remaining two subjects and then robustly re-applied to the original area in a fluid and dynamic fashion. Tangent category theory (cf.\@ Definition \ref{Defn: Section Tangent: Tangent category} for the definition) is an active area of study in modern category theory and a location where it is particularly easy and natural to see this interplay between abstraction, translation, and re-application in action. 

Part of the reason this process is so visible in tangent category theory comes from the fact that while tangent categories first were discovered by Rosick{\'y} in \cite{Rosicky} in the 1980's, they were re-discovered by Cockett and Cruttwell in \cite{GeoffRobinDiffStruct} with two primary purposes. The first purpose was to provide an abstract categorical framework with which to do differential geometric reasoning; the second purpose was to produce a unifying setting to compare and contrast the semantics of differential linear logic, differential geometry, synthetic differential geometry, and more. 

Tangent category theory has proved to be very fruitful in that it has provided a sort of structural neutral ground to compare and contrast different types of geometry which mathematicians study. For example, in \cite{GeoffRobinBundle} Cockett and Cruttwell developed a theory of what are called differential bundles (cf.\@ Definition \ref{Defn: Section Tangent: Differential Bundle} below). Differential bundles are objects in tangent categories which generalize the notion of vector bundles in differential geometry and provide a useful way in which to perform structural arguments and techniques internal to tangent categories in much the same way that one works with vector bundles and modules/quasi-coherent sheaves in the setting of smooth manifolds and schemes, respectively. In fact, in \cite[Theorem 1]{BenVectorBundles} MacAdam proved that the notion of differential bundles in the category of paracompact smooth manifolds $\SManPc$ coincides with vector bundles\footnote{In \cite{BenVectorBundles} MacAdam used a characterization that a morphism $f:X \to Y$ of smooth manifolds was a submersion if and only if the morphism $TX \to X \times_Y TY$ given by $(x,v) \mapsto (x,D[f]v)$ admits a section. However, such a condition is equivalent to having a smooth partition of unity for $f$, and in general one can only ensure such a partition exists in the case that $X$ is paracompact.} (and in particular that fibration $\DBun(\SManPc)$ corresponds to the fibration of vector bundles $\mathbf{VecBun}(\SManPc)$) while in \cite[Theorem 4.28]{GeoffJSDiffBunComAlg} Cruttwell and Lemay proved that for any $S$-scheme $X$, the category of differential bundles over $X$ are opposite-equivalent to the category of quasi-coherent sheaves over $X$. Additionally, in \cite[Theorem 3.28]{VooysInd} it is shown that if $\Cscr$ is a tangent category then so too is it $\operatorname{Ind}$-completion $\operatorname{Ind}(\Cscr)$; this was then used to describe a tangent structure on the category of formal schemes (cf.\@ \cite[Subsection 5.2]{VooysInd}) and the technology of said paper can be used to study a tangent structure on the category of infinite-dimensional smooth manifolds (with the convention that an infinite dimensional smooth manifold is precisely an ind-smooth manifold, i.e., every infinite-dimensional smooth manifold $M$ is the filtered colimit of its finite dimensional submanifolds). These structural results involving differential bundles are not only important for defining and working with the tangent categories of smooth manifolds and $S$-schemes, respectively, but also lend credence to the analogies that ``vector bundles are like modules for smooth manifolds'' and that differential bundles are to tangent categories what modules are to rings are to vector bundles of smooth manifolds. 

Tangent categories have not just been applied to vector bundle theory and to quasi-coherent sheaf theory. For instance, in \cite{DoretteMe} D.\@ Pronk and the author of this paper used tangent category theory to prove that the $2$-category $\fTan$ admits pseudolimits of pseudofunctors which factor as, for a $1$-category $\Cscr$,
\[
F:\Cscr^{\op} \to \fTan_{\operatorname{strong}} \to \fTan
\] 
and showed how to used this to construct tangent structures for use in equivariant algebraic geometry and equivariant differential geometry. More explicitly, \cite{DoretteMe} gives a definition of tangent structures which are sensibly defined for use when smooth group schemes on act on schemes and when Lie groups with finitely many connected components act  on smooth manifolds, respectively; cf.\@ \cite[Sections 5 -- 8]{DoretteMe}). In different directions, \cite{GeoffRory} Cruttwell and Luschyn-Wright defined a tangent categorical version of sector form and de Rham cohomology; in \cite{GeoffRobinConnections} Cockett and Cruttwell extended the definition of a connection to tangent categories; and in \cite{RoryConnectionsTanCAts} Luschyn-Wright significantly expanded upon  both the general study of connections, differential bundles, and the ways in which they interact. We also have seen tangent category theory applied to theoretical computer science in \cite{JonathanThesis}, \cite{GeoffJonBenACT2019}, \cite{ReverseDeriv}, and then again in \cite{GeoffJSReverse} where reverse tangent categories were defined. We have also seen in \cite{GeoffJSElias} that Cruttwell, Lemay, and Vandenberg used connections in tangent category theory to study (and provide explicit examples of) connections as they appear in algebraic geometry. In particular, they show in \cite[Theorem 8]{GeoffJSElias} that module connections of quasi-coherent sheaves on affine schemes correspond to tangent-categorical connections and then deduce in \cite[Corollary 4.11]{GeoffJSElias} that if $f:X \to S$ is a morphism of schemes and $\Fscr$ is a quasi-coherent sheaf on $X$, then to give an $S$-connection on $\Fscr$ it is necessary and sufficient to define a tangent-categorical connection on $\Fscr$.

In addition to the applications of tangent category theory to schemes (cf.\@ \cite{GeoffJSDiffBunComAlg}, \cite{GeoffJSElias}, and \cite{JSMeMapFlavoursInTanCats}), the tangent category of schemes appears as an important example of a tangent category in many tangent category papers (cf.\@ \cite{GeoffRobinDiffStruct}, \cite{GeoffRobinBundle} for the first mentions, albeit in an expository fashion, of the tangent functor on schemes\footnote{The explicit tangent structure on $\Sch$ appears only implicitly. It is given by a combination of a ring-theoretic version of \cite[Proposition 5.16, Proposition 5.17, Corollary 5.18]{GeoffRobinDiffStruct} in order to write the tangent scheme $T_{X/\Z} := [\Spec \Z[x]/(x^2),X]$ as an exponential of $X$ by the spectrum of the ring of dual numbers $\Z[x]/(x^2)$.}; cf. \cite[Example 2.iii]{GarnerEmbeddingTanCat} or \cite[Section 6]{DoretteMe}). Consequently, understanding the tangent category of schemes has been important to category theorists for quite some time. However, the development and techniques that tangent-category theorists tend to employ when working with schemes avoid directly using quasi-qoherent sheaves and Zariski descent/gluing. In addition to the importance of $\Sch_{/S}$ to category-theorists, the development of tangent-categorical techniques come with direct applications to algebraic geometry. Thus, in order to help category theorists use algebraic-geometric techniques and also to help algebraic geometers use tangent-categorical techniques, the author finds that it is important to provide a careful, explicit, and fully detailed description of the tangent category $\Sch_{/S}$, how its various structure morphisms and properties arise, how they may be used to provide new perspectives in scheme theory, and how these techniques interact with the gluing and descent-theory used by algebraic geometers.

In this expository paper we do precisely what is suggested above: we give an explicit and careful description of the tangent category of schemes, the ingredients necessary to define said tangent structure, and the recipe describing the way said tangent structure is built as the gluing of maps between affine schemes through Zariski descent. Because of the expository and careful nature of this paper, much of this paper contains various folkloric technical results which are obvious to experts in category theory, algebraic geometry, and other related disciplines. However, because we are focusing on bridging two somewhat disparate but related areas of mathematics\footnote{And in particular because category theorists need not know many of the foundational results in algebraic geometry while algebraic geometers need not know some of the categorical algebra or tangent category.}, we err on the side of providing all the details of our constructions and techniques.

\subsection{Technical and Expository Contributions: Structure of the Paper}
In the opinion of the author, perhaps the most important technical aspects of this paper are in writing down explicitly, carefully, and fully the following (sub)sections. Additionally, the following results are worth highlighting to \emph{all} practitioners of both algebraic geometry and category theory for a source of examples, potential misunderstandings to circumvent, and techniques of interest.

First, Section \ref{Section: Background Alg} simply recalls the background techniques and results from commutative algebra we use in this paper. While all the results in this section are individually well-known, they are less well-known in their fibration and pseudofunctorial guises. Of particular interest here for framing these well-known algebraic results with a fibrational lens are Proposition \ref{Prop: Section Background in Algebra: Underlying commutes strictly with restriction}, Proposition \ref{Prop: Section Background Alg: Sym and Und adjunction between CAlg and Mod fibrations over crig}, and Corollary \ref{Cor: Section Background Alg: Pseudonat for Sym}. Another important result to keep in mind for this section is the final result of Section \ref{Section: Background Alg}: the cocommutative Hopf algebra structures on the commutative $A$-algebras $\Sym{A}{M}$ assemble to a pseudonatrual transformation . While this is well-known, in order to make sure all $i$'s are dotted and $t$'s are crossed we are required to prove this and hence declare it as an explicit proposition.

\begin{propx}[{cf.\@ Proposition \ref{Prop: Section Background Alg: Strictness of Hopf for Restriction of Scalars}}]
There is a strict transformation of pseudofunctors:
\[
\begin{tikzcd}
	\Cring^{\op} \ar[rrr, bend left = 30, ""{name = U}]{}{\Mod{(-)}} \ar[rrr, bend right = 30, swap, ""{name = D}]{}{\mathbf{Hopf}^{\operatorname{cocom}}(-)} & & & \fCat
	\ar[from = U, to = D, Rightarrow, shorten <= 4pt, shorten >= 4pt]{}{\widehat{\Sym{(-)}}}
\end{tikzcd}
\]
In particular, for any map of commutative rings $f:R \to S$ and for any $S$-module $M$, the induced ring map
\[
\sigma_{M,F}:\Sym{R}(W_f(M)) \to \Sym{S}(M)
\]
is a morphism of cocommutative Hopf algebras.
\end{propx}

In Section \ref{Section: QCoh} we take the commutative algebra studied in Section \ref{Section: Background Alg} and extend it to the theory of quasi-coherent sheaves. Because the category $\QCoh(X)$ is less well-known to our category-theorist audience (and because the pseudouniversal description need not be the most helpful description for working explicitly with quasi-coherent sheaves), we take a very ``from scratch'' approach to working with and defining $\QCoh(X)$. In particular, in Section \ref{Subsection: QCoh Modules} we spend a great deal of time defining and working with quasi-coherent sheaves from basics, as it were. For instance, we present Example \ref{Example: EZPZ QCoh} to show the necessity of working with quasi-coherent sheaves and not the fully general sheaves of $\Ocal_X$-modules. We also give an explicit proof of a folkloric result in categorical algebraic geometry which is, in my opinion, both not at all as well-known as it should be (although it is known to many experts) and also fundamental to working with quasi-coherent sheaves from a category-theoretic lens. In particular, we prove that for any scheme $X$ and for any affine open cover $\lbrace \gamma_i:U_i \to X \; | \; i\in I\rbrace$, the category $\QCoh(X)$ arises as a pseudolimit of the categories $\QCoh(U_i)$ as they are glued over the scheme-theoretic intersections $\QCoh(U_i \times_X U_j)$. That is, the collection of all invertible $2$-cells 
\[
\begin{tikzcd}
\QCoh(X) \ar[rrr, ""{name = U}]{}{\gamma_i^{\ast}} \ar[d,swap]{}{\gamma_j^{\ast}} & & & \QCoh(U_i) \ar[d]{}{(\pi_0^{ij})^{\ast}} \\
\QCoh(U_j) \ar[rrr, swap, ""{name = D}]{}{(\pi_1^{ij})^{\ast}} & & & \QCoh(U_i \times_X U_j)
\ar[from = U, to = D, Rightarrow, shorten <= 4pt, shorten >= 4pt]{}{\cong}
\end{tikzcd}
\]
for all $i, j \in I$ in the $2$-category $\fCat$ of categories forms a psuedouniversal pseudocone (in $\fCat$).
\begin{thmx}[{cf.\@ Theorem \ref{Thm: Section Background Scheme: QCoh is a pseudolimit}}]
Let $X$ be a scheme and let $C = \lbrace \gamma_i:U_i \to X \; | \; i \in I \rbrace$ be an affine open cover of $X$. Then $\QCoh(X)$ is the pseudolimit in the $2$-category $\fCat$ of the cospans 
\[
\QCoh(U_i) \xrightarrow{(\pi_0^{ij})^{\ast}} \QCoh(U_i \times_X U_j) \xleftarrow{(\pi_1^{ij})^{\ast}} \QCoh(U_j)
\]
for all $i, j \in I$. That is, $\QCoh(X)$ is the pseudouniversal pseudocone over the cospans above when equipped with the compositor invertible $2$-cells:
\[
\begin{tikzcd}
	& \QCoh(U_j) \ar[dr]{}{(\pi_1^{ij})^{\ast}} \\
	\QCoh(X) \ar[ur]{}{\gamma_j^{\ast}} \ar[dr, swap]{}{\gamma_i^{\ast}} &	& \QCoh(U_i \times_X U_j) \\
	& \QCoh(U_i) \ar[ur, swap]{}{(\pi_0^{ij})^{\ast}}
	\ar[from = 1-2, to = 3-2, Rightarrow, shorten <= 4pt, shorten >= 4pt]{}{\cong}
\end{tikzcd}
\]
\end{thmx}
We then use this perspective on quasi-coherent sheaves to additionally conclude that $\QCoh(X)$ is the pseudolimit of the categories $\QCoh(U)$ as $U$ runs through the poset $\mathbf{AffOp}(X)$ of affine open subschemes of $X$. This allows us to follow with a clean proof of the fact that $\QCoh(X)$ is a locally presentable category.
\begin{thmx}[cf. Theorem \ref{Thm: Section QCoh: Qcoh is pseudolim over poset of affine opens}]
Let $X$ be a scheme and let $\mathbf{AffOp}(X)$ be the poset of affine open subschemes of $X$, ordered by inclusion. Then there is a pseudolimit decomposition
\[
\QCoh(X) \cong \operatorname*{pslim}_{U \in \mathbf{AffOp}(X)} \QCoh(U).
\]
\end{thmx}
\begin{corx}[cf. Corollary \ref{Cor: Section QCoh: QCoh Locally presentable}]
For any scheme $X$ the category $\QCoh(X)$ is a locally presentable category.
\end{corx}

In addition to the pseudouniversal perspective we describe above, we sketch the proof of the result which shows that the pullback functor $f^{\ast}$ of quasicoherent sheaves admits a right adjoint $f_{\square}$ for any morphism $f:X \to Y$ of schemes. We also take care to indicate that the functor $f_{\square}$ is naturally isomorphic to the direct image functor $f_{\ast}$ in the case that $f$ is both quasi-compact and quasi-separated:
\begin{propx}[{cf.\@ Propositions \ref{Prop: Section Background Scheme: Right adjoint for quasicoherent sheaves}, \ref{Prop: Section Background Scheme: Qcqs maps have quasicompact pushforward}}]
For any morphism $f:X \to Y$ of schemes there is an adjunction:
\[
\begin{tikzcd}
	\QCoh(X) \ar[rr, bend right = 20, swap, ""{name = R}]{}{f_{\square}} & & \QCoh(Y) \ar[ll, bend right = 20, swap, ""{name = L}]{}{f^{\ast}}
	\ar[from = L, to = R, symbol = \dashv]
\end{tikzcd}
\]
Furthermore, when $f$ is quasi-compact and quasi-separated (qcqs) then $f_{\square} \cong f_{\ast}$ for $f_{\ast}$ the direct image functor of sheaves.
\end{propx}

We additionally describe carefully in Section \ref{Subsection: QCoh CAlg} the category of quasicoherent sheaves of commutative $\Ocal_X$-algebras and the corresponding relative symmetric algebra and underlying quasicoherent sheaf functors $\RelSym{\Ocal_X}:\QCoh(X) \to \QCoh(X,\CAlg{\Ocal_X})$ and $\RelUnd{\Ocal_X}:\QCoh(X,\CAlg{\Ocal_X}) \to \QCoh(X)$. In Section \ref{Subscetion: Functoriality between QCoh and QCoh CAlg} we then show that the relative symmetric algebra and the underlying quasi-coherent sheaf functors both assemble to fibrations over $\Sch_{/S}$ and induce an adjunction of fibrations:
\begin{propx}[{cf.\@ Proposition \ref{Prop: Section Background Scheme: Rel Sym Rel Und adjunction at fibration level}}]
	For any base scheme $S$ there is an adjunction between fibrations:
\[
\begin{tikzcd}
	\QQCAlg \ar[dr]{}{} \ar[rr, bend right = 15, swap, ""{name = R}]{}{\RelUnd{(-)}} & & \QQCoh \ar[ll, swap, bend right = 15, ""{name = L}]{}{\RelSym{(-)}} \ar[dl] \\
	& \Sch_{/S}
	\ar[from = L, to = R, symbol = \dashv]
\end{tikzcd}
\]
\end{propx}
This then leads to a clean conceptual proof of the fact that the relative symmetric algebra functors commute up to isomorphism with pullback functors of quasicoherent sheaves (cf.\@ Corollary \ref{Cor: Section Background Scheme: The pullback Sym isos for qcoh}).

After establishing the level on which the symmetric algebra and underlying quasi-coherent sheaf functors live, we move on in Section \ref{Subsection: Coalgebra structure on Relative Sym} to relativizing Proposition \ref{Prop: Section Background Alg: Symmetric Alg is Coalg}. That is, we work towards proving and constructing the fact that the relative symmetric algebra functor $\RelSym{\Ocal_X}$ lifts to a functor from $\QCoh(X)$ to cocommutative Hopf algebra objects in $\QCoh(X,\CAlg{\Ocal_X})$:
\begin{thmx}[{cf.\@ Theorem \ref{Thm: Section Background Scheme: Hopf algebras for RelSym}, Corollary \ref{Cor: Section Background Scheme: Hopf algebra result in opposite land}}]
Let $X$ be a scheme. Then for any quasi-coherent sheaf $\Fscr$, the quadruple
\[
\left(\RelSym{\Ocal_X}(\Fscr), \nabla_{\Fscr}, \epsilon_{\Fscr}, S_{\Fscr}\right)
\]
is a cocommutative Hopf algebra in $\QCoh(X,\CAlg{\Ocal_X})$ and for any quasi-coherent sheaf map $\varphi:\Fscr \to \Gscr$ $\RelSym{\Ocal_X}(\varphi)$ is a morphism of Hopf algebras. In particular, there is an induced functor
\[
\widehat{\RelSym{\Ocal_X}}:\QCoh(X) \to \mathbf{Hopf}^{\operatorname{cocom}}(\QCoh(X))
\]
which sends quasi-coherent sheaves $\Fscr$ to the corresponding Hopf algebra on $\RelSym{\Ocal_X}(\Fscr)$ above and sends sheaf maps $\varphi$ to $\RelSym{\Ocal_X}(\varphi)$. In particular, taking opposites gives a functor to the category of Abelian group objects $\Ab\left(\QCoh(X,\CAlg{\Ocal_X})^{\op}\right)$.
\end{thmx}

In Section \ref{Subsection: Relative Spectrum} we get to the final purely scheme-theoretic results which we discuss carefully: that of the relative spectrum functor and the category of affine structure maps over a fixed base scheme (cf.\@ Definition \ref{Defn: Scheme Background: Affine Morphism}). The relative spectrum functor is a core construction used both by Grothendieck in \cite{EGA2} and \cite{EGA44}. Additionally, the construction of the tangent strucutre on $\Sch_{/S}$ uses this functor crucially. The relative spectrum functor comes down to the following process: associating to a quasi-coherent sheaf $\Fscr$ on an $S$-scheme $X$ a ``fibre bundle with coefficients in $\Fscr$'' over the scheme $X$. This construction is precisely the relative spectrum functor $\RelSpec{X}:\QCoh(X,\CAlg{\Ocal_X})^{\op} \to \Sch_{/X}$ (cf.\@ Definition \ref{Defn: Section Scheme Background: Relspec functor}), and working with it is the last major purely scheme-theoretic ingredient we need to understand before moving to working towards understanding \emph{differential} algebraic geometry. The functor $\RelSpec{X}$ also plays an important role in the proof of \cite[Theorem 4.28]{GeoffJSDiffBunComAlg} (which shows an equivalence $\DBun(X)^{\op} \simeq \QCoh(X)$), as it is the technical tool that allows one to move between quasi-coherent sheaves over $X$ to schemes $Y \to X$ with affine structure map. Because of this we spend time defining and working with both $\RelSpec{X}$ and affine morphisms in algebraic geometry. We ultimately conclude this section with proofs that the relative spectrum functor is an equivalence of categories, how it interacts with pullbacks, how it admits pushforwards by post-composing with affine structure maps, and also how it interacts with the Hopf-algebra valued functor $\widehat{\RelSym{\Ocal_S}}$:
\begin{thmx}[{cf.\@ Corollary \ref{Cor: Section Background Scheme: Equiv of Cats for Relspec}, Proposition \ref{Prop: Section Background Scheme: Incl of Aff is cont}, Corollary \ref{Cor: Section Background Scheme: Relative Spec of Relative Sym is functor into Ab Sch}}]
For any scheme $X$ if $\mathbf{Aff}_{/X}$ denotes the category of schemes over $X$ with affine structure map then there is an equivalence of categories:
\[
\begin{tikzcd}
	\mathbf{Aff}_{/S} \ar[rr, swap, bend right = 30, ""{name = R}]{}{X \mapsto f_{\ast}(\Ocal_X)} & & \QCoh(X, \CAlg{\Ocal_X})^{\op} \ar[ll, bend right = 30, swap, ""{name = L}]{}{\RelSpec{X}}
	\ar[from = L, to = R, symbol = \simeq]
\end{tikzcd}
\]
Furthermore:
\begin{enumerate}
	\item  The inclusion functor $\operatorname{incl}:\mathbf{Aff}_{/X} \to \Sch_{/X}$ is continuous;
	\item For any quasi-coherent sheaf $\Fscr$ on $X$, the scheme $\RelSpec{X}(\RelSym{\Ocal_X}(\Fscr))$ is an Abelian group object in $\Sch_{/X}$.
\end{enumerate}
\end{thmx}

After developing a working familiarity of working with quasi-coherent sheaves and schemes affine over a fixed base scheme, we revisit and recall the theory of K{\"a}hler differentials and derivations as they appear in commutative algebra and algebraic geometry in Section \ref{Section: Diffles}. The basic theory of derivations and K{\"a}hler differentials of commutative rings and commutative algebras is recalled in Section \ref{Subsection:Derivations for rings}; we assume the reader is at least passingly familiar with said theory at the ring-theoretic level and so provide a brief introduction. What is of interest and ``new'' here is simply organizing the fact that the construction of taking a commutative $R$-algebra and forming the module of relative K{\"a}hler differentials actually lifts to a functor from the category $\CAlg{R}$ to the module fibration $\MMod$ and some of the consequences of that construction.
\begin{propx}[{cf.\@ Proposition \ref{Prop: Section Kahlers: Kahler diffles are functors}}]
	The K{\"a}hler differentials arise as a functor
\[
\Kah{(-)}{R}:\CAlg{R} \to \MMod
\]
given on objects by $A \mapsto (A,\Kah{A}{R})$ and on morphisms $f:A \to B$ of $R$-algebras via
\[
\Kah{(-)}{R}(f) := (f,\mathrm{d}f):\left(A,\Kah{A}{R}\right) \to \left(B,\Kah{B}{R}\right).
\]
\end{propx}

In Section \ref{Subsection: Scheme Diffles} we change gears a bit and move to study differentials and derivations for schemes. It is worth noting that in this paper we take a particularly nonstandard approach towards defining the sheaf $\Kah{X}{S}$ of relative K{\"a}hler differentials. Instead of defining $\Kah{X}{S}$ as a pullback of the diagonal $\Delta_f:X \to X \times_S X$ for an $S$-scheme $f:X \to S$, we define $\Kah{X}{S}$ by a direct sheaf-theoretic argument and then prove that it represents the functor $\mathsf{Der}_S(\Ocal_X,-):\QCoh(X) \to \Mod{\Ocal_X(\lvert X \rvert)}$:
\begin{propx}[{cf.\@ Propositions \ref{Prop: Section Kahlers: Relative Kahlers are qcoh on X}, \ref{Prop: Section Kahlers: Kahler Diffles are Corepresenting}}]
Let $f:X \to S$ be a map of schemes. Then the presheaf of relative K{\"a}hler differentials
\[
\Kah{X}{S}\left(U\right) := \Kah{\Ocal_X(U)}{(f^{-1}\Ocal_S)(U)}
\] 
is a quasi-coherent sheaf on $X$ and induces a natural isomorphism
\[
\mathsf{Der}_S\left(\Ocal_X,-\right) \cong \QCoh(X)\left(\Kah{X}{S},-\right).
\]
\end{propx}
By using this construction and taking this perspective, we are able to give a clean and conceptual construction of the fact that the relative K{\"a}hler differentials arise as a functor between the category of $S$-schemes and the fibration of quasi-coherent sheaves.
\begin{propx}[{cf.\@ Proposition \ref{Prop: Section Kahlers: Functoriality of Kahlers for QCoh}}]
For any scheme $S$ there is a functor 
\[
\Kah{(-)}{S}:\Sch_{/S} \to \mathbb{QC}
\]
which sends schemes $X$ to $(X,\Ocal_X)$ and which sends morphisms $f:X \to Y$ in $\Sch_{/S}$ to the morphism $(f,u_f):(X,\Kah{X}{S}) \to (Y,\Kah{Y}{S})$ where $u_f$ is the map in the Relative Cotangent Complex indicated below:
\[
\begin{tikzcd}
	f^{\ast}\Kah{Y}{S} \ar[r]{}{u_f}  & \Kah{X}{S} \ar[r] & \Kah{X}{Y} \ar[r] & 0
\end{tikzcd}
\]
\end{propx}

The penultimate section of this paper, Section \ref{Section: Tan Cats}, concerns tangent categories and the tangent category of schemes in particular. In Section \ref{Subsection: Tan Cat Exposition} we recall the definitions of a tangent category (cf.\@ Definition \ref{Defn: Section Tangent: Tangent category}), differential bundles (cf.\@ Definition \ref{Defn: Section Tangent: Differential Bundle}) and their morphisms (cf.\@ Definition \ref{Defn: Section Tangent: Morphism of dbundles}). Because tangent categories are likely as familiar to the algebraic geometer as the sheaf of relative K{\"a}hler differentials on a scheme and the relative spectrum functor are to the category theorist, we have also taken the perspective that it is important to give a myriad and varied list of examples of tangent categories and their differential bundles. We also give a very explicit and full-detail description of the tangent category of affine schemes (over an affine base, of course) in Example \ref{Example: The tangent category of affine R schemes}, as understanding this is what the core of our construction of the tangent category of $S$-schemes is built upon.

In Section \ref{Subsection: Tangent Category of Schemes} we finally construct the tangent category of schemes by leveraging everything we built up to this point in the paper. We begin this journey by showing how to construct an ``infinitesimal $S$-scheme'' $W_S \cong S[\epsilon]$ for any scheme $S$ and then give an explict scheme-theoretic proof of the fact that there is an adjunction defining the tangent functor $(-) \times_S W_S \dashv T_{(-)/S}$. This shows that in particular, the tangent functor for $\Sch_{/S}$ is representable in the sense studied in \cite{GarnerEmbeddingTanCat}.
\begin{propx}[{cf.\@ Proposition \ref{Prop: Section Tangent: Tangent fucntor is a right adjoint}}]
Let $S$ be a scheme. Then there is an adjunction:
\[
\begin{tikzcd}
	\Sch_{/S} \ar[rr, bend right = 20, swap, ""{name = R}]{}{T_{(-)/S}} & & \Sch_{/S} \ar[ll, bend right = 20, swap, ""{name = L}]{}{(-) \times_S S[\epsilon]}
	\ar[from = L, to = R, symbol = \dashv]
\end{tikzcd}
\]
\end{propx}

After proceeding with some technical but relatively routine algebra, we arrive finally at the main goal of this paper: to provide an explicit and careful proof of the tangent structure on $\Sch_{/S}$. We also show that these tangent structures, as one varies the base schemes $S$, give rise to a pseudofunctor from $\Sch_{/S}^{\op}$ into the $2$-category of tangent categories. After this, we conclude this section of the paper by showing that the tangent structure on $\Sch_{/S}$ gives rise to a corresponding dual tangent structure on $\Sch_{/S}^{\op}$ which has, to the knowledge of the author, yet to be recorded so far.

\begin{thmx}[{cf.\@ Theorem \ref{Thm: Section Tangents: The Tangent Category of Schemes}}]
The category $\Sch_{/S}$ with tangent functor $T_{(-)/S}$ given in Definition \ref{Defn: Section Tangent: Tangent functor for schemes}, with bundle projection $p$ and zero section $0$ given in Definition \ref{Defn: Section Tangent: Zero and Projection Transformations}, addition $\operatorname{add}$ as given in Definition \ref{Defn: Section Tangent: The addition on the tangent bundle}, with vertical lift $\ell$ given in Definition \ref{Defn: Section Tangent: Vertical Lift over Schemes}, and with canonical flip given as in Definition \ref{Defn: Section Tangent: Canonical Flip} is a tangent category.
\end{thmx}
\begin{propx}[{cf.~ Proposition \ref{Prop: Section Tangent: PSeudofunctor in Tan}}]
For any base scheme $S$, there is a slice category pseudofunctor
\[
\mathsf{Sl}_{S}:\Sch_{/S}^{\op} \to \fTan_{\operatorname{str}}
\]
given by sending schemes $X$ to their slice category tangent structures $(\Sch_{/X}, T_{(-)/X})$ and given by sending morphisms $f:X \to Y$ to the strong tangent morphism $(f^{\ast},T_f)$ of Proposition \ref{Prop: Pullback functor strong tangent for schemes}.
\end{propx}
\begin{propx}[{cf.~ Proposition \ref{Prop: Section Tangents: Dual Numbers for Schemes Tangents}}]
The functor
\[
\left((-) \times_S W_S\right):\Sch_{/S}^{\op} \to \Sch_{/S}^{\op}
\]
induces a tangent structure on $\Sch_{/S}^{\op}$.
\end{propx}


In the final section of this paper, Section \ref{Section: Appies ofTan on Sch}, we apply the work done in this paper by showing that for any quasi-separated scheme $S$, the categories of differential bundles for each $S$-scheme are complete invariants on the category $\qsSch_{/S}$ of quasi-separated $S$-schemes (which has not, to the knowledge of the author, been observed prior to this paper). We do this first by proving if $\qsSch_{/S}$ denotes  the category of quasi-separated $S$-schemes, then $\DBun(X) \simeq \DBun(Y)$ if and only if $X \cong Y$ for $X, Y$ quasi-separated $S$-schemes. We do this by proving that $\qsSch_{/S}$ is a strict tangent subcategory of $\Sch_{/S}$ for any scheme $S$ and then showing that $\DBun_{\qsSch}(X) = \DBun_{\Sch}(X)$ for a quasi-separated $S$-scheme $X$. 

\begin{propx}[cf.\@ Propositions \ref{Prop: Section Gabriel: qsSch is tangent sbucat}, \ref{Prop: Section Gabriel: DBun in qs is Dbun in sch}]
For any scheme $S$, the category $\qsSch_{/S}$ of quasi-coherent $S$-schemes is a strict tangent subcategory of the tangent category $\Sch_{/S}$. Furthermore, for any quasi-separated $S$-scheme $X$, there is an equality of categories
\[
\DBun_{\qsSch}(X) = \DBun_{\Sch}(X).
\]
\end{propx}

\begin{thmx}[Tangent-Categorical Reconstruction Theorem; cf.\@ Theorem \ref{Thm: Tan Cat Recon Thm}]
For any quasi-separated schemes $X$ and $Y$, there is an isomorphism $X \cong Y$ if and only if there is an equivalence of categories $\DBun(X) \simeq \DBun(Y)$.
\end{thmx}



%
%
%
%
%
%

\subsection*{Relation to Other Work}
The work done in this paper arose as the intersection of two projects of the author with different collaborates. The first project, \cite{RobinMePartialMapRigs}, is a project with R.\@ Cockett on the tangent restriction structure on the partial map categories $\Par{\CAlg{R}^{\op}}{\Loc}$ and $\Par{\CAlg{R}^{\op}}{\Locfp}$ of $\CAlg{R}^{\op}$ for a commtutative \emph{rig}\footnote{A rig, also known as a semiring in algebra literature, is a ri\emph{n}g without \emph{n}egatives.} $R$ with restriction monics taken to be those given by localizatons (respectively finitely presented localizations).  This project required us to very carefully develop much of the theory of K{\"a}hler differentials and their interactions with tensor products, localizations, and the like in the rig-theoretic case. The main tools there lay in using the module and commutative algebra fibrations over $\Crig$ together with the symmetric algebra and underlying module morphisms of (op)fibrations. Much of the reason the tangent structure on $\CAlg{R}^{\op}$ works as it does lie in how the $\Sym{(-)}$ and $\Und{(-)}$ functors interact, how the K{\"a}hler differentials interact with tensors, and the way in which derivations fundamentally record differential geometric information. The techniques developed in that project led the author to many of the technical tools we use in this paper.

The second project, \cite{JSMeMapFlavoursInTanCats} with J.-S.\@ Lemay, develops a zoo of various structural maps in tangent categories which extend immersions, submersions, unramified morphisms, and local diffeomorphisms\footnote{As a historical note: the notion of submersions and local diffeomorphisms in tangent categories have been defined for some time (with submersions being known as $T$-submersions and local diffeomorphisms being known as $T$-{\'e}tale or {\'e}tale maps --- cf.\@ \cite{BenVectorBundles} for the introduction of $T$-submersions to the literature and \cite{GeoffMarcelloTSubmersionPaper} for a recent in-depth study on display $T$-submersions.; {\'e}tale maps were defined in \cite{GeoffJonBenACT2019} and redefined/reframed in \cite{BenThesis} and \cite{GeoffJSReverse}.} together with an extensive examination of how such morphisms behave. Example \ref{Example: Inclusion of varieties is not strong tangent mor} in particular was built on observations and patterns regarding relative tangent bundles developed in \cite{JSMeMapFlavoursInTanCats}.

\subsection*{Acknowledgments}

I would like to thank Robin Cockett for helpful discussions and aiding me in seeing the ways in which category-theorists would best benefit in seeing presentations of material which is traditionally scheme-theoretic in nature. I also want to thank Robin for all our discussions and work together in developing the material used in \cite{RobinMePartialMapRigs} as well as his insights in giving fibration-theoretic perspectives to the commutative algebraic material in Section \ref{Section: Background Alg}.

I would like to thank JS Lemay for his helpful discussions regarding framing some of the material in this paper and also in reminding me about dual tangent structures and how they allow us to deduce the existence of a tangent structure on $\Sch_{/S}^{\op}$. It is worth noting as well that the main impetus for writing this paper came from the related paper \cite{JSMeMapFlavoursInTanCats}. We needed a particularly strong knowledge of the tangent category of schemes in order to determine which morphisms of schemes coincide with the various tangent-categorical generalizations of immersions, submersions, and local diffeomorphisms. For the interested reader, details regarding these morphisms and the classifications in $\Sch_{/S}$ may be found in \cite[Sections 6 -- 10]{JSMeMapFlavoursInTanCats}.

I would like to thank Kristine Bauer for helpful discussions on writing and ways in which people are excited to see material discussed and presented. Additionally, I'd also like to thank both the Peripatetic Seminar (and especially Robin Cockett, Kristine Bauer, Berndt Brenken, Florian Schwarz, Xanna Little, Durgesh Kumar, Sam Steakley, and Saina Daneshmandjahromi), Deni Salja, and also the Voganish Research Group (and especially Clifton Cunningham, James Steel, Jos{\'e} Cruz, and Alex Cameron) for allowing me to inflict discussions of this topic upon them and see what they would  want to see and/or get out of a paper diving into the \emph{tangent category theory} of schemes.


\section{Category-Theoretic Background on Categories of Commutative Algebras and Modules over Affine Bases}\label{Section: Background Alg}


Much of the material in this section is formally similar to work done by R.\@ Cockett and myself in the forthcoming paper \cite{RobinMePartialMapRigs} which establishes these results in higher generality (\cite{RobinMePartialMapRigs} works with rigs instead of rings) and with more detail. As such, we give a brief account of the fibrations/pseudofunctors of modules and commutative algebras defined over/on the category $\Cring$. The approach here is more to summarize results than give complete proofs of correctness; the interested reader may consult \cite{RobinMePartialMapRigs} (when it is out and for the rig-theoretic versions of all statements here) or various other aspects of standard graduate commutative algebra texts. When appropriate, I have attempted to organize relevant references where they apply and simply recognize that what we are doing, to a large degree, is recognizing and restating old results in terms of their categorical structure.

In more detail: this section will focus on covering background material regarding the way in which the categories $\CAlg{R}$ and $\Mod{R}$ interact and behave over the base category $\Cring$ of commutative rings. These constructions behave as bifibrations over $\Cring$. That is, the corresponding fibrations arise as pseudofunctors from $\Cring^{\op}$ into the $2$-category $\Adj$ of adjoint functors between categories via the Grothendieck construction. We will focus more on framing the results we use rather than proving them explicitly and fully for the sake of time and to keep the page-count of this paper as low as possible.\footnote{The proofs of the results in this section are all individually well-known in the literature (they may be assembled out of the various piecemeal results on commutative rings and modules in most introductory graduate algebra textbooks such as \cite{Eisenbud}) and by applying basic fibration and pseudofunctor theory to our setting. The basic theory of fibrations, pseudofunctors, and the Grothendieck construction may be found in \cite[Chapter 8.1 -- 8.3]{BorceuxHandBook2} (for a purely categorical approach), \cite[Section 3]{Vistoli} (for an algebraic-geometry friendly approach), \cite[Chapter 1]{Jacobs} (for a categorical logic friendly approach), or in the original \cite[Expos{\'e} VI]{SGA1} (where the so-called Grothendieck Construction originally appeared).}

Before proceeding, we give the following remark/warning to the reader. We will frequently use without remark the fact that if $R$ is a commutative ring then there is an equality of categories 
\[\CAlg{R} = R \downarrow \Cring = {}^{R/}\Cring.
\]
That is, the category of $R$-algebras is equal to the coslice category of $R$ over $\Cring$. Note that we also use that if $R$ is a commutative \emph{rig} then $\CAlg{R} = R \downarrow \Crig$ as well.

\subsection{The Pseudofunctor and Bifibration of Commutative Algebras}\label{Subsection: Background Alg: Pseudofunctor of CAlg}
In this subsection we give a relatively terse introduction into the pseudofunctor $\CAlg{(-)}:\Cring^{\op} \to \fCat$ and its associated commutative algebra fibration. We start by recalling the $2$-category $\Adj$ of adjoints before beginning our view of the commutative algebra fibration and pseudofunctor in earnest.

\begin{definition}\label{Defn: Section Background Algebra: Two Cat of ADjoints}
The $2$-category $\Adj$ is the $2$-category defined as follows:
\begin{itemize}
	\item Objects ($0$-cells): Categories $\Cscr$.
	\item Morphisms ($1$-cells): A morphism $f:\Cscr \to \Dscr$ is an adjunction $f^{\ast} \dashv f_{\ast}:\Dscr \to \Cscr$, which we dispaly as:
	\[
	\begin{tikzcd}
	\Cscr \ar[rr, bend right = 20, swap, ""{name = R}]{}{f_{\ast}} & & \Dscr \ar[ll, swap, bend right = 20, ""{name = L}]{}{f^{\ast}}
	\ar[from = L, to = R, symbol = \dashv]
	\end{tikzcd}
	\]
	Note that we follow the topos-theoretic convention of \cite{ElephantVol1} and index the direction of our $1$-cells in the direction of the right adjoint (the so-called geometric direction).
	\item Transformations ($2$-cells): Mate\footnote{Mates are a way of translating $2$-cells between left adjoints to $2$-cells between right adjoints and vice-versa. More accurately, assume we have a $2$-category $\mathsf{C}$, adjoints $(\eta_0,\epsilon_0):\ell_0 \dashv r_0:x_0 \to y_0$ and $(\eta_1,\epsilon_1):\ell_1 \dashv r_1:x_1 \to y_1$, and $1$-cells $f:X_0 \to X_1$, $g:Y_0 \to Y_1$. Then the mate-calculus is the bijection of hom-sets $\mathsf{C}(X_0,Y_1)(\ell_1 \circ f, g \circ \ell_0) \cong \mathsf{C}(Y_0,X_1)(f \circ r_0, r_1 \circ g)$ induced by pasting a $2$-cell $\alpha: \ell_1 \circ f \Rightarrow g \circ \ell_0$ with the unit of one adjunction and the counit of the other. For details, textbook accounts may be found in \cite[Section V.7]{MacLaneCWM} (without the term mate appearing) or \cite[Definition 6.1.12]{TwoDimCat} at the level of bicategories.}-pairs $(\alpha, \hat{\alpha}):f \Rightarrow g$ where
	\[
	\begin{tikzcd}
	\Dscr \ar[rr, bend left = 30, ""{name = U}]{}{f^{\ast}} \ar[rr, bend right = 30, swap, ""{name = D}]{}{g^{\ast}} & & \Cscr
	\ar[from = U, to = D, Rightarrow, shorten <= 4pt, shorten >= 4pt]{}{\alpha}
	\end{tikzcd}
	\]
	is a natural transformation and $\hat{\alpha}$ is the corresponding mate transformation:
	\[
	\begin{tikzcd}
	\Cscr \ar[rr, bend left = 30, ""{name = U}]{}{g_{\ast}} \ar[rr, swap, bend right = 30, ""{name = D}]{}{f_{\ast}} & & \Dscr
	\ar[from = U, to = D, Rightarrow, shorten <= 4pt, shorten >= 4pt]{}{\hat{\alpha}}
	\end{tikzcd}
	\]
	\item Composition functors: as in the $2$-category $\fCat$ of categories.
\end{itemize}
\end{definition}
The reason we work with the $2$-category $\Adj$ lies in the usual Grothendieck Construction dictionary which passes between fibrations and pseudofunctors. Namely for any category $\Cscr$, bifibrations\footnote{Functors $p:\Escr \to \Bscr$ which simultaneously have the properties that $p$ is a fibration and also that $p^{\op}:\Escr^{\op} \to \Bscr^{\op}$ is a fibration, i.e., that $p$ is a fibration and opfibration simultaneously.} $p:\Escr \to \Cscr$ are equivalent to working with pseudofunctors defined on $\Cscr^{\op}$ taking values in $\Adj$. More explicitly, a bifibration $p:\Escr \to \Cscr$ arises from the Grothendieck construction $p \cong \operatorname{El}(F:\Cscr^{\op} \to \fCat)$ where $F$ is a pseudofunctor factoring as
\[
\begin{tikzcd}
\Cscr^{\op} \ar[rr]{}{F} \ar[dr, swap]{}{F} & & \fCat \\
 & \Adj \ar[ur, swap]{}{\pr_R}
\end{tikzcd}
\]
with $\pr_{R}:\Adj \to \fCat$ is the projection onto the right adjoint component of $\Adj$. The corresponding opfibration $p^{\op}:\Escr^{\op} \to \Cscr^{\op}$ arises from the Grothendieck construction applied to a pseudoufnctor $F:\Cscr \to \fCat$ which factors through the left adjoint projection $\pr^{L}:\Adj^{\op} \to \fCat$. We will use this observation to describe the bifibration $p:\CCalg \to \Cring$ of commutative algebras.

\begin{definition}\label{Defn: Section Background Alg: Commutative Algebra Pseudofunctor}
The pseudofunctor $\CAlg{(-)}:\Cring^{\op} \to \fCat$ is defined as follows:
\begin{itemize}
	\item To every commutative ring $R$ we define
	\[
	\big[\CAlg{(-)}\big](R) := \CAlg{R} = R/\Cring.
	\]
	To every morphism $f:R \to S$ of commutative rings we define $\CAlg{f}$ to be the adjunction
	\[
	\begin{tikzcd}
	\CAlg{S} \ar[rr, swap, ""{name = R}, bend right = 20]{}{W_f} & & \CAlg{R} \ar[ll, swap, bend right = 20, ""{name = L}]{}{(-) \otimes_R S}
	\ar[from = L, to = R, symbol = \dashv]
	\end{tikzcd}
	\]
	where the left adjoint is induced by taking the pushout in $\Cring$ and where the right adjoint $W_f$ is given by sending a commutative $S$-algebra $\nu:S \to B$ to the commutative $R$-algebra 
	\[
	R \xrightarrow{f} S \xrightarrow{\nu} B.
	\]
	\item The compositor $\phi_{f,g}$ of a pair of ring maps $f:R \to S$ and $g:S \to T$ is induced by taking the pair $(\alpha_{f,g},\id)$ where $\alpha_{f,g}:((-) \otimes_R S) \otimes_S T \xRightarrow{\cong} (-) \otimes_R T$ is the tensor cancellation natural isomorphism; its mate pair is the identity because the right adjoints compose strictly, i.e., $W_f \circ W_g = W_{g \circ f}$.
\end{itemize}
\end{definition}
\begin{remark}
The right adjoints $W_{f}$ are named as such because they are the affine algebraic-geometric incarnation of Weil restriction along $f$.
\end{remark}

By applying the Grothendieck construction to $\CAlg{(-)}$ we obtain a bifibration $p:\operatorname{El}(\CAlg{(-)}) \to \Cring$ which encodes the total space of commutative algebras. Its opfibration structure $p^{\op}$ is simply the slice category opfibration on $\Cring^{\op}$. We define this fibration explicitly below.

\begin{definition}\label{Defn: Section Background Algebra: Bifibration of commutative algebras}
The fibration $p:\CCalg \to \Cring$ is defined as follows. The category $\CCalg$ is the category with:
\begin{itemize}
	\item Objects: Pairs $(R, A)$ where $R$ is a commutative ring and where $\nu:R \to A$ is a commutative $R$-algebra.
	\item Morphisms: A morphism $(R,A) \to (S,B)$ is given by the following rule:
	\begin{prooftree}
		\AxiomC{$(R,A) \to (S,B)$ in $\CCalg$}\doubleLine
		\UnaryInfC{$(f,\rho):(R,A) \to (S,B)$ in $\CCalg$}\doubleLine
		\UnaryInfC{$f \in \Cring(R,S)$ and $\rho \in \CAlg{R}(A,W_f(B))$}
	\end{prooftree}
	\item Composition: The composition of $(f,\rho):(R,A) \to (S,B)$ and $(g,\varphi):(S,B) \to (T,C)$ is given by:
	\begin{prooftree}
		\AxiomC{$(f,\rho):(R,A) \to (S,B)$ in $\CCalg$}
		\UnaryInfC{$f \in \Cring(R,S)$ and $\rho \in \CAlg{R}(A,W_f(B))$}
		\AxiomC{$(g,\varphi):(S,B) \to (T,C)$ in $\CCalg$}
		\UnaryInfC{$g \in \Cring(S,T)$ and $\varphi \in \CAlg{S}(B,W_g(C))$}
		\BinaryInfC{$g \circ f \in \Cring(R,T)$ and $W_f(\varphi) \circ \rho \in \CAlg{R}(A,W_{g \circ f}(C))$}
	\end{prooftree}
	\item Identities: The identity of $(R,A)$ is $(\id_R, \id_A)$.
\end{itemize}
The fibration functor $p:\CCalg \to \Cring$ sends a pair $(R,A)$ to $R$ and similarly for morphisms.
\end{definition}
\begin{remark}
 Given a map $f:R \to S$ in $\Cring$ and a $\CCalg$-object $(S,B)$ then the Cartesian lifts of $f$ take the form $(f,\id_{W_fB}):(R,W_fB) \to (S,B)$.
\end{remark}

We now give a quick sketch of the fact that the functors $W_f:\CAlg{S} \to \CAlg{R}$ are all monadic. This will be helpful for deducing that in particular $W_f$ is isomorphism-reflecting.

\begin{proposition}\label{Prop: Section Background Alg: Weil Restriction Monadic}
	Let $f:R \to S$ be a morphism of commutative rigs. Then the Weil restriction functor
	\[
	W_{f}:\CAlg{S} \to \CAlg{R}
	\]
	is monadic.
\end{proposition}
\begin{proof}[Sketch]
	Let $\Tbb$ be the monad associated to the adjunction
	\[
	\begin{tikzcd}
		\CAlg{S} \ar[rr, bend right = 20, swap, ""{name = R}]{}{W_{f}} & & \CAlg{R} \ar[ll, bend right = 20, swap, ""{name = L}]{}{(-) \otimes_R S}
		\ar[from = L, to = R, symbol = \dashv]
	\end{tikzcd}
	\]
	and recall that the underlying functor $T$ of $\Tbb$ is $T := W_{f} \circ ((-) \otimes_R S)$. Recall also that the statement that $W_{f}$ is monadic is equivalent to saying that the comparison functor
	\[
	C:\CAlg{S} \to \CAlg{R}^{\Tbb}
	\]
	given by sending an $S$-algebra $A$ to the corresponding $\Tbb$-algebra
	\[
	A \mapsto \begin{tikzcd}
		W_{f}\left(W_{f}(A) \otimes_R S\right) \ar[d]{}{(W_{f} \ast \epsilon)_A} \\
		W_{f}(A)
	\end{tikzcd}
	\]
	is an equivalence of categories. However, this is straightforward and routine to verify and hence is omitted for the sake of paper length. 
\end{proof}
\begin{corollary}\label{Cor: Section Background Alg: Weil restriction is monadic}
	For any map $f:R \to S$ of commutative rings, $W_{f}$ is isomorphism-reflecting.
\end{corollary}
\begin{proof}
	This follows from Beck's Monadicity Theorem (cf.\@  \cite[Theorem 4.4.4]{BorceuxHandBook2}) and the fact that $W_{f}$ is monadic.
\end{proof}

\subsection{The Pseudofunctor of Modules}\label{Subsection: Background Alg: The Pseudofunctor of Modules}
We now preform the same analysis as we did for commutative algebras over commutative rings to constructing the fibration of modules over commutative rings. The pseudofunctor we first define uses the extension/restriction of scalars functors in order to produce a pseudofunctor valued in adjunctions.

\begin{definition}\label{Defn: Background Algebra: Pseudofunctor of Modules}
Define the pseudofunctor $\Mod{(-)}:\Cring^{\op} \to \Adj$ as follows:
\begin{itemize}
	\item For each commutative ring $R$, define
	\[
	\big[\Mod{(-)}\big](R) := \Mod{R}.
	\]
	\item For each morphism $f:R \to S$ of commutative rings, let $\Mod{f}$ denote the adjunction
	\[
	\begin{tikzcd}
	\Mod{S} \ar[rr, bend right = 20, swap, ""{name = R}]{}{\res_f} & & \Mod{R} \ar[ll, swap, bend right = 20, ""{name = L}]{}{(-) \otimes_R S}
	\ar[from = L, to = R, symbol = \dashv]
	\end{tikzcd}
	\]
	where the right adjoint is restriction of scalars along $f$ and the left adjoint is extension of scalars along $f$.
	\item Compositors: Given maps $f:R \to S$ and $g:S \to T$ of commutative rings, the compositor pair takes the form $\phi_{f,g} = (\alpha_{f,g},\id)$ where $\alpha_{f,g}: ((-) \otimes_R S) \otimes_S T \xRightarrow{\cong} (-) \otimes_R T$ is once again the tensor cancellation natural isomorphism; its mate is the identity functor because $\res_{f} \circ \res_{g} = \res_{g \circ f}$.
\end{itemize}
\end{definition}
By applying the Grothendieck construction to the pseudofunctor $\Mod{(-)}$ we obtain a bifibration $p:\operatorname{El}(\Mod{(-)}) \to \Cring$ which encodes the total space of modules. Its opfibration structure $p^{\op}$ is induced by taking the opposite of the usual induced tensor product of modules universal property. We define this fibration $p$ explicitly below.

\begin{definition}\label{Defn: Section Background Algebra: Bifibration of modules}
	The fibration $p:\MMod \to \Cring$ is defined as follows. The category $\MMod$ is the category with:
	\begin{itemize}
		\item Objects: Pairs $(R, M)$ where $R$ is a commutative ring and where $M$ is an $R$-module
		\item Morphisms: A morphism $(R,M) \to (S,N)$ is given by the following rule:
		\begin{prooftree}
			\AxiomC{$(R,M) \to (S,N)$ in $\CCalg$}\doubleLine
			\UnaryInfC{$(f,\rho):(R,M) \to (S,N)$ in $\CCalg$}\doubleLine
			\UnaryInfC{$f \in \Cring(R,S)$ and $\rho \in \Mod{R}(A,\res_f(B))$}
		\end{prooftree}
		\item Composition: The composition of $(f,\rho):(R,M) \to (S,N)$ and $(g,\varphi):(S,N) \to (T,L)$ is given by:
		\begin{prooftree}
			\AxiomC{$(f,\rho):(R,M) \to (S,N)$ in $\MMod$}
			\UnaryInfC{$f \in \Cring(R,S)$ and $\rho \in \Mod{R}(M,\res_f(N))$}
			\AxiomC{$(g,\varphi):(S,N) \to (T,L)$ in $\MMod$}
			\UnaryInfC{$g \in \Cring(S,T)$ and $\varphi \in \MMod{S}(N,W_g(L))$}
			\BinaryInfC{$g \circ f \in \Cring(R,T)$ and $\res_f(\varphi) \circ \rho \in \CAlg{R}(M,\res_{g \circ f}(L))$}
		\end{prooftree}
		\item Identities: The identity of $(R,M)$ is $(\id_R, \id_M)$.
	\end{itemize}
	The fibration functor $p:\MMod \to \Cring$ sends a pair $(R,M)$ to $R$ and similarly for morphisms.
\end{definition}
\begin{remark}
	Given a map $f:R \to S$ in $\Cring$ and a $\MMod$-object $(S,M)$ then the Cartesian lifts of $f$ take the form $(f,\id_{\res_fB}):(R,\res_{f}M) \to (S,M)$.
\end{remark}

We close this subsection with sketch of the fact that for all morphisms $f:R \to S$ in $\Crig$, the restriction of scalars functors
\[
\res_{f}:\Mod{S} \to \Mod{R}
\]
are monadic.
\begin{proposition}\label{Prop: Section Background Alg: Res of Scalars Monadic}
Let $f:R \to S$ be a morphism of commutative rigs. Then the restriction of scalars functor
\[
\res_{f}:\Mod{S} \to \Mod{R}
\]
is monadic.
\end{proposition}
\begin{proof}[Sketch]
Let $\Tbb$ be the monad associated to the adjunction
\[
\begin{tikzcd}
\Mod{S} \ar[rr, bend right = 20, swap, ""{name = R}]{}{\res_{f}} & & \Mod{R} \ar[ll, bend right = 20, swap, ""{name = L}]{}{(-) \otimes_R S}
\ar[from = L, to = R, symbol = \dashv]
\end{tikzcd}
\]
and recall that the underlying functor $T$ of $\Tbb$ is $T := \res_{f} \circ ((-) \otimes_R S)$. Recall also that the statement that $\res_{f}$ is monadic is equivalent to saying that the functor
\[
C:\Mod{S} \to \Mod{R}^{\Tbb}
\]
given by sending a module $M$ to the corresponding $\Tbb$-algebra
\[
M \mapsto \begin{tikzcd}
	\res_{f}\left(\res_{f}(M) \otimes_R S\right) \ar[d]{}{(\res_{f} \ast \epsilon)_M} \\
	\res_{f}(M)
\end{tikzcd}
\]
is an equivalence of categories. However, as in Proposition \ref{Prop: Section Background Alg: Res of Scalars Monadic}, this is straightforward to check and so is omitted. 
\end{proof}
\begin{corollary}\label{Cor: Section Background Alg: Restriction of scalars is monadic}
For any map $f:R \to S$ of commutative rigs, $\res_{f}$ is isomorphism-reflecting.
\end{corollary}
\begin{proof}
This follows from Beck's Monadicity Theorem and the fact that $\res_{f}$ is monadic.
\end{proof}

\subsection{An Adjunction Between the Bifibrations of Commutative Algebtras and Modules}\label{Subsection: Background Alg: Adjunction between bifibration}
In this subsection we again give a description of how the symmetric algebra and underlying module functors may be phrased in terms of an adjunction between the bifibrations $\CCalg$ and $\MMod$. This will require briefly recalling both the symmetric algebra and underlying module functors.

The underlying module functor $\Und{(-)}:\CAlg{(-)} \to \Mod{(-)}$ may be written as the object component of a pseudonatural transformation (and hence can be seen to induce a morphism of fibrations by the Grothendieck construction). The functor
\[
\Und{R}:\CAlg{R} \to \Mod{R}
\]
is defined by sending an $R$-algebra $A$ to its underlying $R$-module by forgetting the multiplication $A$ carries and sending an $R$-algebra map to the same map regarded as a morphism of $R$-modules. For any map of rings $f:R \to S$, the underlying module functors enjoy the following property. 
\begin{proposition}\label{Prop: Section Background in Algebra: Underlying commutes strictly with restriction}
Let $f:R \to S$ be a morphism of crigs. Then the diagram
\[
\begin{tikzcd}
\CAlg{S} \ar[r]{}{\Und{S}} \ar[d, swap]{}{W_{f}} & \Mod{S} \ar[d]{}{\res_{f}} \\
\CAlg{R} \ar[r, swap]{}{\Und{R}} & \Mod{R}
\end{tikzcd}
\]
commutes strictly. In particular, the transformation
\[
\Und{(-)}:\CAlg{(-)} \Rightarrow \Mod{(-)}:\Cring^{\op} \to \fCat
\] 
is a strict natural transformation.
\end{proposition}
\begin{proof}
This is immediate from the definition of each functor in sight.
\end{proof}

We now get to know the symmetric algebra functors $\Sym{(-)}:\Mod{(-)} \to \CAlg{(-)}$ as they vary over the category $\Cring$. Given a module $M$, we can define $\Sym{R}(M)$ as the free $R$-algebra on the set $M$ modulo the addition laws on $M$ and action law:
\[
\Sym{R}(M) := \frac{R[x_m:m \in M]}{(x_{rm} - rx_m, x_m + x_{n} - x_{m+n}: m, n \in M; r \in R)}.
\]
The symmetric algebra functors can then be checked to be left adjoint to the underlying module functor by simply using the universal properties that free algebras $R[x_i:i \in I]$ carry. By observing this together with the fact that $\Und{(-)}$ is a strict natural transformation, we can check that $\Sym{(-)}$ varies op-pseudonaturally in $\Cring^{\op}$ and hence that $\Sym{(-)}$ forms a morphism of opfibrations $\MMod^{\op} \to \CCalg^{\op}$. We give a short proof of the isomorphism between $\Sym{(-)}$ and the various tensor products below. We will not, however, prove that $\Sym{(-)}$ is a pseudonatural transformation here as it is straightforward but tedious; instead we will simply observe it and proceed.

\begin{corollary}\label{Cor: Section Background Alg: Sym and Tensors Commute}
Let $f:R \to S$ be a morphism of commutative rigs. Then there is a natural isomorphism of functors
\[
\Sym{S}(-) \circ \big((-)\otimes_{R}S\big) \cong \big((-) \otimes_{R} S\big) \circ \Sym{R}(-):\Mod{R} \to \CAlg{S}.
\]
\end{corollary}
\begin{proof}
Since left adjoints compose we have that \[
\Sym{S} \circ \big((-) \otimes_{R} S\big) \dashv \res_{f} \circ \Und{S}
\]
and that 
\[
\big((-) \otimes_R S\big) \circ \Sym{R} \dashv \Und{R} \circ W_{f}.
\]
Thus, since $\Und{R} \circ W_{f} = \res_{f} \circ \Und{S}$ by Proposition \ref{Prop: Section Background in Algebra: Underlying commutes strictly with restriction}, both composite left adjoints are left adjoint to the same functor. As such they are naturally isomorphic.
\end{proof}
\begin{corollary}\label{Cor: Section Background Alg: Pseudonat for Sym}
The functors $\Sym{R}:\Mod{R} \to \CAlg{R}$ for all commutative rings $R$ together with the witness natural isomorphisms
\[
\begin{tikzcd}
\Mod{R} \ar[rr, ""{name = U}]{}{\Sym{R}} \ar[d, swap]{}{(-) \otimes_R S} & & \CAlg{R} \ar[d]{}{(-) \otimes_R S} \\
\Mod{S} \ar[rr, swap, ""{name = D}]{}{\Sym{S}} & & \CAlg{S}
\ar[from = U, to = D, Rightarrow, shorten <= 4pt, shorten >= 4pt]{}{\Sym{f}}
\ar[from = U, to = D, Rightarrow, swap, shorten <= 4pt, shorten >= 4pt]{}{\cong}
\end{tikzcd}
\]
form a pseudonatural transformation:
\[
\begin{tikzcd}
	\Cring \ar[rrr, bend left = 30, ""{name = U}]{}{\pi_L \circ \Mod{(-)}^{\circ}} \ar[rrr, bend right = 30, swap, ""{name = D}]{}{\pi_L \circ \CAlg{(-)}^{\circ}} & & & \fCat
	\ar[from = U, to = D, Rightarrow, shorten <= 4pt, shorten >= 4pt]{}{\Sym{(-)}}
\end{tikzcd}
\]
\end{corollary}
\begin{proof}[Sketch]
	This is a striaghtforward but ultimately tedious check which ultimately comes down to the fact that $\Cring$ is regarded as a locally discrete $2$-category and every composite functor in sight is left adjoint to the same functor (on the nose).
\end{proof}

The most important aspects (for our purposes) of the symmetric algebra and underlying module functors are that they assemble to an adjunction between the fibrations $\CCalg$ and $\MMod$. This allows us to deduce the way in which the $\Sym{(-)}$ functors may be extended to the setting of quasi-coherent sheaves.
\begin{proposition}\label{Prop: Section Background Alg: Sym and Und adjunction between CAlg and Mod fibrations over crig}
There is an adjunction between fibrations
\[
\begin{tikzcd}
	\CCalg \ar[rr, bend right = 20, swap, ""{name = R}]{}{\Ubb\!\operatorname{nd}} \ar[ddr, swap]{}{p_{\CAlg{(-)}}} & & \MMod \ar[ddl]{}{p_{\Mod{(-)}}} \ar[ll, bend right = 20, swap, ""{name = L}]{}{\Sbb\!\operatorname{ym}} \\
	\\
	& \Cring
	\ar[from = L, to = R, symbol = \dashv]
\end{tikzcd}
\]
where
\[
\Ubb\!\operatorname{nd}(R,A) := \left(R,\Und{R}(A)\right)
\]
and
\[
\Sbb\!\operatorname{ym}(S,M) := \left(S,\Sym{S}(M)\right)
\]
for all objects $(R,A)$ in $\CCalg$ and for all objects $(S,M)$ in $\MMod$.
\end{proposition}
\begin{proof}
If we have a morphism $(f,\varphi):(R,M) \to (S,\Und{S}(B))$ in $\MMod$ for $(R,M)$ an object of $\CCalg$ and for $(S,\Und{S}(B))$ an object of $\MMod$, then the following chain of deductions
\begin{prooftree}
	\AxiomC{$(f,\varphi):(R,M) \to (S,\Und{S}(B))$ in $\MMod$}\doubleLine
	\UnaryInfC{$f:R \to S$ in $\Crig$, $\varphi:M \to \res_f(\Und{S}(B))$ in $\Mod{R}$}\doubleLine
	\UnaryInfC{$f:R \to S$ in $\Crig$, $\varphi:M \to \Und{R}(W_f(B))$ in $\Mod{R}$}
	\UnaryInfC{$f:R \to S$ in $\Crig$, $\varphi^{\sharp}:\Sym{R}(M) \to W_f(B)$ in $\CAlg{R}$}
	\UnaryInfC{$(f,\varphi^{\sharp}):(R,\Sym{R}(M)) \to (S,B)$ in $\CCalg$}
\end{prooftree}
gives rise to a combinator $(f,\varphi) \mapsto (f,\varphi^{\sharp})$ where $\varphi^{\sharp}$ is the adjoint transpose of $\varphi$. Similarly, given a morphism $(f,\psi):(R,\Sym{R}(M)) \to (S,B)$ there is a corresponding combinator
\begin{prooftree}
	\AxiomC{$(f,\psi):(R,\Sym{R}(M)) \to (S,B)$ in $\CCalg$}\doubleLine
	\UnaryInfC{$f:R \to S$ in $\Crig$, $\psi:\Sym{R}(M) \to W_f(B)$ in $\CAlg{R}$}
	\UnaryInfC{$f:R \to S$ in $\Crig$, $\psi^{\flat}:M \to \Und{R}(W_f(B))$ in $\Mod{S}$}\doubleLine
	\UnaryInfC{$f:R \to S$ in $\Crig$, $\psi^{\flat}:M \to \res_f(\Und{S}(B))$ in $\Mod{S}$}\doubleLine
	\UnaryInfC{$(f,\psi^{\flat}):(R,M) \to (S,\Und{S}(B))$ in $\MMod$}
\end{prooftree}
where $\psi^{\flat}$ is the adjoint transpose of $\psi$. A routine calculation shows that these combinators are mutually inverse to each other and so determine that $\Sbb\!\operatorname{ym} \dashv \Ubb\!\operatorname{nd}$.
\end{proof}

\subsection{The Coalgebra Structure of Symmetric Algebras}\label{Subsection: Coalg of Sym}
An important technical result for constructing the tangent structure on $\mathbf{AffSch}_{/R}$ (and hence, later on, for the tangent structure on $\Sch_{/S}$) regarding the symmetric algebra functor is that it also functorially produces cocommutative Hopf algebras in the symmetric monoidal category $(\CAlg{R},\otimes_R, R)$. When extending these results to the rig-theoretic case, one is obliged to only work with bialgebras, as we cannot necessarily construct the antipode if we do not have negatives in our rig of scalars.

To see how to construct the comultiplication $\nabla$ on $\Sym{R}(M)$, assume that $M$ is an $R$-module.To define the comultiplication $\nabla$ on $\Sym{R}(M)$ in $\CAlg{R}$ is equivalent to defining the map $\nabla^{\sharp}$ in $\Mod{R}$ as below:
\begin{prooftree}
	\AxiomC{$\nabla:\Sym{R}(M) \to \Sym{R}(M) \otimes_R \Sym{R}(M)$}
	\UnaryInfC{$\nabla^{\sharp}:M \to \Und{R}\left(\Sym{R}(M) \otimes_R \Sym{R}(M)\right)$}\doubleLine
	\UnaryInfC{$\nabla^{\sharp}:M \to \Sym{R}(M) \otimes_R \Sym{R}(M)$}
\end{prooftree}
Note in the last map, we regard $\Sym{R}(M) \otimes_R \Sym{R}(M)$ only with respect to its $R$-module structure. We consequently define the module-theoretic map $\nabla^{\sharp}:M \to \Sym{R}(M) \otimes_R \Sym{R}(M)$ by the equation\footnote{From here on in the paper, for any $m \in M$ we write $m \in \Sym{R}(M)$ as a shorthand/abuse of notation for the residue class of the free variable $x_m \in R[x_m:m \in M]/(x_{rm}-rx_m, x_m+x_n-x_{m+n}:r \in R; m,n \in M)$.}
\begin{equation*}\label{Eqn: Section Background Alg: Comultiplication of Sym}
m \mapsto m \otimes 1 + 1 \otimes m.
\end{equation*}
The corresponding map on $\Sym{R}(M)$ is given by $x \mapsto x \otimes 1 + 1 \otimes x$.

To define the counit $\epsilon_M:\Sym{R}(M) \to R$ of the Hopf algebra structure on $\Sym{R}(M)$, we note that since the functor $\Sym{R}$ is a left adjoint, it preserves initial objects. Consequently, $\Sym{R}(0) \cong R$. Writing $0_R:\Sym{R}(0) \xrightarrow{\cong} R$ for this ismorphism and $\bang_M:M \to 0$ for the zero map in $\Mod{R}$, we then obtain a map $\epsilon_M:\Sym{R}(M) \to R$ defined by the composition
\begin{equation*}\label{Eqn: Section Background Alg: The counit map}
\Sym{R}(M) \xrightarrow{\Sym{R}(\bang_M)} \Sym{R}(0) \xrightarrow{0_R} R.
\end{equation*}
The antipode $S:\Sym{R}(M) \to \Sym{R}(M)$ simply sends $x_m$ to $-x_{m} = x_{-m}$ for $m \in M$ and acts on $R$ via the identity.
\begin{proposition}\label{Prop: Section Background Alg: Symmetric Alg is Coalg}
For any commutative ring $R$, the symmetric algebra functor $\Sym{R}:\Mod{R} \to \CAlg{R}$ extends to a functor 
\[
\widehat{\Sym{R}}:\Mod{R} \to \mathbf{Hopf}^{\operatorname{cocom}}(R).
\]
\end{proposition}
\begin{proof}[Sketch]
Because we already know that $\Sym{R}(-)$ is functorial, it suffices to prove that $(\Sym{R}(M),\nabla_M,\epsilon_M, S)$ is a cocommutative Hopf algebra in $\CAlg{R}$ for all $R$-modules $M$ and that each map $\Sym{R}(f)$ is also a morphism of Hopf algebras. However, this is a straightforward check which follows immediately from expanding out the definitions and following one's nose. What is missing is checking that the antipode is in fact a (co)inversion for the Hopf algebra. However, this is trivial to verify as 
\[
\mu((\id \otimes S)(\nabla(x))) = \mu(1 \otimes (-x) + x \otimes 1) = -x + x = 0
\] 
for all $x \in M$. 
\end{proof}
\begin{remark}
If one is interested in developing the rig-theoretic extension of Proposition \ref{Prop: Section Background Alg: Symmetric Alg is Coalg}, one can only prove that $\Sym{R}(-)$ extends to is a functor
\[
\widehat{\Sym{R}}:\Mod{R} \to \mathbf{Bialg}(\Mod{R},\otimes_R,R)
\]
whenever $R$ is a commutative rig. The issue lies in the fact that the definition of the antipode $S:\Sym{R}(M) \to \Sym{R}(M), x_m \mapsto x_{-m}$ only makes sense if each $m \in M$ has an additive inverse.
\end{remark}

We now argue that the Hopf algebra structures on $\Sym{R}(M)$ are compatible with ring maps in the sense that if we have a morphism $f:A \to B$ of commutative rings with identity then the tensor product functor
\[
(-) \otimes_A B: \CAlg{A} \to \CAlg{B}
\]
induces a strong comonoidal functor
\[
\mathbf{Hopf}^{\operatorname{cocom}}(\CAlg{A}) \to \mathbf{Hopf}^{\operatorname{cocom}}(\CAlg{B}).
\]
However, this is routine, as the tensor product is the pushout and so is cocontinuous (and hence preserves cogroup objects in $\CAlg{A}$). Together with the isomorphism $\Sym{B}((-) \otimes_A B) \cong \Sym{A}(-) \otimes_A B$, this allows us to deduce that the functor $(-) \otimes_A B:\CAlg{A} \to \CAlg{B}$ induces a comonoidal natural isomorphism
\[
\begin{tikzcd}
\Mod{A} \ar[rr, ""{name= U}]{}{(-) \otimes_A B} \ar[d, swap]{}{\widehat{\Sym{A}}} & & \Mod{B} \ar[d]{}{\widehat{\Sym{B}}} \\
\mathbf{Hopf}^{\operatorname{cocom}}(A) \ar[rr, swap, ""{name = D}]{}{(-) \otimes_A B} & & \mathbf{Hopf}^{\operatorname{cocom}}(B)
\ar[from = U, to = D, Rightarrow, shorten <= 4pt, shorten >= 4pt]{}{\cong}
\ar[from = U, to = D, Rightarrow, shorten <= 4pt, shorten >= 4pt, swap]{}{\widehat{\Sym{f}}}
\end{tikzcd}
\]
By the universal properties of the tensor products on which these natural isomorphisms are defined, it follows that they vary pseudonaturally in $\Cring$. In particular, we record this as proposition below for future use.

\begin{proposition}\label{Prop: Section Background Alg: Sym Hopf Structure Pseudonatural}
There is a pseudonatural transformation
\[
\begin{tikzcd}
\Cring \ar[rr, bend left = 30, ""{name = U}]{}{\Mod{(-)}} \ar[rr, bend right = 30, swap, ""{name = D}]{}{\mathbf{Hopf}^{\operatorname{cocom}}(-)} & & \fCat
\ar[from = U, to = D, Rightarrow, shorten <= 4pt, shorten >= 4pt]{}{\widehat{\sym}}
\end{tikzcd}
\]
where the object functors are given by $\widehat{\Sym{A}}$ and the commutativity witness natural isomorphisms are given by the natural isomorphisms $\widehat{\Sym{f}}$.
\end{proposition}


Less surprising than the case of tensor products, the Hopfification functors $\widehat{\Sym{(-)}}$ are particularly well-suited to the restriction of scalars functors $W_f:\Mod{S} \to \Mod{R}$ for any ring map $f:R \to S$. That is, we argue that the diagram
\[
\begin{tikzcd}
\Mod{S} \ar[d, swap]{}{\widehat{\Sym{S}}} \ar[rr]{}{W_f} & & \Mod{R} \ar[d]{}{\widehat{\Sym{R}}} \\
\mathbf{Hopf}^{\operatorname{cocom}}(S) \ar[rr, swap]{}{W_f^{\operatorname{cocom}}} & & \mathbf{Hopf}^{\operatorname{cocom}}(R)
\end{tikzcd}
\]
commutes. This amounts to showing that for any $S$-module $M$ that the diagrams
\[
\begin{tikzcd}
\Sym{R}(W_f(M)) \ar[rr]{}{F} \ar[d, swap]{}{} & & \Sym{S}(M)\ar[d]{}{} \\ 
\Sym{R}(W_f(M)) \otimes_R \Sym{R}(W_f(M)) \ar[rr, swap]{}{F \otimes F} & & \Sym{S}(M) \otimes_S \Sym{S}(M)
\end{tikzcd}
\]
\[
\begin{tikzcd}
\Sym{R}(W_fM) \ar[r]{}{\epsilon_{W_fM}} \ar[d, swap]{}{F} & R \ar[d]{}{f} \\
\Sym{S}(M) \ar[r, swap]{}{\epsilon_M} & S
\end{tikzcd}
\]
both commute; note that the map $F:\Sym{R}(W_f(M)) \to \Sym{S}(M)$ is induced by the ring maps 
\[
\frac{R[x_{m} : m \in M]}{(rx_m-x_{rm}, x_m+x_{n} - x_{m+n}:m,n \in M; r \in R)} \to \frac{S[x_{m}:m \in M]}{(sx_m - x_{sm}, x_mx_n - x_{m+n}:m,n \in M; s \in S)}
\]
 via $rx_m \mapsto f(r)x_m$ for all $m \in M$ and for all $r \in R$. The commutativity of the first diagram follows from the fact that for all $rx_m$,
\begin{align*}
\left(F\otimes F\right)\left(\nabla_{W_fM}(rx_m)\right) &= (F \otimes F)\left(r(x_m \otimes 1)\right) + (F \otimes F)\left(r(1 \otimes x_m)\right) \\
&= f(r)(x_m \otimes 1) + f(r)(1 \otimes x_m) = f(r)x_m \otimes 1 + 1 \otimes f(r)x_m \\
&= \nabla_{M}(F(rx_m))
\end{align*}
while the commutativity of the second diagram follows from the computations that 
\[
\epsilon_M(F(r)) = \epsilon_M(f(r)) = f(r) = f(\epsilon_{W_fM}(r))
\]
for all $r \in R$ while for all $m \in M$,
\[
\epsilon_M(F(rx_m)) = \epsilon_M(f(r)x_m) = f(r)\cdot 0 = 0 = f(0) = f\left(\epsilon_{W_fM}(x_m)\right).
\]
This allows us to deduce the following proposition which indicates that the symmetric Hopf algebra functors $\widehat{\Sym{(-)}}$ give a strict natural transformation between the module pseudofunctor and the (cocommutative) Hopf algebra pseudofunctor on $\Cring^{\op}$. It will also allow us to deduce later that the morphisms $\nabla$ may be extended to quasi-coherent sheaves.
\begin{proposition}\label{Prop: Section Background Alg: Strictness of Hopf for Restriction of Scalars}
There is a strict pseudonatural transformation:
\[
\begin{tikzcd}
\Cring^{\op} \ar[rrr, bend left = 30, ""{name = U}]{}{\Mod{(-)}} \ar[rrr, bend right = 30, swap, ""{name = D}]{}{\mathbf{Hopf}^{\operatorname{cocom}}(-)} & & & \fCat
\ar[from = U, to = D, Rightarrow, shorten <= 4pt, shorten >= 4pt]{}{\widehat{\Sym{(-)}}}
\end{tikzcd}
\]
In particular, for any map of commutative rings $f:R \to S$ and for any $S$-module $M$, the induced ring map
\[
F:\Sym{R}(W_f(M)) \to \Sym{S}(M)
\]
is a morphism of Hopf algebras.
\end{proposition}

\section{A Review of Quasi-Coherent Sheaves and Related Structures: Schemifying and Sheafifying the Catgorical Algebra Background}\label{Section: QCoh}

Quasi-coherent sheaves occupy a central place in the study of scheme-theoretic algebraic geometry because they are to schemes what modules are to rings. Given an affine scheme $\Spec A = (\lvert \Spec A\rvert, \Ocal_A)$, the notion of an $\Ocal_A$-module is too ``fuzzy'' to be of help if we want to say that $\Ocal_A$-modules come from the sheafification of $A$-modules; in particular, we have ``too many'' $\Ocal_{A}$-modules because $\Ocal_{A}$-modules need not even locally be quotients of free $\Ocal_{A}$-modules. The way to pare this down is to use quasi-coherent sheaves: these are built with the property that there are pseudofunctorial equivalences $\QCoh(\Spec A) \simeq \Mod{A}$. Consequently, when working with $\QCoh(\Spec A)$ we are restricting ourselves to studying the sheaves of $\Ocal_A$-modules which are fundamentally ring-theoretic in nature. In this section we review the basics of quasi-coherent sheaves, study the coalgebra algebra structures that the relative spectrum functor carries, and then show how one can encode so-called affine morphisms of schemes in terms of what are essentially relative spectra of quasi-coherent sheaves of $\Ocal_X$-algebras.

In the first subsection below, we pay particular attention to the $2$-categorical ways we can phrase the fact that the categories of quasi-coherent sheaves are gluings of categories of modules. This both gives a categorical explanation of why it is that algebraic geometers care about quasi-coherent sheaves as much as they do (and in what sense they really are the scheme-theoretic analogue of modules), as well as indicates category-theoretically precisely what it is that the phrase ``affine-local gluing'' really means (namely taking specific pseudolimits).

\subsection{A Short Primer on the Fibration of Quasi-Coherent Sheaves}\label{Subsection: QCoh Modules}

The definition and construction of quasi-coherent sheaves was given originally in \cite[Section 5.1.3]{EGA01}. A sheaf of modules $\Fscr$ on a locally ringed space $(\lvert X \rvert, \Ocal_X)$ is quasi-coherent when it is locally a cokernel of maps of free sheaves of modules $\Ocal_X^{\oplus I}$. This definition is not well-behaved for general locally ringed spaces, but \emph{is} relatively well-behaved\footnote{The failure of this to be well-behaved lies in the fact that if $f:X \to Y$ is a map of schemes then $f_{\ast}$ need not send quasi-coherent $X$-sheaves to quasi-coherent $Y$-sheaves without some finiteness conditions on $f$; cf.\@ Proposition \ref{Prop: Section Background Scheme: Qcqs maps have quasicompact pushforward}.} for schemes.The main reason this works for schemes  is that if a scheme $X$ has an open affine cover $\lbrace f_i:\Spec A_i \to X \; | \; i \in I \rbrace$, then $\QCoh(X)$ arises as the pseudolimit of the categories of modules $\Mod{A_i}$ (cf.\@ Theorem \ref{Thm: Section Background Scheme: QCoh is a pseudolimit}). It is this which allows us to extend the symmetric algebra and underlying module functors to the scheme-theoretic setting.

\begin{definition}[{\cite[Section 5.1.3]{EGA01}}]
Let $X = (\lvert X \rvert, \Ocal_X)$ be a locally ringed space. A sheaf $\Fscr$ of $\Ocal_X$-modules is \emph{quasi-coherent} if for all $x \in \lvert X \rvert$ there exists an open $U \subseteq \lvert X \rvert$ with $x \in U$ and immersion $j:U \to X$ for which there is a cokernel diagram
\[
\begin{tikzcd}
\bigoplus\limits_{i \in I} \Ocal_{U} \ar[r] & \bigoplus\limits_{k \in J} \Ocal_{U} \ar[r] & j^{\ast}\Fscr \ar[r] & 0
\end{tikzcd}
\]
for arbitrary index sets $I, J$ in the category of $\Ocal_{U}$-modules.
\end{definition}
An important but key structural result is that when $X$ is a scheme, the category $\QCoh(X)$ is Abelian.
\begin{proposition}[{\cite[Proposition II.5.7]{Hartshorne}}]\label{Prop: Section QCoh: Qcoh Abelian}
If $X$ is a scheme then $\QCoh(X)$ is an Abelian category.
\end{proposition}

The most important foundational result regarding quasi-coherent sheaves on schemes is, of course, that $\QCoh(\Spec A) \simeq \Mod{A}$ (cf.\@ \cite[Th{\'e}or{\`e}me 1.4.1]{EGA1} or \cite[Corollary II.5.5]{Hartshorne}). This is done by using the same sheafification functor that defines the structure sheaf on $\Spec A$. That is the functor
\begin{equation}\label{Eqn: Tilde functor for ring modules}
\widetilde{(-)}:\Mod{A} \to \QCoh(\Spec A)
\end{equation}
is defined by defining $\widetilde{M}$ to be the sheaf generated by the assignments
\[
\widetilde{M}(D(f)) := M_f \cong M \otimes_A A[f^{-1}]
\]
on the basic opens\footnote{By a standard result in site theory (cf.\@ \cite[Proposition III.4.1]{MacLaneMoerdijk} for instance) if $\tau$ is a pretopolgoy generating a Grothendieck topology $J$, a presheaf $\Fscr$ is a $J$-sheaf if and only if it satisfies the sheaf axiom for all covers in $\tau$; in particular, the functor from $J$-sheaves on $\Cscr$ to $\tau$-sheaves is an equivalence of categories. This yields the standard topological argument that it suffices to define a sheaf on a basis of opens.} $D(f) = \lbrace \pfrak \in \lvert \Spec A \rvert \; : \; f \notin \pfrak \rbrace$, $f \in A$. The global sections functor
\[
\Gamma(\widetilde{M}) = \widetilde{M}(\lvert \Spec A \rvert)
\]
gives the quasi-inverse of the sheafification functor.

\begin{Theorem}[{\cite[Th{\'e}or{\`e}me 1.4.1]{EGA1}, \cite[Corollary II.5.5]{Hartshorne}}]\label{Thm: Section Background Scheme: QCoh on affine is just modules}
Let $A$ be a commutative ring. Then the functors
\[
\widetilde{(-)}:\Mod{A} \to \QCoh(\Spec A)
\]
and
\[
\Gamma(-):\QCoh(\Spec A) \to \Mod{A}
\]
are inverse equivalences of categories.
\end{Theorem}

The fact that $\widetilde{(-)}:\Mod{A} \to \Mod{\Ocal_A}$ does not give an equivalence of categories (and hence the necessity of adding the adjective ``quasi-coherent'' to ``sheaf of $\Ocal_A$-modules'' everywhere in sight) can be seen below. The example below follows structure well-known to algebraic geometers, but presented in such a way so that the reader who is unfamiliar with quasi-coherent sheaves can see exactly what may fail.
\begin{example}\label{Example: EZPZ QCoh}
Consider the ring $\Z_p$ of $p$-adic integers for a positive integer prime $p$. Then $\lvert \Spec \Z_p \rvert = \lbrace (0), (p)  \rbrace$ with the Sierpinski topology: its open subsets are 
\[
\lbrace \emptyset, \lbrace (0) \rbrace, \lvert \Spec \Z_p \rvert \rbrace.
\]
As such, $\lvert \Spec \Z_p \rvert$ has \emph{no} proper covers, i.e., any cover $U$ of $\lvert \Spec \Z_p \rvert$ contains $\lvert \Spec \Z_p \rvert$. Thus we see that $\Fscr$ is a quasi-coherent sheaf on $\Spec \Z_p$ if and only if in the diagram
\[
\begin{tikzcd}
	\Fscr(\Spec \Z_p) \ar[dr]{}{} \ar[dd] \\
	 & \Fscr((0)) \ar[dl] \\
	 0
\end{tikzcd}
\]
of Abelian groups we have an isomorphism $\Fscr((0)) \cong \Fscr(\Spec \Z_p) \otimes_{\Z_p} \Qbb_p$. For an explicit non-example of a quasi-coherent sheaf of $\Ocal_{\Z_p}$-modules, the sheaf $\Fscr$ defined by
\[
\begin{tikzcd}
\Z_p \ar[dr] \ar[dd] & \\
 & \Qbb_p^4 \ar[dl] \\
 0
\end{tikzcd}
\]
with map $\Z_p \to \Qbb_p^4$ given by $1 \mapsto (1,1,1,1)$ is not quasi-coherent over $\Spec \Z_p$.
\end{example}

In general for any morphism $f:X \to Y$ of schemes, the functor $f^{\ast}:\Mod{\Ocal_Y} \to \Mod{\Ocal_X}$ is given by 
\[
\Fscr \mapsto f^{\ast}\Fscr := f^{-1}\Fscr \otimes_{f^{-1}\Ocal_Y} \Ocal_X.
\]
This functor is cocontinuous and so it follows that $f^{\ast}$ preserves quasi-coherent sheaves\footnote{To argue this find an open cover $\lbrace j_i:U_i \to Y \; | \; i \in I \rbrace$ for which each $j_i^{\ast}\Fscr$ is a cokernel of free sheaves in $\Mod{\Ocal_{U_i}}$. Now pull this cover back along $f:X \to Y$ in order to obtain a cover $U_i \times_Y X =: f^{-1}U_i$ of $X$. By construction the pullback projections $\pr_{2,i}:f^{-1}U_i \to X$ are open immersions. By inspection, $\pr_{1,i}^{\ast}\Fscr$ is a cokernel in $\Mod{\Ocal_{f^{-1}U_i}}$ of free sheaves and from the natural isomorphisms $\pr_{2,i}^{\ast} \circ f^{\ast} \cong (f \circ \pr_{2,i})^{\ast} = (j_i \circ \pr_{1,i})^{\ast} = \pr_{1,i}^{\ast} \circ j_i^{\ast}$ we find that $\pr_{2,i}^{\ast}(f^{\ast}\Fscr)$ is indeed a cokernel of free sheaves in each $\Mod{\Ocal_{f^{-1}U_i}}$.}  and hence restricts to a functor fitting into a commuting diagram
\[
\begin{tikzcd}
\QCoh(Y) \ar[r]{}{f^{\ast}} \ar[d, swap]{}{\operatorname{incl}_Y} & \QCoh(X) \ar[d]{}{\operatorname{incl}_X} \\
\Mod{\Ocal_Y} \ar[r, swap]{}{f^{\ast}} & \Mod{\Ocal_X}
\end{tikzcd}
\]
of categories. This means in particular that we have a pullback functor $f^{\ast}:\QCoh(Y) \to \QCoh(X)$ between quasi-coherent sheaves. Additionally, if $j:U \to X$ is an open immersion of schemes (that is if $\lvert j \rvert$ homeomorphically embeds $\lvert U\rvert$ onto an open subspace of $\lvert X \rvert$ and if $j^{\sharp}:j^{-1}\Ocal_X \to \Ocal_U$ is an isomorphism of $U$-sheaves) then $j^{\ast}$ is both left exact and cocontinuous.

Let us now show that the category $\QCoh(X)$ of quasi-coherent sheaves arises as a pseudolimit of its open subschemes. We go through this because the result is nice and not as well known as it should be, in the author's opinion, and also because to state and prove the theorem is precisely to show that a quasi-coherent sheaf is completely determined by its Zariski descent. At first we will show how, if $\lbrace U_i \to X \; | \; i \in I \rbrace$ is an open affine cover of $X$, then $\QCoh(X)$ arises as the pseudolimit of the $\QCoh(U_i)$ as glued along their intersections $\QCoh(U_i \times_X U_j)$. After this we will give the local characterization of quasi-coherent sheaves as gluings of quasi-coherent sheaves on affine open subschemes. Finally, we will also show that $\QCoh(X)$ is the pseudolimit of \emph{all} the categories $\QCoh(U)$ for affine open subschemes $U$ and use this to conclude that $\QCoh(X)$ is a locally presentable category. In particular, we will use that $\QCoh(X)$ is locally presentable in order to prove the existence of a right adjoint of the functor $f^{\ast}:\QCoh(Y) \to \QCoh(X)$ when given a scheme morphism $f:X \to Y$.

\begin{Theorem}[Folklore]\label{Thm: Section Background Scheme: QCoh is a pseudolimit}
	Let $X$ be a scheme and let $C = \lbrace \gamma_i:U_i \to X \; | \; i \in I \rbrace$ be an affine open cover of $X$. Then $\QCoh(X)$ is the pseudolimit in the $2$-category of categories of the cospans
	\[
	\QCoh(U_i) \xrightarrow{(\pi_0^{ij})^{\ast}} \QCoh(U_i \times_X U_j) \xleftarrow{(\pi_1^{ij})^{\ast}} \QCoh(U_j)
	\]
	for all $i, j \in I$.
\end{Theorem}
\begin{proof}[Sketch (category-theortic proof)]
	Whenever we have any two open affine subschemes $U_i, U_j$ of $X$ in our cover, define
	\[
	\begin{tikzcd}
		U_i \times_X U_j \ar[r]{}{\pi_1^{ij}} \ar[d, swap]{}{\pi_0^{ij}} & U_j \ar[d]{}{\gamma_j} \\
		U_i \ar[r, swap]{}{\gamma_i} & X
	\end{tikzcd}
	\]
	for the corresponding pullback maps for all $i,j \in I$. Additionally, recall that since $\QCoh(-):\Sch^{\op} \to \fCat$ is a pseudofunctor, for all $i, j \in I$ we have an invertible $2$-cell:
	\[
	\begin{tikzcd}
		\QCoh(X) \ar[rrrr, bend right = 20, swap, ""{name = D}]{}{(\pi_1^{ij})^{\ast} \circ \gamma_j^{\ast}} \ar[rrrr, bend left = 20, ""{name = U}]{}{(\pi_0^{ij})^{\ast} \circ \gamma_i^{\ast}} & & & & \QCoh(U_i \times_X U_j)
		\ar[from = U, to = D, Rightarrow, shorten <= 4pt, shorten >= 4pt]{}[description]{\phi_{\pi_1,\gamma_j}^{-1} \circ \phi_{\pi_0,\gamma_i}}
	\end{tikzcd}
	\]
	For the sake of brevity, write this compositor natural isomorphism as
	\[
	\phi_{ij} :=  \phi_{\pi_1,\gamma_j}^{-1} \circ \phi_{\pi_0,\gamma_i}.
	\]
	
	Now that we have initialized our data appropriately, to prove the Theorem we first must have a pseudocone above the $\QCoh(U_i)$. To this end, assume that we have a category $\Cscr$ together with functors $F_i:\Cscr \to \QCoh(U_i)$ such that whenever there is an inclusion map $\gamma_{ij}:U_i \to U_j$ there is an invertible $2$-cell
	\[
	\begin{tikzcd}
		\Cscr \ar[rr]{}{F_j} \ar[dr, swap]{}{F_j} & {} & \QCoh(U_i) \\
		& \QCoh(U_j) \ar[ur, swap]{}{\gamma_{ij}^{\ast}}
		\ar[from = 2-2, to = 1-2, Rightarrow, shorten <= 4pt, shorten >= 4pt]{}{\alpha_{ij}}
	\end{tikzcd}
	\] 
	and for all $i, j \in I$ we have invertible $2$-cells
	\[
	\begin{tikzcd}
		\Cscr \ar[rr, ""{name = U}]{}{F_i}	\ar[d, swap]{}{F_j} & & \QCoh(U_i)  \ar[d]{}{(\pi_{0}^{ij})^{\ast}} \\
		\QCoh(U_j) \ar[rr, swap, ""{name = D}]{}{(\pi_1^{ij})^{\ast}} & & \QCoh(U_i \times_X U_j)
		\ar[from = U, to = D, Rightarrow, shorten <= 4pt, shorten >= 4pt]{}{\pi_{ij}}
	\end{tikzcd}
	\]
	which satisfy the cocycle condition pasting diagrams. That is, for any open affines $U_i, U_j$ with an inclusion $\gamma_{ij}:U_i \to U_j$ the pasting diagram
	\[
	\begin{tikzcd}
		& \QCoh(U_j) \ar[dr, ""{name = UR}]{}{(\pi_{1}^{ij})^{\ast}} \ar[dd]{}{\gamma_{ij}^{\ast}} \\
		\Cscr \ar[ur, ""{name = UL}]{}{F_i} \ar[dr, swap, ""{name = DL}]{}{F_j}  & & \QCoh(U_i \times_{X} U_j) \\
		& \QCoh(U_i) \ar[ur, swap, ""{name = DR}]{}{(\pi_0^{ij})^{\ast}}
		\ar[from =UR, to = DR, swap, Rightarrow, shorten <= 4pt, shorten >= 4pt]{}{\phi_{ij}}
		\ar[from = UL, to = DL, Rightarrow, shorten <= 4pt, shorten >= 4pt]{}{\alpha_{ij}}
	\end{tikzcd}
	\]
	is equal to the $2$-cell:
	\[
	\begin{tikzcd}
		\Cscr \ar[rr, ""{name = U}]{}{F_i}	\ar[d, swap]{}{F_j} & & \QCoh(U_i)  \ar[d]{}{(\pi_{0}^{ij})^{\ast}} \\
		\QCoh(U_j) \ar[rr, swap, ""{name = D}]{}{(\pi_1^{ij})^{\ast}} & & \QCoh(U_i \times_X U_j)
		\ar[from = U, to = D, Rightarrow, shorten <= 4pt, shorten >= 4pt]{}{\pi_{ij}}
	\end{tikzcd}
	\]
	
	Our goal now is to define a functor $F:\Cscr \to \QCoh(X)$. Since the maps $\gamma_i:U_i \to X$ constitute an open affine cover of $X$, and since we have functors $F_i:\Cscr \to \QCoh(U_i)$ for all $i \in I$ for every object $A$ of $\Cscr$ we have a system of sheaves $F_iA$ in each category $\QCoh(U_i)$ with isomorphisms
	\[
	\pi_{ij}^{A}:\left(\pi_0^{ij}\right)^{\ast}\left(F_iA\right) \xrightarrow{\cong} \left(\pi_1^{ij}\right)^{\ast}\left(F_jA\right)
	\]
	which satisfy the cocycle condition for all $i, j \in I$. We can thus apply the Gluing Lemma and hence obtain a sheaf
	\[
	FA := \mathsf{Glue}\left(F_iA, \pi_{ij}^{A}:\left(\pi_0^{ij}\right)^{\ast}(F_iA) \xrightarrow{\cong} \left(\pi_1^{ij}\right)^{\ast}(F_jA)\right)
	\]
	on $X$. To see that this is a quasi-coherent sheaf on $X$ we note first that for all $i \in I$,
	\[
	(FA)|_{U_i} = \gamma_i^{\ast}(FA) = \gamma_i^{\ast}\mathsf{Glue}\left(F_iA, \pi_{ij}^{A}:\left(\pi_0^{ij}\right)^{\ast}(F_iA) \xrightarrow{\cong} \left(\pi_1^{ij}\right)^{\ast}(F_jA)\right) \cong F_iA
	\]
	by construction. Since $F_iA$ is an object of $\QCoh(U_i)$ and $U_i$ is affine for all $i \in I$, we can find an $\Ocal_{U_i}(\lvert U_i\rvert)$-module $M_i$ for which $F_iA \cong \widetilde{M}_i$ for all $i \in I$ by Theorem \ref{Thm: Section Background Scheme: QCoh on affine is just modules}. Thus we obtain isomorphisms
	\[
	(FA)|_{U_i} = \gamma_i^{\ast}FA \cong \widetilde{M}_i
	\]
	for all $i \in I$ and hence conclude that $FA$ is an object in $\QCoh(X)$. Additionally, this construction is natural in $A$ and hence allows us to lift maps $f:A \to B$ in $\Cscr$ to morphisms $Ff:FA \to FB$ in $\QCoh(X)$. Finally, the isomorphisms
	\[
	\rho_i^{A} := \gamma_i^{\ast}(FA) \xrightarrow{\cong} F_iA
	\]
	assemble to form invertible $2$-cells
	\[
	\begin{tikzcd}
		\Cscr \ar[rr]{}{F_i} \ar[dr, swap]{}{F} & {} & \QCoh(U_i) \\
		& \QCoh(X) \ar[ur, swap]{}{\gamma_i^{\ast}}
		\ar[from = 2-2, to = 1-2, Rightarrow, shorten <= 4pt, ]{}{\rho_i}
	\end{tikzcd}
	\]
	for all $i \in I$. By construction, since $\rho_i$ and $\rho_j$ invert local gluing and restrictions, we additionally have that the induced pasting diagram
	\[
	\begin{tikzcd}
		& & \QCoh(U_j) \ar[dr]{}{(\pi_1^{ij})^{\ast}} \\
		\Cscr \ar[drr, swap, bend right = 30, ""{name = DL}]{}{F_i} \ar[r, ""{name = M}]{}[description]{F} \ar[urr, bend left = 30, ""{name = UL}]{}{F_j} & \QCoh(X) \ar[ur]{}{\gamma_j^{\ast}} \ar[dr, swap]{}{\gamma_i^{\ast}} & & \QCoh(U_i \times_X U_j) \\
		& & \QCoh(U_i) \ar[ur, swap]{}{(\pi_0^{ij})^{\ast}}
		\ar[from = 1-3, to = 3-3, Rightarrow, shorten <=4pt, shorten >= 4pt]{}{\phi}
		\ar[from = UL, to = 2-2, Rightarrow, shorten <= 4pt, shorten >= 4pt]{}{\rho_j^{-1}}
		\ar[from = 2-2, to = DL, Rightarrow, shorten <= 4pt, shorten >= 4pt]{}{\rho_i}
	\end{tikzcd}
	\]
	is equal to the diagram
	\[
	\begin{tikzcd}
		\Cscr \ar[d, swap]{}{F_i} \ar[rr, ""{name = U}]{}{F_j} & & \QCoh(U_j) \ar[d]{}{(\pi_1^{ij})^{\ast}} \\
		\QCoh(U_i) \ar[rr, swap]{}{(\pi_0^{ij})^{\ast}} & & \QCoh(U_i \times_X U_j)
		\ar[from = U, to = D, Rightarrow, shorten <= 4pt, shorten = 4pt]{}{\pi_{ij}}
	\end{tikzcd}
	\]
	for all $i, j \in I$. Thus every pseudocone $F$ over the $\QCoh(U_i)$ which glue coherently over the intersections $U_i \times_X U_j$ factors through the pseudocone $\QCoh(X)$.
	
	We now need to show that the pseudocone $\QCoh(X)$ has the desired pseudouniversal property. That is, we must prove that for any functors $G, H:\Cscr \to \QCoh(X)$ for which there are natural transformations
	\[
	\begin{tikzcd}
		\Cscr \ar[rr, ""{name = U}]{}{G} \ar[d, swap]{}{H} & & \QCoh(X) \ar[d]{}{\gamma_i^{\ast}} \\
		\QCoh(X) \ar[rr, swap, ""{name = D}]{}{\gamma_i^{\ast}} & & \QCoh(U_i)
		\ar[from = U, to = D, Rightarrow, shorten <= 4pt, shorten >= 4pt]{}{\psi_i}
	\end{tikzcd}
	\]
	for all $i \in I$ then there exists a unique natural transformation $\psi:G \to H$ making $\psi_i = \gamma_i^{\ast} \ast \psi$, i.e., for which $\psi_i$ is the restriction of $\psi$ to $U_i$: $\psi_i = \psi|_{U_i}$. To this end it suffices to prove that we can apply the Gluing Lemma to the $\psi_i$. However, as post-composing the functors $G$ and $H$ by the restrictions $\gamma_i^{\ast}$ give rise to pseudocones $G:\operatorname{constant}(\Cscr) \Rightarrow \QCoh(-)$ on the cover $C$, we know that the $\psi_i$ make the diagrams
	\[
	\begin{tikzcd}
		\left(\pi_0^{ij}\right)^{\ast}\left(\gamma_i^{\ast}GA\right) \ar[rrr]{}{\left(\pi_0^{ij}\right)^{\ast}(\psi_i)} \ar[d, swap]{}{\pi_{ij}^{GA}} & & & \left(\pi_0^{ij}\right)^{\ast}\left(\gamma_i^{\ast}HA\right) \ar[d]{}{\pi_{ij}^{HA}} \\
		\left(\pi_1^{ij}\right)^{\ast}\left(\gamma_j^{\ast}GA\right) \ar[rrr, swap]{}{\left(\pi_{1}^{ij}\right)^{\ast}\psi_j} & & & \left(\pi_1^{ij}\right)^{\ast}\left(\gamma_j^{\ast}HA\right) 
	\end{tikzcd}
	\]
	commute for all $i, j \in I$ and for all $A \in \Cscr_0$. However, this shows that we can apply the Gluing Lemma and obtain a unique morphism $\psi_A:GA \to HA$ in $\QCoh(X)$ which is necessarily the $A$-component of a natural transformation which, by construction, satisfies $(\psi_A)|_{U_i} = \gamma_i^{\ast}(\psi_A) = \psi_i^{A}$. Thus we obtain a unique natural transformation $\psi:G \Rightarrow H$ for which $\psi_i = \gamma_i^{\ast} \ast \psi$ for all $i \in I$. This shows that $\QCoh(X)$ satisfies the pseudounviersal property required of a pseudolimit and so completes the proof that $\QCoh(X)$ is the claimed pseudolimit.
\end{proof}
\begin{Theorem}[{\cite[Tag 01BK]{stacks-project}}]\label{Thm: Section Background Scheme: Checking QCoh affine locally}
	Let $X$ be a scheme and let $\Fscr$ be a sheaf of $\Ocal_X$-modules on $X$. Then $\Fscr$ is quasi-coherent if and only if there exists an open affine cover $\lbrace \gamma_i:U_i  \to X \; | \; i \in I \rbrace$ of $X$ for which if $U_i \cong \Spec A_i$ then for all $i \in I$ there exists an $A_i$-module $M_i$ such that there are isomorphisms
	\[
	\gamma_i^{\ast}\left(\Fscr\right) \cong \widetilde{M_i}, \quad \left(\pi_0^{ij}\right)^{\ast}\widetilde{M_i} \cong \left(\pi_1^{ij}\right)\widetilde{M_j}
	\]
	in $\QCoh(U_i)$ and $\QCoh(U_i \times_X U_j)$.
\end{Theorem}
\begin{proof}
	This follows immediately from the pseudouniversal property of $\QCoh(X)$ in the following way.
	
	$\implies$: Fix an affine open cover $\lbrace \gamma_i:U_i \to X \; | \; i \in I \rbrace$. If $\Fscr$ is quasicoherent then applying each functor $\gamma_i^{\ast}:\QCoh(X) \to \QCoh(U_i)$ to $\underline{\Fscr}$ gives rise to a quasi-coherent sheaf $\gamma_i^{\ast}(\Fscr)$ on $U_i$. Since $U_i$ is affine for all $i \in I$, we can again apply Theorem \ref{Thm: Section Background Scheme: QCoh on affine is just modules} to find an $\Ocal_{U_i}(\lvert U_i \rvert)$-module $M_i$ with $\gamma_i^{\ast}(\Fscr) \cong \widetilde{M}_i$.
	
	$\impliedby$: Consider the functor
	\[
	\underline{\Fscr}:\mathbbm{1} \to \Mod{\Ocal_X}
	\]
	which picks out the sheaf $\Fscr$. The information we assume to give this direction of the proof implies that the functors $\gamma_i^{\ast} \circ \underline{\Fscr}:\mathbbm{1}\to \QCoh(U_i)$, for all $i \in I$, form a pseudocone:
	\[
	\begin{tikzcd}
		\mathbbm{1} \ar[rrr, ""{name = U}]{}{\gamma_i^{\ast} \circ \underline{\Fscr}} \ar[d, swap]{}{\gamma_j^{\ast} \circ \underline{\Fscr}} & & & \QCoh(U_i) \ar[d]{}{(\pi_0^{ij})^{\ast}} \\
		\QCoh(U_j) \ar[rrr, swap, ""{name= D}]{}{(\pi_1^{ij})^{\ast}} & & & \QCoh(U_i \times_{X} U_j)
		\ar[from = U, to = D, Rightarrow, shorten <= 4pt, shorten >= 4pt]{}{\cong}
	\end{tikzcd}
	\]
	By Theorem \ref{Thm: Section Background Scheme: QCoh is a pseudolimit} this implies that the $\gamma_i^{\ast} \underline{\Fscr}$ factor through $\QCoh(X)$ via
	\[
	\underline{\Fscr}:\mathbbm{1} \to \QCoh(X)
	\]
	and so gives the result.
\end{proof}
A convenient corollary of Theorem \ref{Thm: Section Background Scheme: QCoh is a pseudolimit} is that we can also write $\QCoh(X)$ as the pseudolimit of module categories taken over the poset of affine open subschemes of $X$. We simply state this theorem, as it is established similarly. However, making use of this theorem is a much more convenient way to prove that $\QCoh(X)$ is locally presentable.
\begin{Theorem}[Folklore]\label{Thm: Section QCoh: Qcoh is pseudolim over poset of affine opens}
	Let $X$ be a scheme and let $\mathbf{AffOp}(X)$ be the poset of affine open subschemes of $X$, ordered by inclusion. Then there is a pseudolimit decomposition
	\[
	\QCoh(X) \cong \operatorname*{pslim}_{U \in \mathbf{AffOp}(X)} \QCoh(U).
	\]
\end{Theorem}
\begin{corollary}[Folklore]\label{Cor: Section QCoh: QCoh Locally presentable}
	The category $\QCoh(X)$ is locally presentable.
\end{corollary}
\begin{proof}
	Recall that for any commutative ring $A$, the category $\Mod{A}$ of $A$-modules is locally presentable because $\Mod{A}$ is an AB5 Abelian category which contains a generator (namely $A$ itself). Now write
	\[
	\QCoh(X) \cong \operatorname*{pslim}_{U \in \mathbf{AffOp}(X)} \QCoh(U)
	\]
	and note that if we have an inclusion $j:U \to V$ in $\mathbf{AffOp}(X)$ (so notably $U$ is an affine open subscheme of $V$ via the open immersion $j$), then the functor $j^{\ast}:\QCoh(V) \to \mathbf{QCoh}(U)$ is left exact\footnote{Open immersions $j:U \to V$ are flat (cf.~ \cite[Proposition III.9.2]{Hartshorne}) and so $j^{\ast}$ is left exact.} and cocontinuous. Thus each functor $j^{\ast}$ is an accessible functor and so by \cite[Theorem 5.1.6, Corollary 5.1.8]{MakkaiPareAccess} we have that the pseudolimit $\QCoh(X)$ is an accessible category. Finally, as $\QCoh(X)$ is cocomplete because it is an Abelian category\footnote{It is known that if $\Cscr$ is a $1$-category and if $F:\Cscr^{\op} \to \fCat$ is a pseudofunctor with the properties that for objects $X \in \Cscr_0$ and morphisms $f \in \Cscr_1$, $F(X)$ is an Abelian category and $F(f)$ is exact, then the pseudolimit $\operatorname{pslim}(F)$ is an Abelian category as well. An explicit reference for this fact may be found in \cite[Proposition 3.1.12]{PseudoconeMonograph}.} it follows from \cite[Corollary 2.47]{AdamekRosicky} that $\QCoh(X)$ is a locally presentable category.
\end{proof}

Let us now return to the situation where $f:X \to Y$ is a morphism of schemes. While we have seen that the pullback functor
\[
f^{\ast}:\QCoh(Y) \to \QCoh(X)
\] 
is cocontinuous, it unfortunately need not be the case that the usual sheaf-theoretic pushforward $f_{\ast}:\Mod{\Ocal_X} \to \Mod{\Ocal_Y}$ only restricts to a functor 
\[
f_{\ast}:\QCoh(X) \to \Mod{\Ocal_Y}
\] 
because $f_{\ast}$ need not be right exact (as a right adjoint it generically only preserves limits). It is true, however, that $f^{\ast}$still admits a right adjoint; we will argue below that it simply need not coincide with the pushforward functor $f_{\ast}$ usually defined on $\Ocal_X$-modules. The construction of this right adjoint $f_{\square}$ of the pullback $f^{\ast}$ essentially follows from the facts that the categories of quasi-coherent sheaves are locally presentable and from the Special Adjoint Functor Theorem.
\begin{proposition}[Folklore]\label{Prop: Section Background Scheme: Right adjoint for quasicoherent sheaves}
For any morphism $f:X \to Y$ of schemes there is an adjunction:
\[
\begin{tikzcd}
\QCoh(X) \ar[rr, bend right = 20, swap, ""{name = R}]{}{f_{\square}} & & \QCoh(Y) \ar[ll, bend right = 20, swap, ""{name = L}]{}{f^{\ast}}
\ar[from = L, to = R, symbol = \dashv]
\end{tikzcd}
\]
\end{proposition}
\begin{proof}
Begin by recalling from Corollary \ref{Cor: Section QCoh: QCoh Locally presentable} that both categories $\QCoh(X)$ and $\QCoh(Y)$ are locally presentable. Additionally, the functor $f^{\ast}:\QCoh(Y) \to \QCoh(X)$ is cocontinuous. Thus we may apply the Special Adjoint Functor Theorem (\cite[Theorem 1.66]{AdamekRosicky}) to derive that $f^{\ast}$ admits a right adjoint.
\end{proof}

While the above argument does \emph{not} indicate that the right adjoint $f_{\square}$ coincides with the usual sheaf-theoretic pushforward, it is the case that under moderate finiteness conditions on the map $f$ the two functors agree. That is, we \emph{can} get that $f_{\ast}$ restricts to a functor between quasi-coherent sheaves whenever $f$ is quasi-compact and quasi-separated\footnote{Following \cite[D{\'e}finition 6.6.1]{EGA1}, a morphism $f:X \to Y$ of schemes is quasi-compact for all quasi-compact open subsets $V$ of $\lvert Y \rvert$, the preimage $f^{-1}(V)$ is quasi-compact in $X$. Because we study quasi-separated morphisms in Section \ref{Section: Appies ofTan on Sch}, we defer the precise statement of what it means to be quasi-separated to Definition \ref{Defn: qsMorphism} below.} This is well-known to algebraic geometers, but we record it here for the reader who is not an algebraic geometer or for the algebraic geometer who tends to work only with varieties.
\begin{proposition}[Well-Known to Experts]\label{Prop: Section Background Scheme: Qcqs maps have quasicompact pushforward}
Let $f:X \to Y$ be a quasi-compact quasi-separated morphism of schemes. Then $f_{\ast}:\Mod{\Ocal_X} \to \Mod{\Ocal_Y}$ restricts to a functor
\[
f_{\ast}:\QCoh(X) \to \QCoh(Y)
\]
and fits into the adjunction:
\[
\begin{tikzcd}
\QCoh(X) \ar[rr, swap, bend right = 20, ""{name = R}]{}{f_{\ast}} & & \QCoh(Y) \ar[ll, swap, bend right = 20, ""{name = L}]{}{f^{\ast}}
\ar[from = L, to = R, symbol = \dashv]
\end{tikzcd}
\]
\end{proposition}
\begin{proof}[Sketch]
By \cite[Proposition 9.2.1]{EGA1} $f_{\ast}$ restricts to a functor of quasi-coherent sheaves precisely when $f$ has the following property: there exists an open affine cover $\lbrace U_i \; | \; i \in I \rbrace$ of $Y$ for which each pullback $f^{-1}U_i$ admits a finite cover of open affine subschemes $\lbrace V_j \to f^{-1}(U_i)\; | \; j \in J_i \rbrace$ for which each pullback $V_j \times_X V_k$ is itself covered by finitely many affine open subschemes of $X$. Since affine schemes are compact (in the sense that if $S$ is affine $\lvert S \rvert$ is compact as a topological space) and since $f$ is quasi-compact, each pullback scheme $f^{-1}(U_i)$ is quasi-compact and so any affine open cover of $f^{-1}(U_i)$ may be refined to a finite subcover. Now since $f$ is quasi-separated, by \cite[Proposition 1.2.7]{EGA41} for any (finite) open affine cover $\lbrace V_j \; | \; j \in J_i \rbrace$ of $f^{-1}(U_i)$, the pullback $V_j \times_X V_k$ is quasi-compact. As such any open affine cover admits a finite refinement and so $V_j \times_X V_k$ may be covered by finitely many affine opens. Applying \cite[Proposition 9.2.1]{EGA1} thus gives the result.
\end{proof}
\begin{remark}
It is known that, unfortunately, there are examples of scheme maps for which the pushforward $f_{\ast}$ does not carry quasi-coherent sheaves to quasi-coherent sheaves. Instead we note that Proposition \ref{Prop: Section Background Scheme: Qcqs maps have quasicompact pushforward} can be rephrased as saying that when $f$ is quasi-compact quasi-separated there is a natural isomorphism $f_{\square} \cong f_{\ast}$.
\end{remark}

Perhaps unsurprisingly, the categories $\QCoh(X)$ and adjunctions $f^{\ast} \dashv f_{\square}$ fit into a pseudofunctor taking values in the $2$-category $\Adj$ of adjunctions. For any morphism $f:X \to Y$ of schemes, we want to send $f$ to the adjunction:
\[
\begin{tikzcd}
\QCoh(X) \ar[rr, bend right, swap, ""{name = R}]{}{f_{\square}} & & \QCoh(Y) \ar[ll, swap, bend right = 20, ""{name = L}]{}{f^{\ast}}
\ar[from = L, to = R, symbol = \dashv]
\end{tikzcd}
\] 
Because each pullback functor $f^{\ast}:\QCoh(Y) \to \QCoh(X)$ is defined by
\[
f^{\ast}(\Fscr) = f^{-1}\Fscr \otimes_{f^{-1}\Ocal_Y} \Ocal_X
\] 
the usual tensor yoga and preimage compositor yoga shows that we have compositor natural isomorphisms
\[
\phi_{f,g}:g^{\ast} \circ f^{\ast} \xRightarrow{\cong} (g \circ f)^{\ast}
\]
whenever $f:X \to Y$ and $g:Y \to Z$ are composable morphisms. Note that, as the exposition above implies, these isomorphisms are the usual natural isomorphisms between
\begin{align*}
f^{\ast}\left(g^{\ast}\Fscr\right) &= f^{\ast}\left(g^{-1}\Fscr \otimes_{g^{-1}\Ocal_Z} \Ocal_Y\right) \cong \left(f^{-1}\left(g^{-1}\Fscr\right) \otimes_{f^{-1}(g^{-1}\Ocal_Z)} f^{-1}\Ocal_Y\right) \otimes_{f^{-1}\Ocal_Y} \Ocal_X \\
&\cong f^{-1}\left(g^{-1}\Fscr\right) \otimes_{f^{-1}(g^{-1}\Ocal_Z)} \Ocal_X
\end{align*}
and
\[
(g \circ f)^{\ast}\Fscr = (g \circ f)^{-1}\Fscr \otimes_{(g \circ f)^{-1}\Ocal_Z} \Ocal_X
\]
for quasi-coherent sheaves $\Fscr$ on $Z$. The mates of these natural isomorphisms are precisely the witness natural isomorphisms
\[
\phi_{f,g}^{m}:(g \circ f)_{\square} \xRightarrow{\cong} g_{\square} \circ f_{\square}
\]
which indicate that the right adjoints $f_{\square}$ and $g_{\square}$ compose. Assembling these give rise to a pseudofunctor
\[
\QCoh:\Sch \to \Adj
\]
whose corresponding left adjoint projection is the pseudofunctor
\[
\QCoh:\Sch^{\op} \to \fCat, f \mapsto f^{\ast}
\]
and whose corresponding right adjoint is the pseudofunctor
\[
\QCoh:\Sch \to \fCat, f \mapsto f_{\square}.
\]
We will denote this pseudofunctor as $\QCoh(-):\Sch \to \Adj$.

If instead we work with relative schemes $X \to S$, we can apply the same construction as above and get the relative version of the pseudofunctor of quasi-coherent sheaves
\[
\QCoh:\Sch_{/S} \to \Adj, \quad f\mapsto \begin{tikzcd}
\QCoh(X) \ar[rr, bend right = 20, swap, ""{name = R}]{}{f_{\square}} & & \QCoh(Y) \ar[ll, bend right = 20, swap, ""{name = L}]{}{f^{\ast}}
\ar[from = L, to = R, symbol = \dashv]
\end{tikzcd}.
\]
The associated Grothendieck construction applied to this pseudofunctor gives the bifibration of quasi-coherent sheaves.
\begin{definition}
	The \emph{bifibration of quasi-coherent sheaves over schemes} is the bifibration $p:\mathcal{QC} \to \Sch^{\op}$ associated to the Grothendieck construction of the pseudofunctor $\QCoh(-):\Sch\to \Adj$.
\end{definition}
For what is most relevant to this paper, we will work with the left adjoint fibration $p:\QQCoh \to \Sch$.
\begin{definition}
	The \emph{fibration of quasi-coherent sheaves over schemes} is the fibration $p:\QQCoh \to \Sch$ associated to the Grothendieck construction of the pseudofunctor $\QCoh:\Sch^{\op} \to \fCat$ given by $f \mapsto f^{\ast}$ on morphisms.
\end{definition}
This means in particular that morphisms $(X,\Fscr) \to (Y,\Gscr)$ in $\QQCoh$ are pairs $(f,\rho)$  where $f:X \to Y$ is a morphism of schemes and where $\rho:\Fscr \to f^{\ast}\Gscr$ is a morphism of quasi-coherent sheaves.

Another category of interest for us (which will become important for framing the flavour of functoriality that the sheaves of relative K{\"a}hler differentials satisfy; cf.\@ Proposition \ref{Prop: Section Kahlers: Kahler diffles are functors}) regarding quasi-coherent sheaves is the corresponding dual fibration (cf.\@ \cite[Definition 8.3.5]{BorceuxHandBook2}) of the fibration $p:\QQCoh \to \Sch$ described above. While formally this is simply the fibration induced via the diagram of pseudofunctors
\[
\begin{tikzcd}
\Sch_{/S} \ar[rr]{}{\QCoh(-)} & &\fCat \ar[r]{}{(-)^{\op}} & \fCat^{\operatorname{co}}
\end{tikzcd}
\]
it is important enough for our purposes that it warrants an explicit description. Consequently we will prove explicitly that it is a fibration as described.
\begin{definition}
Consider the category $\mathbb{QC}$ defined by:
\begin{itemize}
	\item Objects: Pairs $(X,\Fscr)$ where $X$ is a scheme and where $\Fscr$ is a quasi-coherent sheaf on $X$.
	\item Morphisms: A morphism $(X,\Fscr) \to (Y,\Gscr)$ is a pair $(f,\rho)$ where $f:X \to Y$ is a map of schemes and where $\rho:f^{\ast}\Gscr \to \Fscr$ is a map of quasi-coherent sheaves.
	\item Composition: Given maps $(f,\rho):(X,\Fscr) \to (Y,\Gscr)$ and $(g,\varphi):(Y,\Gscr) \to (Z,\Hscr)$, the composition is defined to be $(g \circ f, \rho \bullet \varphi):(X,\Fscr) \to (Z,\Hscr)$ where $\rho \bullet \varphi$ is the map:
	\[
	\begin{tikzcd}
	(g \circ f)^{\ast}\Hscr \ar[r]{}{\phi_{f,g}^{-1}} & f^{\ast}\big(g^{\ast}\Hscr\big) \ar[r]{}{f^{\ast}\varphi} & f^{\ast}\Gscr \ar[r]{}{\rho} & \Fscr
	\end{tikzcd}
	\]
	\item Identities: The identity on $(X,\Fscr)$ is $(\id_X,\id_{\Fscr})$.
\end{itemize}
\end{definition}
\begin{proposition}\label{Prop: Section Background Scheme: QCoh cofibration}
Let $S$ be a base scheme. The category $\mathbb{QC}$ together with the projection 
\[
\pr:\mathbb{QC} \to \Sch_{S}, \quad (X,\Fscr) \mapsto X, \quad (f,\rho) \mapsto f
\]
is a fibration.
\end{proposition}
\begin{proof}
Begin by observing that if $f:X \to Y$ is any morphism of schemes and if $(Y,\Fscr)$ is an object in $\mathbb{QC}$ then $\pr(Y,\Fscr) = Y$. Additionally, the map
\[
(X,f^{\ast}\Fscr) \xrightarrow{(f,\id_{f^{\ast}\Fscr})} (Y,\Fscr)
\]
is a map in $\mathbb{QC}$ with $\pr(f,\id) = f$. Our goal now is to prove that this is a Cartesian lift of $f$.

Assume that we have a cospan
\[
\begin{tikzcd}
(Z,\Gscr) \ar[drr]{}{(h,\rho)} \\
(X,f^{\ast}\Fscr) \ar[rr, swap]{}{(f,\id_{f^{\ast}\Fscr})} & & (Y,\Fscr)
\end{tikzcd}
\]
in $\mathbb{QC}$ for which there is a morphism $g:Z \to X$ making the diagram
\[
\begin{tikzcd}
Z \ar[dr]{}{h} \ar[d, swap]{}{g} \\
X \ar[r, swap]{}{f} & Y
\end{tikzcd}
\]
commute in $\Sch_{/S}$. Consider now that since $f \circ g = h$ and since $\rho:h^{\ast}\Fscr \to \Gscr$, by post-composing $\rho$ with the compositor isomorphism $\phi_{g,f}$ we get a map $\hat{\rho}$ defined via the composition:
\[
\begin{tikzcd}
g^{\ast}(f^{\ast}\Fscr) \ar[d, swap]{}{\hat{\rho}} \ar[r]{}{\phi_{g,f}} & (f \circ g)^{\ast}\Fscr \ar[d, equals] \\
\Gscr & h^{\ast}\Fscr \ar[l]{}{\rho} 
\end{tikzcd}
\]
Then $(g,\hat{\rho}):(Z,\Gscr) \to (X,f^{\ast}\Fscr)$ is a map in $\mathbb{QC}$. Furthermore, we compute also that
\[
\hat{\rho} \bullet \id_{f^{\ast}\Fscr} = \hat{\rho} \circ g^{\ast}\id_{f^{\ast}\Fscr} \circ \phi_{g,f}^{-1} = \hat{\rho} \circ \phi_{g,f}^{-1} = \rho \circ \phi_{g,f} \circ \phi_{g,f}^{-1} = \rho 
\]
which shows that the diagram
\[
\begin{tikzcd}
	(Z,\Gscr) \ar[drr]{}{(h,\rho)} \ar[d, swap]{}{(g,\hat{\rho})} \\
	(X,f^{\ast}\Fscr) \ar[rr, swap]{}{(f,\id_{f^{\ast}\Fscr})} & & (Y,\Fscr)
\end{tikzcd}
\]
commutes.

We now must show that $(g,\hat{\rho})$ is the unique such map above making the diagram commute. To this end we observe that if there is a map $(g,\alpha):(Z,\Gscr) \to (X,f^{\ast}\Fscr)$ making
\[
\begin{tikzcd}
	(Z,\Gscr) \ar[drr]{}{(h,\rho)} \ar[d, swap]{}{(g,\alpha)} \\
	(X,f^{\ast}\Fscr) \ar[rr, swap]{}{(f,\id_{f^{\ast}\Fscr})} & & (Y,\Fscr)
\end{tikzcd}
\]
commute then we necessarily have that $\rho = \alpha \bullet \id_{f^{\ast}\Fscr} = \alpha \circ \phi_{g,f}^{-1}$. However, this in turn gives that
\[
\hat{\rho} = \rho \circ \phi_{g,f} = \alpha
\]
and hence that $(g,\alpha) = (g,\hat{\rho})$. Thus $(f,\id_{f^{\ast}\Fscr})$ is a Cartesian arrow above $f$ in $\mathbb{QC}$.
\end{proof}

\subsection{A Short Primer on the Fibration of Quasi-Coherent Sheaves of Algebras}\label{Subsection: QCoh CAlg}

We can specialize the notion of a quasi-coherent sheaf of $\Ocal_X$-modules to that of a quasi-coherent sheaf of $\Ocal_X$-algebras both by picking out the monoids internal to $(\QCoh(X),\otimes_{\Ocal_X},\Ocal_X)$ or by carefully extending the definition of the underlying module functor from $\CAlg{A} \to \Mod{A}$ to the category of sheaves of $\Ocal_X$-algebras in $\Mod{\Ocal_X}$.
The trick here is to recall that a sheaf of $\Ocal_X$-algebras is a sheaf $\Ascr$ on $\lvert X \rvert$ for which for all $U \subseteq \lvert X \rvert$ open, the ring $\Ascr(U)$ is an $\Ocal_X(U)$-algebra and asking that the restriction maps be ring morphisms.
\begin{definition}\label{Defn: Section Background Scheme: QCoh Algebras Defn Internal}
	If $X$ is a locally ringed space then we define the category $\CAlg{\Ocal_X}$ of sheaves of commutative $\Ocal_X$-algebras to be the category of commutative monoid objects in $(\Mod{\Ocal_X}, \otimes_{\Ocal_X}, \Ocal_X)$.
\end{definition}

By employing said definition we define the functor
\[
\RelUnd{\Ocal_X}:\CAlg{\Ocal_X} \to \Mod{\Ocal_X}
\]
which locally is given by $\Ascr(U) \mapsto \Und{\Ocal_X(U)}(\Ascr(U))$. Grothendieck used this functor to exactly define the \emph{quasi-coherent} sheaves of $\Ocal_X$-modules as those $\Ocal_X$-algebras whose underlying sheaf of $\Ocal_X$-modules is quasi-coherent.
\begin{definition}[{\cite[Section 1.5.3, Sentence 3]{EGA01}}]\label{Defn: Section Background Scheme: QCoh Algebras}
Let $X = (\lvert X \rvert, \Ocal_X)$ be a locally ringed space. A sheaf $\Ascr$ of $\Ocal_X$-algebras is a \emph{quasi-coherent sheaf of $\Ocal_X$-modules} if the sheaf $\RelUnd{\Ocal_X}(\Ascr)$ is quasi-coherent on $X$. We write $\QCoh(X,\CAlg{\Ocal_X})$ for the category of such sheaves.
\end{definition}
\begin{remark}
	If $X$ is a scheme then to define a quasi-coherent sheaf of  commutative $\Ocal_X$-algebras $\Ascr$ is, perhaps unsurprisingly, equivalent defining a quasi-coherent sheaf $\Ascr$ with the property that each $\Ascr(U)$ is a commutative $\Ocal_X(U)$-algebra for all open subsets $U$ of $X$.
\end{remark}

This gives us the relative underlying module functor we use in this paper. It is worth noting that for our purposes we only work with the relative underlying module functor on the level of quasi-coherent sheaves; while it does extend to all $\Ocal_X$-modules, for our purposes we do not need it at that level.
\begin{definition}\label{Defn: Section Kahlers and Quasi Coherent Sheaves}
The relative underlying quasi-coherent sheaf functor 
\[
\RelUnd{\Ocal_X}:\QCoh\left(X,\CAlg{\Ocal_X}\right) \longrightarrow \QCoh(X)
\]
is given on a sheaf $\Ascr$ by sending each ring $\Ascr(U)$ to the module
\[
\Ascr(U) \mapsto \Und{\Ocal_X(U)}\left(\Ascr(U)\right)
\]
for all $U$ open in $X$ and taking the restriction morphisms to be their underlying module morphisms.
\end{definition}

Immediate from these definitions, we get that for any map of schemes $f:X \to Y$ the functor
\[
f^{\ast}:\QCoh(Y) \to \QCoh(X)
\]
preserves quasi-coherent sheaves of $\Ocal_Y$-algebras --- this is particularly evident upon applying Definition \ref{Defn: Section Background Scheme: QCoh Algebras Defn Internal}. This implies that we have a pseudofunctor
\[
\QCoh\left(-,\CAlg{\Ocal_{(-)}}\right):\Sch^{\op} \to \fCat
\]
which sends a scheme $X$ to its category of quasi-coherent sheaves of $\Ocal_X$-modules. Once again, applying the Grothendieck construction to this pseudofunctor now gives the fibration of quasi-coherent sheaves of commutative algebras.
\begin{definition}
The \emph{fibration of quasi-coherent sheaves of commutative algebras} is the (cloven) fibration $p:\QQCAlg \to \Sch$ induced by applying the Grothendieck construction to the pseudofunctor of quasi-coherent sheaves of $\Ocal_{(-)}$-algebras $\QCoh(-,\CAlg{\Ocal{(-)}})$.
\end{definition}

\subsection{Functoriality Between Fibrations of Quasi-Coherent Sheaves}\label{Subscetion: Functoriality between QCoh and QCoh CAlg}

In this subsection we will indicate the nature in which the quasi-coherent fibrations are \emph{functorial} over schemes. In particular, we will show that as in the case of the fibrations $\CCalg$ and $\MMod$, we can produce a symmetric and underlying algebra adjunction. To do this, however, we first need a relative symmetric algebra functor at the level of quasi-coherent sheaves. This is also one of the key functors which go into the construction of the tangent functor on $\Sch_{/S}$.

Let $\Fscr$ be a sheaf of $\Ocal_X$-modules on a locally ringed space $X = (\lvert X \rvert, \Ocal_X)$. By applying for each open $U \subseteq \lvert X \rvert$ we have the functor $\Sym{\Ocal_X(U)}$ we have an object-local assignment
\[
\Fscr(U) \mapsto \Sym{\Ocal_X(U)}\big(\Fscr(U)\big)
\]
which may be assembled into a sheaf $\RelSym{\Ocal_X}(\Fscr)$ whose object assignments are exactly
\[
\left(\RelSym{\Ocal_X}(\Fscr)\right)(U) := \Sym{\Ocal_X(U)}\big(\Fscr(U)\big)
\] 
by the pseudonaturality of the $\Sym{(-)}$ transformation. While this does give rise to a functor on the full level, we are primarily interested in the case when the modules we feed this relative $\sym$ functor are quasicoherent. Luckily for us, this is covered and argued in \cite{EGA2}.

\begin{lemma}[{cf.\@ \cite[Proposition 1.7.6, Corollaire 1.7.7]{EGA2}}]
	If $A$ is a commutative ring and if $M$ is an $A$-module then there is an isomorphism of quasicoherent sheaves of $\Ocal_A$-algebras
	\[
	\RelSym{A}(\widetilde{M}) \cong \widetilde{\Sym{A}(M)}
	\]
	where $\widetilde{(-)}$ is the functor appearing in Line (\ref{Eqn: Tilde functor for ring modules}). In particular, for any quasicoherent sheaf $\Fscr$ on any scheme $X$, $\RelSym{\Ocal_X}(\Fscr)$ is a quasicoherent sheaf of (commutative) $\Ocal_X$-algebras.
\end{lemma}
This allows us to define our quasicoherent sheaf relative symmetric algebra functor.
\begin{definition}
	Let $X$ be a scheme. The relative symmetric algebra functor
	\[
	\RelSym{\Ocal_X}:\QCoh(X) \to \QCoh\left(X, \CAlg{\Ocal_X}\right)
	\]
	is given on a quasicoherent sheaf $\Fscr$ by sending each $\Ocal_X(U)$-module $\Fscr(U)$ to the $\Ocal_X(U)$-algebra
	\[
	\Fscr(U) \mapsto \Sym{\Ocal_X(U)}\left(\Fscr(U)\right)
	\]
	and similarly on restriction morphisms.
\end{definition}

While not recorded in \cite{EGA2}, an immediate and well-known consequence of this construction and the adjoints $\Sym{A} \dashv \Und{A}$ is an adjunction $\RelSym{\Ocal_X} \dashv \RelUnd{\Ocal_X}.$ For the sake of completeness, we prove this below.

\begin{proposition}[Folklore]\label{Prop: Section Background Scheme: Folklore Sym is left adj to Rel Und}
	For any scheme $X$ there is an adjunction:
	\[
	\begin{tikzcd}
		\QCoh(X) \ar[rr, bend right = 20, swap, ""{name = R}]{}{\RelUnd{\Ocal_X}} & & \QCoh(X, \CAlg{\Ocal_X}) \ar[ll, bend right = 20, swap, ""{name= L}]{}{\RelSym{\Ocal_X}}
		\ar[from = L, to = R, symbol = \dashv]
	\end{tikzcd}
	\]
\end{proposition}
\begin{proof}
	Fix a quasicoherent sheaf $\Fscr$ on $X$ and a quasicoherent sheaf of $\Ocal_X$-algebras $\Ascr$. We then compute:
	\begin{prooftree}
		\AxiomC{$\varphi:\RelSym{\Ocal_X}(\Fscr) \to \Ascr$ in $\QCoh(X,\CAlg{\Ocal_X})$}
		\UnaryInfC{For all $U$ open $\varphi_U:\Sym{\Ocal_X(U)}(\Fscr(U)) \to \Ascr(U)$ in $\CAlg{\Ocal_X(U)}$ together with naturality in $U$}
		\UnaryInfC{For all $U$ open $\varphi_U^{\sharp}:\Fscr(U) \to \Und{\Ocal_X(U)}(\Ascr(U))$ in $\Mod{\Ocal_X(U)}$ together with naturality in $U$}
		\UnaryInfC{$\varphi^{\sharp}:\Fscr \to \RelUnd{\Ocal_X}(\Ascr)$ in $\QCoh(X)$}
	\end{prooftree}
	Arguing dually allows us to produce the combinator
	\begin{prooftree}
		\AxiomC{$\psi:\Fscr \to \RelUnd{\Ocal_X}(\Ascr)$ in $\QCoh(X)$}
		\UnaryInfC{$\psi^{\flat}:\RelSym{\Ocal_X}(\Fscr) \to \Ascr$ in $\QCoh(X,\CAlg{\Ocal_X})$}
	\end{prooftree}  
	from the inverse combinators $\psi_U \mapsto \psi_U^{\flat}$ for opens $U$ of $X$. Finally using the equalities $(\varphi_U^{\sharp})^{\flat} = \varphi_U$ and $(\psi_U^{\flat})^{\sharp} = \psi_U$ for all opens $U$ of $X$ give that $(\varphi^{\sharp})^{\flat} = \varphi$ and $(\psi^{\flat})^{\sharp} = \psi$, proving that $\RelSym{\Ocal_X} \dashv \RelUnd{\Ocal_X}$.
\end{proof}

As before, the relative underlying sheaf functor and the relative symmetric algebra functor assemble to give morphisms of fibrations.
\begin{proposition}\label{Prop: Section Background Scheme: Rel Und is Fibration Map}
	The functor
	\[
	\RelUnd{(-)}:\QQCAlg \to \QQCoh
	\]
	given by
	\[
	(X,\Ascr) \mapsto \left(X,\RelUnd{\Ocal_X}(\Ascr)\right)
	\]
	on objects and by, for morphisms $(f,\rho):(X,\Ascr) \to (Y,\Bscr)$,
	\[
	(f,\rho) \mapsto (f, \RelUnd{\Ocal_X}(\rho))
	\]
	is a morphism of fibrations.
\end{proposition}
\begin{proof}
	First, that $\RelUnd{(-)}$ forms a functor is straightforward to verify and omitted once we recognize that $\RelUnd{(-)}$ commutes with pullback functors by the calculation
	\[
	\RelUnd{\Ocal_X}(f^{\ast}\Ascr) = \RelUnd{\Ocal_X}\!\left(f^{-1}\!\Ascr \otimes_{f^{-1}\Ocal_Y} \Ocal_X\right) = f^{-1}\!\left(\RelUnd{\Ocal_Y}\!(\Ascr)\right) \otimes_{f^{-1}\Ocal_Y} \Ocal_X =  f^{\ast}\left(\RelUnd{\Ocal_X}(\Ascr)\right).
	\]
	As such, we must show that a Cartesian lift of $f$ in $\QQCAlg$ gets sent to a Cartesian lift of $f$ in $\QQCoh$.
	
	Assume that we have a scheme $X$, a quasi-coherent sheaf of $\Ocal_Y$-algebras $\Ascr$, and a map of schemes $f:X \to Y$.  Because the Cartesian lifts of $f$ in $\QQCAlg$ are of the form $(f,\id_{f^{\ast}\Ascr}):(X,f^{\ast}\Ascr) \to (Y,\Ascr)$, we must prove that applying $\RelUnd{(-)}$ to said map gives a Cartesian lift of $f$ in $\QQCoh$.
	
	To prove that this is indeed the case, first note that since $\RelUnd{\Ocal_X}(f^{\ast}\Ascr) = f^{\ast}(\RelUnd{\Ocal_Y}\Ascr)$ the components of $\RelUnd{(-)}(f,\id_{f^{\ast}\Ascr})$ are $f$ and the identity on the pullback of the underlying quasi-coherent sheaf of $\Ascr$. As such it suffices to show that given a $\QQCoh$ map $(g,\rho):(Z,\Fscr) \to (Y,\RelUnd{\Ocal_Y}(\Ascr))$ and a commuting diagram
	\[
	\begin{tikzcd}
		Z \ar[dr]{}{g} \ar[d, swap]{}{h} \\
		X \ar[r, swap]{}{f} & Y
	\end{tikzcd}
	\]
	of schemes then there exists a unique $\QQCoh$ map $(Z,\Fscr) \to (X,f^{\ast}(\RelUnd{\Ocal_Y}(\Ascr)))$ rendering the diagram
	\[
	\begin{tikzcd}
		(Z,\Fscr) \ar[d, dashed, swap]{}{\exists?} \ar[drrr]{}{(g,\rho)} & & & \\
		(X,\RelUnd{\Ocal_X}(f^{\ast}\Ascr)) \ar[rrr, swap]{}{(f, \id_{f^{\ast}\RelUnd{\Ocal_Y}\Ascr})} & & & (Y,\RelUnd{\Ocal_Y}(\Ascr))
	\end{tikzcd}
	\]
	commutative. However, as $\RelUnd{(-)}$ commutes with the pullback functors building a lift is routine: given the map $\rho:\Fscr \to g^{\ast}(\RelUnd{\Ocal_Y}(\Ascr))$ and the fact that $g = h \circ f$, we define the lift $\hat{\rho}:\Fscr \to h^{\ast}(\RelUnd{\Ocal_X}(f^{\ast}\Ascr))$ via the diagram
	\[
	\begin{tikzcd}
	\Fscr \ar[r]{}{\rho} \ar[dr, swap]{}{\hat{\rho}} & g^{\ast}\left(\RelUnd{\Ocal_Y}(\Ascr)\right) \ar[r, equals] & \RelUnd{\Ocal_Z}(g^{\ast}\Ascr) \ar[d]{}{\RelUnd{\Ocal_Z} \ast \phi_{h,f}^{-1}} \\
	 & h^{\ast}\left(\RelUnd{\Ocal_X}(f^{\ast}\Ascr)\right) & \RelUnd{\Ocal_Z}(h^{\ast}(f^{\ast}\Ascr)) \ar[l, equals]{}{}
	\end{tikzcd}
	\]
	where $\phi_{h,f}$ is the compositor natural isomorphism for the pullback functors. The commutativity of the triangle follows from the computation that
	\begin{align*}
	\id_{f^{\ast}\RelUnd{\Ocal_Y}\Ascr} \bullet \hat{\rho} &= \left(\RelUnd{\Ocal_Z} \ast \phi_{h,f}\right) \circ \id_{f^{\ast}\RelUnd{\Ocal_Y}\Ascr} \circ \left(\RelUnd{\Ocal_Z} \ast \phi_{h,f}^{-1}\right) \circ \rho \\
	&= \left(\RelUnd{\Ocal_Z} \ast \phi_{h,f}\right) \circ \left(\RelUnd{\Ocal_Z} \ast \phi_{h,f}^{-1}\right) \circ \rho = \rho.
	\end{align*}
	Finally the uniqueness of said lift follows from the fact that if there is a map $(h,\alpha):(Z,\Fscr) \to (X,\RelUnd{\Ocal_X}(f^{\ast}\Ascr))$ making
	\[
		\begin{tikzcd}
		(Z,\Fscr) \ar[d, swap]{}{(h,\alpha)} \ar[drrr]{}{(g,\rho)} & & & \\
		(X,\RelUnd{\Ocal_X}(f^{\ast}\Ascr)) \ar[rrr, swap]{}{(f, \id_{f^{\ast}\RelUnd{\Ocal_Y}\Ascr})} & & & (Y,\RelUnd{\Ocal_Y}(\Ascr))
	\end{tikzcd}
	\]
	 commute then
	 \[
	 \left(\RelUnd{\Ocal_Z} \ast \phi_{h,f}\right) \circ \alpha = \id_{f^{\ast}\RelUnd{\Ocal_Y}\Ascr} \bullet \alpha = \rho = \left(\RelUnd{\Ocal_Z} \ast \phi_{h,f}\right) \circ \left(\RelUnd{\Ocal_Z} \ast \phi_{h,f}\right)^{-1} \circ \rho
	 \]
	 and hence
	 \[
	 \alpha = \left(\RelUnd{\Ocal_Z} \ast \phi_{h,f}\right)^{-1} \circ \rho = \hat{\rho}.
	 \]
	 Thus $(h,\hat{\rho})$ is the unique map rendering
	 \[
	 		\begin{tikzcd}
	 	(Z,\Fscr) \ar[d, swap, dashed]{}{\exists!(h,\hat{\rho})} \ar[drrr]{}{(g,\rho)} & & & \\	 	(X,\RelUnd{\Ocal_X}(f^{\ast}\Ascr)) \ar[rrr, swap]{}{(f, \id_{f^{\ast}\RelUnd{\Ocal_Y}\Ascr})} & & & (Y,\RelUnd{\Ocal_Y}(\Ascr))
	 \end{tikzcd}
	 \]
	 commutative and so proves that $\RelUnd{(-)}$ preserves Cartesian lifts. Thus $\RelUnd{(-)}$ is a morphism of fibrations.
\end{proof}

We now want to make the same observation regarding the relative symmetric algebra functor so that we can conclude $\RelUnd{(-)}$ and $\RelSym{(-)}$ are mutual adjoints between fibrations. To do this we make the following ring-theoretic observation: because the relative symmetric algebra functors have natural isomorphisms, for any commutative ring map $f:R \to S$,
\[
\begin{tikzcd} 
\Mod{R}  \ar[d, swap]{}{\Sym{R}(-)} \ar[r, ""{name = U}]{}{(-) \otimes_R S} & \Mod{S} \ar[d]{}{\Sym{S}(-)}	\\
\CAlg{R} \ar[r, swap, ""{name = D}]{}{(-) \otimes_R S} & \CAlg{S}
\ar[from = U, to = D, Rightarrow, shorten <= 4pt, shorten >= 4pt]{}{\cong}
\end{tikzcd}
\]
which vary pseudonaturally in $\Cring$, the relative symmetric algebra functors (at the level of relative schemes now) commute up to isomorphism with pullbacks. That is, for any morphism $f:X \to Y$ of $S$-schemes there is a natural isomorphism
\[
\begin{tikzcd}
\QCoh(Y) \ar[rr, ""{name = U}]{}{f^{\ast}} \ar[d, swap]{}{\RelSym{\Ocal_Y}(-)} & & \QCoh(X) \ar[d]{}{\RelSym{\Ocal_X}} \\
\QCoh(Y, \CAlg{\Ocal_Y}) \ar[rr, swap, ""{name = D}]{}{f^{\ast}} & & \QCoh(X,\CAlg{\Ocal_X})
\ar[from = U, to = D, Rightarrow, shorten <= 4pt, shorten >= 4pt]{}{\cong}
\ar[from = U, to = D, Rightarrow, shorten <= 4pt, shorten >= 4pt, swap]{}{\alpha_f}
\end{tikzcd}
\]
which also varies pseudonaturally in the opposite category of (relative) schemes $\Sch_{/S}^{\op}$. Implementing these natural isomorphisms as coherence data for a natural isomorphism $\alpha_f:\RelSym{\Ocal_X}(f^{\ast}\Fscr) to f^{\ast}(\RelSym{\Ocal_Y}(\Fscr))$
allows us to view $\RelSym{(-)}$ as a functor between $\QQCoh$ and $\QQCAlg$ defined by sending an object $(X,\Fscr)$ to the pair $(X,\RelSym{\Ocal_X}(\Fscr))$ and sending a morphism $(f,\rho):(X,\Fscr) \to (Y,\Gscr)$ given by the data
\[
f \in \Sch_{/S}(X,Y), \qquad \rho \in \QCoh(\Fscr, f^{\ast}\Gscr)
\] 
to the pair $(f,\RelSym{(-)}(\rho))$ where $\RelSym{(-)}(\rho)$ is the map:
\[
\begin{tikzcd}
\RelSym{\Ocal_X}(\Fscr) \ar[rr]{}{\RelSym{\Ocal_X}(\rho)} & & \RelSym{\Ocal_X}(f^{\ast}\Gscr) \ar[r]{}{\alpha_f} & f^{\ast}\left(\RelSym{\Ocal_Y}(\Gscr)\right)
\end{tikzcd}
\]

Running the same argument given in Proposition \ref{Prop: Section Background Scheme: Rel Und is Fibration Map} mutatis mutandis for the functor $\RelSym{(-)}$ save that we replace the identities commuting the pullback functors with $\RelUnd{(-)}$ functors with the corresponding natural transformations $\alpha_f$ and then using the pseudonaturality of the $\alpha_f$ gives that $\RelSym{(-)}$ is also a fibration morphism.
\begin{proposition}\label{Prop: Section Background Scheme: Rel Sym is Fibration Map}
	The functor
	\[
	\RelUnd{(-)}:\QQCoh \to \QQCAlg
	\]
	given by
	\[
	(X,\Fscr) \mapsto \left(X,\RelSym{\Ocal_X}(\Fscr)\right)
	\]
	on objects and by, for morphisms $(f,\rho):(X,\Fscr) \to (Y,\Gscr)$,
	\[
	(f,\rho) \mapsto (f, \alpha_f \circ \RelSym{\Ocal_X}(\rho))
	\]
	is a morphism of fibrations.
\end{proposition}
\begin{proof}[Sketch]
The argument prior to the statement of the proposition gives that if we have a commuting triangle of schemes
\[
\begin{tikzcd}
Z \ar[dr]{}{g} \ar[d, swap]{}{h} & \\
X \ar[r, swap]{}{f} & Y
\end{tikzcd}
\]
with corresponding cospan 
\[
\begin{tikzcd}
(Z,\Ascr) \ar[drr]{}{\rho} \\
(X,\RelSym{\Ocal_X}(f^{\ast}\Fscr)) \ar[rr, swap]{}{(f, \alpha_f)} & & (Y, \RelSym{\Ocal_Y}(\Fscr))
\end{tikzcd}
\]
in $\QQCAlg$ then the unique map $\hat{\rho}$ rendering the diagram
\[
\begin{tikzcd}
(Z,\Ascr) \ar[drr]{}{(g,\rho)} \ar[d, dashed, swap]{}{\exists!(h,\hat{\rho})} \\
(X,\RelSym{\Ocal_X}(f^{\ast}\Fscr)) \ar[rr, swap]{}{(f, \alpha_f)} & & (Y,\RelSym{\Ocal_Y}(\Fscr))
\end{tikzcd}
\]
commutative in $\QQCAlg$ is the map  described by the diagram:
\[
\begin{tikzcd}
\Ascr \ar[r]{}{\rho} \ar[drrr, swap]{}{\hat{\rho}} & g^{\ast}(\RelSym{\Ocal_Y}(\Fscr)) \ar[rr]{}{\phi_{h,f}^{-1} \ast \RelSym{\Ocal_Y}} & & h^{\ast}(f^{\ast}(\RelSym{\Ocal_Y}(\Fscr))) \ar[d]{}{h^{\ast} \ast \alpha_f} \\
 & & & h^{\ast}(\RelSym{\Ocal_X}(f^{\ast}\Fscr))
\end{tikzcd}
\]
\end{proof}

We now can conclude that $\RelSym{(-)} \dashv \RelUnd{(-)}$ is an adjunction between fibrations.
\begin{proposition}\label{Prop: Section Background Scheme: Rel Sym Rel Und adjunction at fibration level}
For any base scheme $S$, there is an adjunction of fibrations:
\[
\begin{tikzcd}
\QQCAlg \ar[dr]{}{} \ar[rr, bend right = 15, swap, ""{name = R}]{}{\RelUnd{(-)}} & & \QQCoh \ar[ll, swap, bend right = 15, ""{name = L}]{}{\RelSym{(-)}} \ar[dl] \\
 & \Sch_{/S}
 \ar[from = L, to = R, symbol = \dashv]
\end{tikzcd}
\]
In particular, the relative symmetric algebra and relative underlying quasi-coherent sheaf pseudonatural transformations $\RelSym{(-)}:\QCoh(-) \Rightarrow \QCoh(-,\CAlg{\Ocal_{(-)}})$ and $\RelUnd{(-)}:\QCoh(-,\CAlg{\Ocal_{(-)}}) \Rightarrow \QCoh(-)$ are left-and-right adjoint $1$-cells in the $2$-category of pseudofunctors, pseudonatural transformations, and modifications $\Bicat(\Sch_{/S}^{\op},\fCat)$.
\end{proposition}
\begin{proof}
The second claim simply follows from the fact that the Grothendieck construction is a $2$-equivalence of categories and hence transfers adjoints. As such, it suffices to prove that $\RelSym{(-)} \dashv \RelUnd{(-)}$. To do this we show that to give a map of the form $(X,\RelSym{\Ocal_X}(\Fscr)) \to (Y,\Ascr)$ is to give a map of the form $(X,\Fscr) \to (Y,\RelUnd{\Ocal_Y}(\Ascr))$ and vice-versa. 

Fix schemes $X$ and $Y$ with $\Fscr$ a quasi-coherent sheaf over $X$ and with $\Ascr$ a quasi-coherent sheaf of $\Ocal_Y$-modules. We now derive that on one hand
\begin{prooftree}
	\AxiomC{$(f,\rho):(X,\RelSym{\Ocal_X}(\Fscr)) \rightarrow (Y,\Ascr)$ in $\QQCAlg$}\doubleLine
	\UnaryInfC{$f:X \to Y$, $\rho:\RelSym{\Ocal_X}(\Fscr) \to f^{\ast}\Ascr$}\RightLabel{Proposition \ref{Prop: Section Background Scheme: Folklore Sym is left adj to Rel Und}}
	\UnaryInfC{$f:X \to Y$, $\rho^{\sharp}:\Fscr \to \RelUnd{\Ocal_X}(f^{\ast}\Ascr)$}\doubleLine\RightLabel{$\RelUnd{\Ocal_X}(f^{\ast}\Ascr) = f^{\ast}(\RelUnd{\Ocal_Y}(\Ascr))$}
	\UnaryInfC{$f:X \to Y$, $\rho^{\sharp}:\Fscr \to f^{\ast}(\RelUnd{\Ocal_Y}(\Ascr))$}\doubleLine
	\UnaryInfC{$(f,\rho^{\sharp}):(X,\Fscr) \to (Y,\RelUnd{\Ocal_Y}(\Fscr))$ in $\QQCoh$}
\end{prooftree}
while on the other hand:
\begin{prooftree}
	\AxiomC{$(f,\rho):(X,\Fscr) \rightarrow (Y,\RelUnd{\Ocal_Y}(\Ascr))$ in $\QQCoh$}\doubleLine
	\UnaryInfC{$f:X \to Y$, $\rho:\Fscr \to f^{\ast}(\RelUnd{\Ocal_Y}(\Ascr)$}\doubleLine\RightLabel{$f^{\ast}(\RelUnd{\Ocal_Y}(\Ascr)) = \RelUnd{\Ocal_X}(f^{\ast}\Ascr)$}
	\UnaryInfC{$f:X \to Y$, $\rho:\Fscr \to \RelUnd{\Ocal_X}(f^{\ast}\Fscr)$}\RightLabel{Proposition \ref{Prop: Section Background Scheme: Folklore Sym is left adj to Rel Und}}
	\UnaryInfC{$f:X \to Y$, $\rho^{\flat}:\RelSym{\Ocal_X}(\Fscr) \to f^{\ast}\Ascr$}\doubleLine
	\UnaryInfC{$(f,\rho^{\flat}):(X,\RelSym{\Ocal_X}(\Fscr)) \to (Y,\Ascr)$ in $\QQCAlg$}
\end{prooftree}
Because the sharp and flat combinators used in each derivation are mutually inverse by virtue of the adjunction $\RelSym{\Ocal_X} \dashv \RelUnd{\Ocal_X}$, it follows that $(f,\rho^{\flat})^{\sharp} = (f,\rho)$ and $(f, \rho^{\sharp})^{\flat} = (f,\rho)$, proving our claimed adjunction holds.
\end{proof}

As a final lone result in this section we record that pullbacks commute with the relative symmetric algebra functors. 

\begin{corollary}\label{Cor: Section Background Scheme: The pullback Sym isos for qcoh}
For any morphism $f:X \to Y$ of schemes there is a natural isomorphism
\[
\begin{tikzcd}
\QCoh(Y) \ar[rr, ""{name = U}]{}{f^{\ast}} \ar[d, swap]{}{\RelSym{\Ocal_Y}} & & \QCoh(X) \ar[d]{}{\RelSym{\Ocal_X}} \\
\QCoh(Y,\CAlg{\Ocal_Y}) \ar[rr, swap, ""{name = D}]{}{f^{\ast}} & & \QCoh(X,\CAlg{\Ocal_X})
\ar[from = U, to = D, Rightarrow, shorten <= 4pt, shorten >= 4pt]{}{\cong}
\end{tikzcd}
\]
\end{corollary}
\begin{proof}[Sketch]
Fix an affine open cover $\lbrace V_i \; | \; i \in I \rbrace$ of $Y$ and an affine open cover $\lbrace U_j \; | \; j \in J_i, i \in I \rbrace$ of $X$ such that for all indices $i \in I$ and for all $j \in J_i$, the diagram
\[
\begin{tikzcd}
U_j \ar[d, swap] \ar[r]{}{f} & V_i \ar[d] \\
X \ar[r, swap]{}{f} & Y
\end{tikzcd}
\]
commutes. Then we have the natural isomorphisms
\[
\Sym{\Ocal_Y(V_i)}\big(\Fscr(V_i)\big) \otimes_{\Ocal_Y(V_i)} \Ocal_X(U_j) \cong \Sym{\Ocal_X(U_i)}\left(\Fscr(V_i) \otimes_{\Ocal_Y(V_i)} \Ocal_X(U_j)\right)
\]
for all quasi-coherent sheaves $\Fscr$ on $Y$ and for all $i \in I, j \in J_i$. These isomorphisms hence induce an isomorphism of presheaf assignments
\[
\Sym{(f^{-1}\Ocal_Y)(U_j)}\left((f^{-1}\Fscr)(U_j)\right) \otimes_{(f^{-1}\Ocal_Y)(U_j)} \Ocal_X(U_j) \cong \Sym{\Ocal_X(U_j)}\left((f^{-1}\Fscr)(U_j) \otimes_{(f^{-1}\Ocal_Y)(U_j)} \Ocal_X(U_j)\right)
\]
which are natural in the $U_j$. Sheafifying this gives our desired isomorphism
\[
f^{\ast} \circ \RelSym{\Ocal_Y} \cong \RelSym{\Ocal_X} \circ f^{\ast}
\]
\end{proof}

\subsection{The Coalgebra Structure on Relative Symmetric Algebras}\label{Subsection: Coalgebra structure on Relative Sym}

In this short section we simply record the fact that the coalgebra maps $\nabla, S, \epsilon$ defined making $(\Sym{R}(M),\nabla_M,S_M,\epsilon_M)$ into a commutative and cocommutative Hopf algebra for all commutative rings $R$ and for all $R$-modules $M$ may be appropriately sheafified. The main tool we use here is that for any module map $f:M \to N$, the Hopf map $\Sym{R}(f)$ is a Hopf algebra morphism.

We first rephrase Proposition \ref{Prop: Section Background Alg: Symmetric Alg is Coalg} in its affine scheme-theoretic incarnation. While this is obvious, it is crucial for setting the stage.

\begin{corollary}\label{Cor: Section Background Alg: Sym as a functor into CMon of CAlg op}
	For any commutative ring $R$, the functor $\Sym{R}^{\op}$ extends to a functor
	\[
	\Spec\circ \widehat{\Sym{R}}^{\op}:\Mod{R}^{\op} \to \mathbf{Ab}(\mathbf{AffSch}_{/R})
	\]
	where $\Ab(\mathbf{AffSch}_{/R})$ denotes the category of Abelian group objects in $\mathbf{AffSch}_{/R}$.
\end{corollary}
\begin{proof}
	By the formal calculus of taking opposites and spectra, it follows that cocommutative Hopf algebras in the symmetric monoidal category $(\CAlg{R},\otimes_R,R)$ are precisely internal Abelian groups in $(\mathbf{AffSch}_{/R},\times_R, \Spec R)$.
\end{proof}

The strict naturality of the functors $\widehat{\Sym{(-)}}$ with restriction of scalars functors (cf.\@ Proposition \ref{Prop: Section Background Alg: Strictness of Hopf for Restriction of Scalars}) imply that for any scheme $X$ and for any quasicoherent sheaf $\Fscr$ on $X$, the map
\begin{prooftree}
	\AxiomC{$\nabla^{\operatorname{pre}}_{\Fscr}:\RelSym{\Ocal_X}(\Fscr) \to \RelSym{\Ocal_X}\Fscr \otimes_{\Ocal_X} \RelSym{\Ocal_X}\Fscr$ in $[\Open(X)^{\op},\Set]$}\doubleLine
	\UnaryInfC{$U \in \mathbf{Open}(X)$. $(\nabla_{\Fscr}^{\operatorname{pre}})_U := \nabla_{\Fscr(U)}:\Sym{\Ocal_X(U)}\Fscr(U) \to \Sym{\Ocal_X(U)}\Fscr(U) \otimes_{\Ocal_X(U)} \Sym{\Ocal_X(U)}\Fscr(U)$}
\end{prooftree}
is a morphism of presheaves of commutative $\Ocal_X$-algebras; note that the naturality of $\nabla$ is nontrivial, as it asks that
\[
\begin{tikzcd}
\Sym{\Ocal_X(U)}\Fscr(U) \ar[d] \ar[rrr]{}{\nabla_U} & & & \Sym{\Ocal_X(U)}\Fscr(U) \otimes_{\Ocal_X(U)} \Sym{\Ocal_X(U)}\Fscr(U) \ar[d] \\
\Sym{\Ocal_X(V)}\Fscr(V) \ar[rrr, swap]{}{\nabla_{V}} & & & \Sym{\Ocal_X(V)}\Fscr(V) \otimes_{\Ocal_X(V)} \Sym{\Ocal_X(V)}\Fscr(V)
\end{tikzcd}
\]
commute for all opens $V \subseteq U$ in $X$. However, this is precisely given by the analysis of the map $F$ in the statement prior to Proposition \ref{Prop: Section Background Alg: Strictness of Hopf for Restriction of Scalars}. Using this we define $\nabla$ via sheafification.
\begin{definition}\label{Defn: Section Background Scheme: Nabla on QCoh}
For any scheme $X$ and any quasi-coherent sheaf $\Fscr$, define the map
\[
\nabla_{\Fscr}:\RelSym{\Ocal_X}\!\left(\Fscr\right) \longrightarrow \RelSym{\Ocal_X}\!\left(\Fscr\right) \otimes_{\Ocal_X} \RelSym{\Ocal_X}\!\left(\Fscr\right)
\]
by the sheafification
\[
\nabla_{\Fscr} := \left(\nabla_{\Fscr}^{\operatorname{pre}}\right)^{++}.
\]
\end{definition}

In order to see the maps $\nabla_{\Fscr}$ are natural transformations in the sense of arising as the $\Fscr$-component of a $2$-cell
\[
\begin{tikzcd}
\QCoh(X) \ar[rrr, bend left = 20, ""{name = U}]{}{\RelSym{\Ocal_X}} \ar[rrr, swap, bend right = 20, ""{name = D}]{}{\left(\RelSym{\Ocal_X}(-) \otimes \RelSym{\Ocal_X}(-)\right) \circ \Delta} & & & \QCoh(X,\CAlg{\Ocal_X})
\ar[from = U, to = D, Rightarrow, shorten <= 4pt, shorten >= 4pt]{}{\nabla}
\end{tikzcd}
\]
for $\Delta$ the diagonal functor $\QCoh(X) \to \QCoh(X) \times \QCoh(X)$, we also need to know that if $\varphi:\Fscr \to \Gscr$ is a morphism of quasi-coherent sheaves on $X$ then 
\[
\begin{tikzcd}
\RelSym{\Ocal_X}\!(\Fscr) \ar[d, swap]{}{\RelSym{\Ocal_X}\!(\varphi)} \ar[rr]{}{\nabla_{\Fscr}} & & \RelSym{\Ocal_X}\!(\Fscr) \otimes_{\Ocal_X} \RelSym{\Ocal_X}\!(\Fscr) \ar[d]{}{\RelSym{\Ocal_X}\!(\varphi) \otimes \RelSym{\Ocal_X}\!(\varphi)} \\
\RelSym{\Ocal_X}\!(\Gscr) \ar[rr, swap]{}{\nabla_{\Gscr}} & & \RelSym{\Ocal_X}\!(\Gscr)\otimes_{\Ocal_X} \RelSym{\Ocal_X}\!(\Gscr)
\end{tikzcd}
\]
commutes. Because $\nabla_{\Fscr} = (\nabla_{\Fscr}^{\operatorname{pre}})^{++}$ and similarly for $\nabla_{\Gscr}$, it suffices to argue for all opens $U$ of $X$ that the diagram
\[
\begin{tikzcd}
\Sym{\Ocal_X(U)}\!\big(\Fscr(U)\big) \ar[rrr]{}{(\nabla_{\Fscr}^{\operatorname{pre}})_U} \ar[d, swap]{}{\Sym{\Ocal_X(U)}(\varphi_U)} & & & \Sym{\Ocal_X(U)}\!\big(\Fscr(U)\big) \otimes_{\Ocal_X(U)} \Sym{\Ocal_X(U)}\!\big(\Fscr(U)\big) \ar[d]{}{\Sym{\Ocal_X(U)}(\varphi_U) \otimes \Sym{\Ocal_X(U)}(\varphi_U)} \\
\Sym{\Ocal_X(U)}\big(\Gscr(U)\big) \ar[rrr, swap]{}{(\nabla_{\Gscr}^{\operatorname{pre}})_U} & & & \Sym{\Ocal_X(U)}\big(\Gscr(U)\big) \otimes_{\Ocal_X(U)} \Sym{\Ocal_X(U)}\big(\Gscr(U)\big)
\end{tikzcd}
\]
commutes. However, by Proposition \ref{Prop: Section Background Alg: Symmetric Alg is Coalg} we know that every map $\rho$ of $\Ocal_X(U)$-algebras has $\Sym{\Ocal_X(U)}(\rho)$ a coalgebra morphism, so the diagram above indeed commutes for all $U$ open. Thus 
\[
\nabla_{\Gscr}^{\operatorname{pre}} \circ \RelSym{\Ocal_X}\!(\varphi) = \left(\RelSym{\Ocal_X}\!(\varphi) \otimes \RelSym{\Ocal_X}\!(\varphi)\right) \circ \nabla_{\Fscr}^{\operatorname{pre}}
\]
and so sheafifying gives us the lemma below.
\begin{lemma}\label{Lemma: Section Background Scheme: Nabla as a natural comultiplication}
Let $X$ be a scheme and let $\varphi:\Fscr \to \Gscr$ be a morphism of quasi-coherent sheaves on $X$. Then
\[
\begin{tikzcd}
\RelSym{\Ocal_X}\!(\Fscr) \ar[d, swap]{}{\RelSym{\Ocal_X}\!(\varphi)} \ar[rr]{}{\nabla_{\Fscr}} & & \RelSym{\Ocal_X}\!(\Fscr) \otimes_{\Ocal_X} \RelSym{\Ocal_X}\!(\Fscr) \ar[d]{}{\RelSym{\Ocal_X}\!(\varphi) \otimes \RelSym{\Ocal_X}\!(\varphi)} \\
\RelSym{\Ocal_X}\!(\Gscr) \ar[rr, swap]{}{\nabla_{\Gscr}} & & \RelSym{\Ocal_X}\!(\Gscr) \otimes_{\Ocal_X} \RelSym{\Ocal_X}\!(\Gscr)
\end{tikzcd}
\]
commutes. In particular, $\nabla$ arises as a natural transformation:
\[
\begin{tikzcd}
	\QCoh(X) \ar[rrr, bend left = 20, ""{name = U}]{}{\RelSym{\Ocal_X}} \ar[rrr, swap, bend right = 20, ""{name = D}]{}{\left(\RelSym{\Ocal_X}(-) \otimes \RelSym{\Ocal_X}(-)\right) \circ \Delta} & & & \QCoh(X,\CAlg{\Ocal_X})
	\ar[from = U, to = D, Rightarrow, shorten <= 4pt, shorten >= 4pt]{}{\nabla}
\end{tikzcd}
\]
\end{lemma}

We also need to define the counit $\epsilon_{\Fscr}:\RelSym{\Ocal_X}\!(\Fscr) \to \Ocal_X$ and antipode $S_{\Fscr}:\RelSym{\Ocal_X}\!(\Fscr) \to \RelSym{\Ocal_X}\!(\Fscr)$ morphisms, respectively, defining the Hopf algebra structure on $\RelSym{\Ocal_X}\!(\Fscr)$. However, this is significantly more straightforward than defining the comultiplication $\nabla$. For $\epsilon_{\Fscr}$ we simply declare
\begin{prooftree}
	\AxiomC{$\epsilon_{\Fscr}:\RelSym{\Ocal_X}\!(\Fscr) \to \Ocal_X$ in $\QCoh(X,\CAlg{\Ocal_X})$}\doubleLine
	\UnaryInfC{For $U$ open, $(\epsilon_{\Fscr})_U := \epsilon_{\Fscr(U)}:\Sym{\Ocal_X(U)}\!(\Fscr(U)) \to \Ocal_X(U)$ in $\CAlg{\Ocal_X(U)}$}
\end{prooftree}
while for $S_{\Fscr}$ we define
\begin{prooftree}
	\AxiomC{$S_{\Fscr}:\RelSym{\Ocal_X}\!(\Fscr) \to \RelSym{\Ocal_X}\!(\Fscr)$ in $\QCoh(X,\CAlg{\Ocal_X})$}\doubleLine
	\UnaryInfC{For all $U$ open, $(S_{\Fscr})_U := S_{\Fscr(U)}:\Sym{\Ocal_X(U)}\!(\Fscr(U)) \to \Sym{\Ocal_X(U)}\!(\Fscr(U))$ in $\CAlg{\Ocal_X(U)}$}
\end{prooftree}
where each $\epsilon$ and $S$ are the corresponding structure maps asserted by Proposition \ref{Prop: Section Background Alg: Symmetric Alg is Coalg}. The fact that these all assemble to sheaf morphisms arises as a consequence of Proposition \ref{Prop: Section Background Alg: Strictness of Hopf for Restriction of Scalars}, as it appealing to said result implies that for any quasi-coherent sheaf $\Fscr$ and for any opens $V \subseteq U$ in $X$ the diagrams
\[
\begin{tikzcd}
\Sym{\Ocal_X(U)}\!(\Fscr(U)) \ar[r]{}{\epsilon_{\Fscr(U)}} \ar[d, swap] & \Ocal_X(U) \ar[d] \\
\Sym{\Ocal_X(V)}\!(\Fscr(V)) \ar[r, swap]{}{\epsilon_{\Fscr(V)}} & \Ocal_X(V)
\end{tikzcd}\quad
\begin{tikzcd}
\Sym{\Ocal_X(U)}\!(\Fscr(U)) \ar[r]{}{S_{\Fscr(U)}} \ar[d, swap] & \Sym{\Ocal_X(U)}\!(\Fscr(U)) \ar[d] \\
\Sym{\Ocal_X(V)}\!(\Fscr(V)) \ar[r, swap]{}{S_{\Fscr(V)}} & \Sym{\Ocal_X(V)}\!(\Fscr(V))
\end{tikzcd}
\]
commute with the unlabelled maps the induced restriction maps. Finally the fact that they induce natural transformations
\[
\begin{tikzcd}
\QCoh(X) \ar[rr, bend left = 20, ""{name = U}]{}{\RelSym{\Ocal_X}} \ar[rr, swap, bend right = 20, ""{name = D}]{}{\cnst(\Ocal_X)} & & \QCoh(X,\CAlg{\Ocal_X})
\ar[from = U, to = D, Rightarrow, shorten <= 4pt, shorten >= 4pt]{}{\epsilon}
\end{tikzcd}\quad
\begin{tikzcd}
	\QCoh(X) \ar[rr, bend left = 20, ""{name = U}]{}{\RelSym{\Ocal_X}} \ar[rr, swap, bend right = 20, ""{name = D}]{}{\RelSym{\Ocal_X}} & & \QCoh(X,\CAlg{\Ocal_X})
	\ar[from = U, to = D, Rightarrow, shorten <= 4pt, shorten >= 4pt]{}{S}
\end{tikzcd}
\]
for $\cnst(\Ocal_X)$ the constant functor outputting $\Ocal_X$, arises by an application of Proposition \ref{Prop: Section Background Alg: Symmetric Alg is Coalg} (which is used to apply the fact that each map $\Sym{\Ocal_X(U)}\!(\rho_U)$ is a Hopf algebra morphism). Putting these observations together we get the lemma below and can immediately prove the resulting theorem.
\begin{lemma}\label{Lemma: Section Background Scheme: Counit and Unit as natural transforms}
Let $X$ be a scheme. Then there are natural transformations
\[
\begin{tikzcd}
	\QCoh(X) \ar[rr, bend left = 20, ""{name = U}]{}{\RelSym{\Ocal_X}} \ar[rr, swap, bend right = 20, ""{name = D}]{}{\cnst(\Ocal_X)} & & \QCoh(X,\CAlg{\Ocal_X})
	\ar[from = U, to = D, Rightarrow, shorten <= 4pt, shorten >= 4pt]{}{\epsilon}
\end{tikzcd}\quad
\begin{tikzcd}
	\QCoh(X) \ar[rr, bend left = 20, ""{name = U}]{}{\RelSym{\Ocal_X}} \ar[rr, swap, bend right = 20, ""{name = D}]{}{\RelSym{\Ocal_X}} & & \QCoh(X,\CAlg{\Ocal_X})
	\ar[from = U, to = D, Rightarrow, shorten <= 4pt, shorten >= 4pt]{}{S}
\end{tikzcd}
\]
whose $\Fscr$-components are the sheaf morphisms $\epsilon_{\Fscr}$ and $S_{\Fscr}$, respectively.
\end{lemma}
\begin{Theorem}\label{Thm: Section Background Scheme: Hopf algebras for RelSym}
Let $X$ be a scheme. Then for any quasi-coherent sheaf $\Fscr$, the quadruple
\[
\left(\RelSym{\Ocal_X}(\Fscr), \nabla_{\Fscr}, \epsilon_{\Fscr}, S_{\Fscr}\right)
\]
is a cocommutative Hopf algebra in $\QCoh(X,\CAlg{\Ocal_X})$ and for any quasi-coherent sheaf map $\varphi:\Fscr \to \Gscr$ $\RelSym{\Ocal_X}(\varphi)$ is a morphism of Hopf algebras. In particular, there is an induced functor
\[
\widehat{\RelSym{\Ocal_X}}:\QCoh(X) \to \mathbf{Hopf}^{\operatorname{cocom}}(\QCoh(X))
\]
which sends quasi-coherent sheaves $\Fscr$ to the corresponding Hopf algebra on $\RelSym{\Ocal_X}(\Fscr)$ above and sends sheaf maps $\varphi$ to $\RelSym{\Ocal_X}(\varphi)$.
\end{Theorem}
\begin{proof}
That each given quadruple is a Hopf algebra follows from (sheafifying) the result that for all opens $U \subseteq \lvert X \rvert$, each quadruple
\[
\left(\Sym{\Ocal_X(U)}\!\big(\Fscr(U)\big), (\nabla_{\Fscr}^{\operatorname{pre}})_U, \epsilon_{\Fscr(U)}, S_{\Fscr(U)}\right)
\]
is a cocommutative Hopf algebra in $\Mod{\Ocal_X(U)}$ by Proposition \ref{Prop: Section Background Alg: Symmetric Alg is Coalg}. Finally that this assignment is functorial is immediate from the naturality of $\nabla, \epsilon,$ and $S$ implied by Lemmas \ref{Lemma: Section Background Scheme: Nabla as a natural comultiplication} and \ref{Lemma: Section Background Scheme: Counit and Unit as natural transforms}.
\end{proof}

In order to apply it in the section below we conclude that the opposite of all functors above takes values in the category of Abelian group objects internal to $\QCoh(X,\CAlg{\Ocal_X})^{\op}$.
\begin{corollary}\label{Cor: Section Background Scheme: Hopf algebra result in opposite land}
For any scheme $X$, the opposite of the relative symmetric algebra functor extends to a functor:
\[
\left(\widehat{\RelSym{\Ocal_X}}\right)^{\op}:\QCoh(X)^{\op} \to \Ab\left(\QCoh(X,\CAlg{\Ocal_X})^{\op}\right)
\]
\end{corollary}
\begin{proof}
This is immediate from applying the formal calculus of taking opposites of the functor appearing in Theorem \ref{Thm: Section Background Scheme: Hopf algebras for RelSym}.
\end{proof}

\subsection{The Relative Spectrum Functor and the Category of Affine Structure Morphisms}\label{Subsection: Relative Spectrum}

As we will see below, one of the more important tools which appears in the presentation of the tangent category $\Sch_{/S}$ lies in the application of the relative spectrum functor $\RelSpec{X}:\QCoh(X,\CAlg{\Ocal_X}) \to \Sch_{/X}$ and the fact that it is one-half of an equivalence of categories with the subcategory $\mathbf{Aff}_{/X}$ of $\Sch_{/X}$ of schemes $q:Y \to X$ for which $q$ is an affine morphism. On one hand, it is possible to first define the functor $\RelSpec{X}$ and then \emph{define} affine structure morphisms by way of said functor. However, this is not practical. It is more practical to have a structural characterization of said maps and then prove that they arise as the structure maps of the essential image of the relative spectrum functor.

In view of taking the second approach, we recall the definition of affine structure morphisms and then what will turn out to be the main structural characterization of said maps. We will then describe the relative spectrum functor and show that our structural characterization is indeed exactly that.
\begin{definition}[{\cite[D{\'e}finition 1.2.1]{EGA2}}]\label{Defn: Scheme Background: Affine Morphism}
Let $S$ be a scheme and let $f:X \to S$ be a scheme over $S$. We say that the structure map $f$ is \emph{affine (over $S$)} if there exists an open affine cover $\lbrace \iota_i:S_i \to S \; | \; i \in I \rbrace$ cover of $S$ for which for each $i \in I$ the subscheme $S_i \times_S X \to X$ is an affine open subscheme of $X$.
\end{definition}

A useful fact about affine morphisms of schemes is that if $f:X \to \Spec A$ is an affine morphism, then $X$ is necessarily an affine scheme as well.
\begin{proposition}\label{Prop: Affines over affines are affine}
	Let $f:X \to \Spec A$ be an affine morphism. Then $X$ is an affine scheme as well.
\end{proposition}
\begin{proof}
	Because $\Spec A$ is affine, the singleton cover $\lbrace\id_A:\Spec A \to \Spec A\rbrace$ is an affine open cover of $\Spec A$. But then since
	\[
	\begin{tikzcd}
		X \ar[r, equals] \ar[d, swap]{}{f} & X\ar[d]{}{f} \\
		\Spec A \ar[r, equals] & \Spec A
	\end{tikzcd}
	\]
	is a pullback, we see that $X \cong \Spec A \times_A X$ is affine as well.
\end{proof}

We now move on to define the category of $S$-schemes with affine structure map $f:X\to S$. Afterwards, we will give an example of a scheme with affine structure map.

\begin{definition}
	Let $S$ be a scheme. The category $\mathbf{Aff}_{/S}$ of \emph{schemes affine over $S$} is the full subcategory of $\Sch_{/S}$ generated by $S$-schemes $f:X \to S$ for which $f$ is affine.
\end{definition}

\begin{example}\label{Example: Section QCoh: Reduction of a scheme}
Let $X$ be a scheme. Define the scheme $X_{\red}$ as follows: the underlying space of $X_{\red}$ is $\lvert X \rvert$, and the structure sheaf of $X_{\red}$ is defined by declaring, for $U \subseteq \lvert X \rvert$ open,
\[
\Ocal_{X_{\red}}(U) := \frac{\Ocal_X(U)}{\mathsf{nil}(\Ocal_X(U))} = \frac{\Ocal_X(U)}{\sqrt{(0)}}.
\]
That is $\Ocal_{X_{\red}}(U)$ is the reduction of the ring $\Ocal_X(U)$, i.e., is the algebra obtained by quotienting at nilpotents. Then the map
\[
\epsilon_X:X_{\red} \to X
\]
given by defining the underlying topological morphism to be given by $\lvert \epsilon_X \rvert := \id_{\lvert X \rvert}$ and taking the sheaf map to be given by locally taking the quotient of $\Ocal_X(U)$ at its nilradical:
\begin{prooftree}
	\AxiomC{$\epsilon^{\sharp}:\Ocal_X \to \Ocal_{X_{\red}}$}
	\UnaryInfC{$\forall\,U\subseteq \lvert X \rvert\,\text{open}.\,\epsilon^{\sharp}_U := \pi_{\mathsf{nil}}:\Ocal_X(U) \to \Ocal_X(U)/\mathsf{nil}(\Ocal_X(U))$}
\end{prooftree}
The $\epsilon_X:X_{\red} \to X$ is an affine morphism of schemes. The scheme $X_{\red}$ is called the \emph{reduction} of $X$.
\end{example}
\begin{remark}\label{Remark: Section QCoh: Reduction functor}
For any fixed base scheme $S$, because the reduction $X_{\red}$ is a subscheme of $X$ via the map $\epsilon_X$, $X_{\red}$ is an $S$-scheme as well. The operation $X \mapsto X_{\red}$ defines a functor
\[
(-)_{\red}:\Sch_{/S} \to \Sch_{/S}
\]
on objects which is idempotent in the sense that $(X_{\red})_{\red} \cong X_{\red}$ for all schemes $X$. When we define the category $\rSch_{/S}$ of reduced $S$-schemes to be the category of $S$-schemes for which the map $\epsilon_X:X_{\red} \to X$ is an isomorphism, we obtain an adjuction
	\[
	\begin{tikzcd}
	\Sch_{/S} \ar[rr, bend right = 30, ""{name = R}, swap]{}{\red} & & \rSch_{/S} \ar[ll, bend right = 30, ""{name = L}, swap]{}{\operatorname{incl}}
	\ar[from = L, to = R, symbol = \dashv]
	\end{tikzcd}
	\]
which recognizes the category $\rSch_{/S}$ of reduced $S$-schemes as a coreflective subcategory of $S$-schemes. This functor is important because when $S = \Spec K$ for a field $K$, the category of $K$-varieties is a full subcategory of $\rSch_{/K}$. In fact, as varieties are reduced separated schemes of finite type over a field $K$, some basic calculations showing that taking the reduction of a separated and finite type\footnote{A morphism $f:X \to S$ of schemes if finite type if, following \cite[D{\'e}finition 6.3.1]{EGA1}, given an affine open cover $\lbrace V_i \; | \; i \in I \rbrace$ of $S$, each pullback $V_i \times_S X$ is covered by finitely many affine open subschemes $U_{ij}$ of $X$ for which $\Ocal_X(U_{ij})$ is a finite type $\Ocal_Y(V_i)$-algebra. We will not be using this seriously in this paper.} morphism $f:X \to \Spec K$ remains separated and finite type\footnote{The fact that if $f$ is separated so too is $f_{\red}$ is \cite[Proposition 5.5.1.vi]{EGA1}, while the fact that if $f$ is finite type so too is $f_{\red}$ is \cite[Proposition 6.3.4.(vi)]{EGA1}.} shows that we can restrict the reduction functor to the category of finite type separated $K$-schemes to get a coreflection:
\[
\begin{tikzcd}
\Sch_{/K}^{\operatorname{sep,f.t.}} \ar[rr, bend right = 20, ""{name = R}, swap]{}{\red} & & \mathbf{Var}_{/K} \ar[ll, bend right = 20, ""{name = L}, swap]{}{\operatorname{incl}}
\ar[from = L, to = R, symbol = \dashv]
\end{tikzcd}
\]
We will see the definition of separated morphisms in Section \ref{Subsection: DBun as invariant} below when we examine the category of quasi-separated schemes.
%
\end{remark}

A potentially surprising, but important, result regarding affine morphisms is that the pullback of \emph{all} affine open covers remains affine. The trick here is to ultimately use the observation that affine morphisms over affine schemes have affine domain.

\begin{proposition}[{cf.\@ \cite[Proposition 1.2.5]{EGA2}}]\label{Prop: Section Background Scheme: Affine iff affine for all covers}
	Let $f:X \to S$ be a morphism of schemes. Then $f$ is affine if and only if for all open affine covers $\lbrace U_i \; | \; i \in I \rbrace$ of $S$, $\lbrace f^{-1}(U_i) \; | \; i  \in I \rbrace$ is an affine open cover of $X$.
\end{proposition}
\begin{proof}
	The proposition in the reference above, \cite[Proposition 1.2.5]{EGA2}, shows that if $f:X \to S$ is affine and if $\gamma:U \to S$ is an affine open then the pullback projection $U \times_S X \to U$ is affine over $U$. Because the only affine morphisms over affine schemes are maps of affine schemes by Proposition \ref{Prop: Affines over affines are affine}, the result follows.
\end{proof}

\begin{lemma}\label{Lemma: Background Scheme: Affine morphisms stable under composition}
	If $f:X \to Y$ and $g:Y \to Z$ are both affine morphisms then so too is $g \circ f$.
\end{lemma}
\begin{proof}
Fix an affine open cover $\lbrace U_i \; | \; i \in I \rbrace$ of $Z$ and observe that since $g$ is affine, the cover $\lbrace g^{-1}(U_i) \; | \; i \in I\rbrace$ is an affine open cover of $Y$ by Proposition \ref{Prop: Section Background Scheme: Affine iff affine for all covers}. But then using that $f$ is affine and appealing again to Proposition \ref{Prop: Section Background Scheme: Affine iff affine for all covers} gives that the open cover
\[
\lbrace f^{-1}(g^{-1}U_i) \; | \; i \in I \rbrace \cong \lbrace (g \circ f)^{-1}(U_i) \; | \; i \in I \rbrace
\]
is an affine open cover of $X$. Thus $g \circ f$ is affine.
\end{proof}

One of the most important properties of affine morphisms is that the sheaf $f_{\ast}(\Ocal_X)$ is a quasi-coherent $S$-algebra. We will use this in giving a characterization of affine $S$-schemes by using it to compare $\mathbf{Aff}_{/S}$ with the category $\QCoh(S,\CAlg{\Ocal_S})$.

\begin{proposition}
	[{\cite[Proposition 1.2.6]{EGA2}}]
Let $f:X \to S$ be an affine morphism of schemes. Then $f_{\ast}(\Ocal_X)$ is a quasi-coherent $\Ocal_S$-algebra.
\end{proposition}
As stated above, we can sharpen the proposition above into give a \emph{functorial} and fully faithful assignment of affine schemes to the category of quasi-coherent sheaves of $\Ocal_S$-algebras provided we recognize that this functor must \emph{reverse} the direction of either the maps of schemes or the maps of quasi-coherent sheaves.
\begin{proposition}[{\cite[Proposition 1.2.7]{EGA2}}]\label{Prop: Section Scheme Background: Affine morphisms in terms of algebra}
For any scheme $S$ and any schemes $f:X \to S$ and $g:Y \to S$ in $\mathbf{Aff}_{/S}$, there is a natural isomorphism
\[
\mathbf{Aff}_{/S}(X,Y) \cong \QCoh\left(S,\CAlg{\Ocal_S}\right)\big(g_{\ast}(\Ocal_Y), f_{\ast}(\Ocal_X)\big)
\]
In particular, up to the Axiom of Choice, the object assignment $X \mapsto f_{\ast}(\Ocal_X)$ induces a functor 
\[
\left((-) \mapsto (\nu_{(-)})_{\ast}\Ocal_{(-)}\right):\mathbf{Aff}_{/S} \to \QCoh(S,\CAlg{\Ocal_S})^{\op}.
\]
Furthermore, this functor is fully faithful.
\end{proposition}
\begin{proof}[Sketch]
Let us describe the assignment of the functor on morphisms. Given a commuting diagram of affine $S$-scheme
\[
\begin{tikzcd}
X \ar[rr]{}{f} \ar[dr, swap]{}{\nu_X} & & Y \ar[dl]{}{\nu_Y} \\
 & S
\end{tikzcd}
\]
we then have an induced diagram of quasi-coherent sheaves of $S$-modules
\[
\begin{tikzcd}
\Ocal_{S} \ar[d, swap]{}{\nu_Y^{\sharp}} \ar[rr, equals]  & & \Ocal_S \ar[d]{}{\nu_X^{\sharp}}  \\
(\nu_Y)_{\ast}\Ocal_Y \ar[dr, swap]{}{(\nu_Y)_{\ast}f^{\sharp}} & & (\nu_X)_{\ast}\Ocal_X \\
 & (\nu_{Y})_{\ast}(f_{\ast}\Ocal_X) \ar[ur, equals]
\end{tikzcd}
\]
by virtue of the fact that the pushforward functors $(\nu_Y)_{\ast} \circ f_{\ast} = (\nu_Y \circ f)_{\ast} = (\nu_X)_{\ast}$ of quasi-coherent sheaves commute strictly. Writing $F := ((-) \mapsto (\nu_{(-)})_{\ast}\Ocal_{(-)})$ for brevity and clarity, we then get
\[
F\left(\begin{tikzcd}
X \ar[rr]{}{f} \ar[dr, swap]{}{\nu_X} & & Y \ar[dl]{}{\nu_Y}  \\
 & Y
\end{tikzcd}\right) := (\nu_Y)_{\ast}f^{\sharp}.
\]
That this is functorial is straightforward to check and that it is fully faithful follows from a straightfoward application of Zariski descent. That is, because $\nu_X:X \to S$ and $\nu_Y:Y \to S$ are affine morphisms, for any affine open cover $\lbrace \gamma_i:U_i \to S \; | \; i \in I \rbrace$ we can pullback along both $\nu_X$ and $\nu_Y$ simultaneously as in the diagram
\[
\begin{tikzcd}
X \ar[r]{}{\nu_X} & S & Y \ar[l, swap]{}{\nu_Y} \\
X \times_S U_i \ar[u]{}{\pi_0^{i,X}} \ar[r, swap]{}{\pi_1^{i,X}} & U_i \ar[u]{}{\gamma_i} & Y \times_S U_i \ar[l]{}{\pi_1^{i,Y}} \ar[u, swap]{}{\pi_0^{i,Y}}
\end{tikzcd}
\]
and derive open affine covers $\lbrace \pi_0^{i,X}:X \times_S U_i \to X \; | \; i \in I \rbrace$ and $\lbrace \pi_0^{i,Y}:Y \times_S U_i \to Y \; | \; i \in I \rbrace$. Writing $U_{ij} := U_i \times_S U_j$ for the intersection of the affine opens of $S$, we have corresponding intersections
\[
X \times_S U_{ij} \cong (X \times_S U_i) \times_X (X \times_S U_j)
\]
and
\[
Y \times_S U_{iJ} \cong (Y \times_S U_i) \times_Y (Y \times_S U_j)
\]
of affine opens of both $X$ and $Y$. Now, by virtue of Theorem \ref{Thm: Section Background Scheme: QCoh is a pseudolimit} and \ref{Thm: Section Background Scheme: Checking QCoh affine locally}, to give a morphism of quasi-coherent sheaves $(\nu_Y)_{\ast}\Ocal_Y \to (\nu_X)_{\ast}\Ocal_X$ it suffices to give morphisms of $\Ocal_S$-algebras
\[
\varphi_{U_i}:\left[(\nu_Y)_{\ast}\Ocal_Y\right](U_i) \longrightarrow \left[(\nu_X)_{\ast}\Ocal_X\right](U_i)
\]
which agree on the pullbacks $U_{ij}$. However, as
\[
\left[(\nu_Y)_{\ast}\Ocal_Y(U_i)\right] = \Ocal_Y(Y \times_S U_i),\quad \left[(\nu_X)_{\ast}\Ocal_X\right](U_i) = \Ocal_X(X\times_S U_i)
\]
and
\[
\left[(\nu_Y)_{\ast}\Ocal_Y\right](U_{ij}) = \Ocal_Y(Y \times_S U_{ij}),\quad \left[(\nu_X)_{\ast}\Ocal_X\right](U_{ij}) = \Ocal_X(X \times_S U_{ij})
\]
for all $i, j \in I$, it follows from the Gluing Lemma and the fact that each of $X \times_S U_i$ and $Y \times_S U_i$ are affine\footnote{The fact that $X \times_S U_i$ and $Y \times_S U_i$ are affine means that maps $X \times_S U_i \to Y \times_S U_i$ are equivalent to giving ring maps $\Ocal_{Y}(Y \times_S U_i) \to \Ocal_X(X \times_S U_i)$.} open covers of $X$ and $Y$, respectively, that to determine such a collection $\varphi_{U_i}$ is equivalent to defining a scheme morphism $f:X \to Y$.
\end{proof}

Proposition \ref{Prop: Section Scheme Background: Affine morphisms in terms of algebra} is a translation of a geometry-to-algebra result in the sense that it tells us that the geometry of affine maps $f:X \to S$ give rise to quasi-coherent sheaves of algebras $f_{\ast}(\Ocal_X)$; our goal at the moment is to show that this is a precise characterization. To do this we use the relative spectrum functor of \cite{EGA2}: this functor is encoded in the following result of Grothendieck.

\begin{proposition}[{\cite[Proposition 1.3.1]{EGA2}}]\label{Prop: Section Scheme Background: Relspec functor}
Let $S$ be a scheme and let $\Ascr$ be a quasi-coherent sheaf of $\Ocal_S$-algebras. Then there is a (unique up to isomorphism of $S$-schemes) scheme $X$ with affine structure map $f:X \to S$ for which $f_{\ast}(\Ocal_X) \cong \Ascr$ in $\QCoh(X,\CAlg{\Ocal_X})$. In particular, up to the Axiom of Choice, there is a functor
\[
\RelSpec{S}:\QCoh(S,\CAlg{\Ocal_S})^{\op} \to \mathbf{Aff}_{/S}
\]
which is fully faithful.
\end{proposition}
\begin{proof}[Sketch]
Let $\Ascr$ be a quasi-coherent sheaf of $\Ocal_S$-algebras. The construction of a scheme and affine structure map $p_{\Ascr}:\RelSpec{S}(\Ascr) \to S$ associated to $\Ascr$ is described in \cite[Proposition 1.3.1]{EGA2}. That the scheme $\underline{\Spec}_{S}(\Ascr)$ and corresponding structure map $p_{\Ascr}$ are produced functorially is also derived from, but not explicitly stated in, the proof of \cite[Proposition 1.3.1]{EGA2}. Finally, that this functor is fully faithful is immediate from the isomorphisms of hom-sets
\[
\mathbf{Aff}_{/S}(X,Y) \cong \QCoh(S,\CAlg{\Ocal_S})(g_{\ast}(\Ocal_Y),f_{\ast}(\Ocal_X)) \cong \QCoh(S,\CAlg{\Ocal_S})^{\op}(f_{\ast}(\Ocal_X),g_{\ast}(\Ocal_Y)).
\]
\end{proof}
\begin{corollary}[{\cite[Proposition 1.4.1]{EGA2}}]\label{Cor: Section Background Scheme: Equiv of Cats for Relspec}
	There is an equivalence of categories
	\[
	\begin{tikzcd}
		\mathbf{Aff}_{/S} \ar[rr, swap, bend right = 30, ""{name = R}]{}{X \mapsto f_{\ast}(\Ocal_X)} & & \QCoh(X, \CAlg{\Ocal_X})^{\op} \ar[ll, bend right = 30, swap, ""{name = L}]{}{\RelSpec{X}}
		\ar[from = L, to = R, symbol = \simeq]
	\end{tikzcd}
	\]
	with quasi-inverses $\RelSpec{X}$ and $X \mapsto f_{\ast}(\Ocal_X)$.
\end{corollary}
\begin{proof}
All that remains to be shown is that both functors are inverse equivalences to each other. However, by the construction in the proof of \cite[Proposition 1.3.1]{EGA2} this is a straightforward argument using Zariski descent along affine opens\footnote{That the composite $\underline{\Spec}_X \circ (X \mapsto (\nu_X)_{\ast}\Ocal_X)$ is naturally isomorphic to the identity follows by doing the same reasoning as performed at the end of the proof of Proposition \ref{Prop: Section Scheme Background: Affine morphisms in terms of algebra} applied to the structure map $\nu_X:X\to S$. That $(X\mapsto (\nu_X)_{\ast}\Ocal_X) \circ \underline{\Spec}_X$ follows from the Gluing Lemma and the fact that the scheme $\underline{\Spec}_X(\Ascr)$ has a structure sheaf with $\Ocal_{\underline{\Spec}_X(\Ascr)}(p_{\Ascr}^{-1}U) \cong \Ascr(U)$ for affine open subschemes $U$ of $S$.}.
\end{proof}
\begin{definition}\label{Defn: Section Scheme Background: Relspec functor}
Let $X$ be a scheme. The \emph{relative spectrum functor over $X$}
\[
\RelSpec{X}:\QCoh(X,\CAlg{\Ocal_X})^{\op} \to \mathbf{Aff}_{/X}
\] 
is the functor constructed in Proposition \ref{Prop: Section Scheme Background: Relspec functor}.
\end{definition}
\begin{example}
With $X_{\red}$ the reduction of $X$ constructed in Example \ref{Example: Section QCoh: Reduction of a scheme}. Then with $\Ocal_{X_{\red}}$ the sheaf of commutative $\Ocal_X$-algebras also constructed in Example \ref{Example: Section QCoh: Reduction of a scheme}, we have an isomorphism of $X$-schemes
\[
X_{\red} \cong \RelSpec{X}(\Ocal_{X_{\red}}).
\]
It is additionally worth noting that $\lvert X_{\red} \rvert \cong \lvert X \rvert$ --- this follows from the fact that for any ring $A$, $\lvert \Spec A_{\red} \rvert \cong \lvert \Spec A \rvert$.
\end{example}

Because affine morphisms have the property that they are separated (cf.\@ \cite[Proposition 1.2.4]{EGA2}) and quasi-compact (which follows from the fact that affine schemes are all quasi-compact as topolgical spaces), they satisfy the hypotheses of Proposition \ref{Prop: Section Background Scheme: Right adjoint for quasicoherent sheaves}. In particular, this shows tha the structure sheaf $f_{\ast}\Ascr$ of a quasi-coherent sheaf of $\Ocal_X$-algebras is a quasi-coherent sheaf of $\Ocal_Y$-algebras by pre-composition (cf.\@ \cite[Proposition 1.2.4]{EGA2}). Thus, since $f_{\ast}$ is exact in this case, we obtain the following lemma.
\begin{lemma}\label{Lemma: Section Background Scheme: Forward functoriality of relspec for affine maps}
Let $f:X \to Y$ be an affine morphism of schemes. Then there is a commuting diagram
\[
\begin{tikzcd}
\mathbf{Aff}_{/X} \ar[rr]{}{f \circ (-)} \ar[d, swap]{}{Q_X} & & \mathbf{Aff}_{/Y} \ar[d]{}{Q_Y} \\
\QCoh(X,\CAlg{\Ocal_X})^{\op} \ar[rr, swap]{}{f_{\ast}} & & \QCoh(Y,\CAlg{\Ocal_Y})^{\op}
\end{tikzcd}
\]
and an invertible $2$-cell:
\[
\begin{tikzcd}
	\QCoh(X,\CAlg{\Ocal_X})^{\op} \ar[d, swap]{}{\RelSpec{X}} \ar[rr, ""{name = U}]{}{f_{\ast}} & & \QCoh(Y,\CAlg{\Ocal_Y})^{\op}	\ar[d]{}{\RelSpec{Y}} \\
	\mathbf{Aff}_{/X} \ar[rr, swap, ""{name = D}]{}{f \circ (-)}  & & \mathbf{Aff}_{/Y} 
	\ar[from = U, to = D, Rightarrow, shorten <= 4pt, shorten >= 4pt]{}{\cong}
\end{tikzcd}
\]
where $Q_X$ and $Q_Y$ are the inverse equivalences of $\RelSpec{X}$ and $\RelSpec{Y}$, respectively, while the functor $f \circ (-)$ takes an affine morphism $q:Z \to X$ and sends it to the corresponding affine morphism $f \circ q:Z \to Y$.
\end{lemma} 
\begin{proof}
Recall that the inverse equivalences $Q_X:\mathbf{Aff}_{/X} \to \QCoh(X,\CAlg{\Ocal_X})^{\op}$ and $Q_Y:\mathbf{Aff}_{/Y} \to \QCoh(Y,\CAlg{\Ocal_Y})^{\op}$ of $\RelSpec{X}$ and $\RelSpec{Y}$, respectively, each take a corresponding affine morphism $q:Z \to S$ (for the cases $S = X$ and $S = Y$, respectively) and send it to the sheaf $q_{\ast}\Ocal_Z$ on $S$; note that this is a quasi-coherent sheaf on $S$ by \cite[Proposition 1.2.6]{EGA2}. 

Now on one hand we compute that for any affine morphism $q:Z \to X$,
\[
Q_{Y}\left(\bigg(f \circ (-)\bigg)(q:Z \to X)\right) = Q_{Y}\left(f \circ q:Z \to Y\right) = (f \circ q)_{\ast}\Ocal_Z.
\]
On the other hand we compute
\[
f_{\ast}\left(Q_X(q:Z \to X)\right) = f_{\ast}\left(q_{\ast}\Ocal_Z\right) = (f_{\ast} \circ q_{\ast})\Ocal_Z.
\]
Because pushforwards of sheaves on topological spaces commute strictly\footnote{Given maps $\alpha:A \to B$ and $\beta:B \to C$ of spaces, for any open $U \subseteq C$, $(\beta \circ \alpha)^{-1}(U) = \alpha^{-1}(\beta^{-1}(U))$ which implies that $(\beta \circ \alpha)_{\ast}\Fscr = (\beta_{\ast} \circ \alpha_{\ast})\Fscr$ for any sheaf $\Fscr$ on $A$.} we thus have that
\[
Q_{Y} \circ \bigg(f \circ (-)\bigg) = f_{\ast} \circ Q_{X}
\]
as functors. Taking the mates of the corresponding strictly commuting diagram, get the desired invertible $2$-cell:
\[
\begin{tikzcd}
	\QCoh(X,\CAlg{\Ocal_X})^{\op} \ar[d, swap]{}{\RelSpec{X}} \ar[rr, ""{name = U}]{}{f_{\ast}} & & \QCoh(Y,\CAlg{\Ocal_Y})^{\op}	\ar[d]{}{\RelSpec{Y}} \\
	\mathbf{Aff}_{/X} \ar[rr, swap, ""{name = D}]{}{f \circ (-)}  & & \mathbf{Aff}_{/Y} 
	\ar[from = U, to = D, Rightarrow, shorten <= 4pt, shorten >= 4pt]{}{\cong}
\end{tikzcd}
\]
\end{proof}

Another important but immediate lemma is that the pullbacks in $\mathbf{Aff}_{/S}$ conicide with those in $\Sch_{/S}$. In particular, the inclusion functor $\mathbf{Aff}_{/S} \to \Sch_{/S}$ is continuous.
\begin{proposition}\label{Prop: Section Background Scheme: Incl of Aff is cont}
For any scheme $S$ the pullback in $\mathbf{Aff}_{/S}$ coincides with the pullback in $\Sch_{/S}$. In particular, the inclusion functor $\mathbf{Aff}_{/S} \to \Sch_{/S}$ is continuous.
\end{proposition}
\begin{proof}[Sketch]
The fact that for any span $\Ascr\leftarrow \Rscr \rightarrow \Bscr$ of quasi-coherent sheaves of $\Ocal_S$-algebras we have
\[
\RelSpec{S}\left(\Ascr \otimes_{\Rscr} \Bscr\right) \cong \RelSpec{S}(\Ascr) \times_{\RelSpec{S}(\Rscr)} \RelSpec{S}(\Bscr)
\]
in $\mathbf{Aff}_{/S}$ is a formal consequence of the fact that $\RelSpec{S}$ is an (opposite) equivalence of categories and that the tensor product $(-) \otimes_{\Rscr}(-)$ is a pushout in $\QCoh(S,\CAlg{\Ocal_S})$. For the fact that the pullback $\RelSpec{S}(\Ascr) \times_{\RelSpec{S}(\Rscr)}\RelSpec{S}(\Bscr)$ computes the pullback in $\Sch_{/S}$ of each scheme affine over $S$, we defer to \cite[Proposition 1.4.6]{EGA2}. 
\end{proof}
\begin{corollary}[{\cite[Corollaire 1.5.2]{EGA2}}]\label{Cor: Section Background Scheme: The pullback iso for relspec}
Let $f:X \to Y$ be a morphism of schemes. Then there is an invertible $2$-cell
\[
\begin{tikzcd}
\QCoh\left(Y,\CAlg{\Ocal_Y}\right)^{\op} \ar[rr, ""{name = U}]{}{f^{\ast}} \ar[d, swap]{}{\RelSpec{Y}} & &\QCoh\left(X,\CAlg{\Ocal_X}\right)^{\op} \ar[d]{}{\RelSpec{X}} \\
\mathbf{Aff}_{/Y} \ar[rr, swap, ""{name = D}]{}{f^{\ast}} & & \mathbf{Aff}_{/X}
\ar[from = U, to = D, Rightarrow, shorten <= 4pt, shorten >=4pt]{}{\cong}
\end{tikzcd}
\]
where the functor $f^{\ast}:\mathbf{Aff}_{/Y} \to \mathbf{Aff}_{/X}$ is given by chosen pullbacks $f^{\ast}(Z) := Z \times_Y X$.
\end{corollary}
\begin{proof}
The isomorphism of the diagram is an immediate consequence of Proposition \ref{Prop: Section Background Scheme: Incl of Aff is cont} once we know that $f^{\ast}:\mathbf{Aff}_{/Y} \to \mathbf{Aff}_{/X}$ is defined. However, \cite[Proposition 1.5.1]{EGA2} shows that the pullback of an affine morphism is affine.
\end{proof}

As an immediate application of the equivalence of categories $\RelSpec{X}$ provides, we can see that for any quasi-coherent sheaf $\Fscr$ on $S$, the scheme
\[
\Vbb(\Fscr) := \underline{\Spec}_X\left(\Sym{\Ocal_X}(\Fscr)\right)
\] 
is an internal Abelian group to $\mathbf{Aff}_{/S}$ and hence also an Abelian group internal to $\Sch_{/S}$.
\begin{corollary}\label{Cor: Section Background Scheme: Relative Spec of Relative Sym is functor into Ab Sch}
For any scheme $S$ there is an induced functor:
\[
\begin{tikzcd}
\QCoh(S)^{\op} \ar[rr]{}{\left(\widehat{\RelSym{\Ocal_S}}\right)^{\op}} & & \Ab\left(\QCoh(S,\CAlg{\Ocal_S})^{\op}\right) \ar[rr]{}{\RelSpec{S}} \ar[rr, swap]{}{\simeq} & & \Ab\left(\mathbf{Aff}_{/S}\right) \ar[r] & \Ab(\Sch_{/S})
\end{tikzcd}
\]
\end{corollary}
\begin{proof}
The first functor is simply the functor of Corollary \ref{Cor: Section Background Scheme: Hopf algebra result in opposite land}. The fact that $\RelSpec{S}$ translates to the equivalence of categories between categories of internal Abelian groups is immediate from the fact that equivalences of categories lift to equivalences between models of algebraic theories. The presence of the final functor into $\Ab(\Sch_{/S})$ is precisely the statement of Proposition \ref{Prop: Section Background Scheme: Incl of Aff is cont}.
\end{proof}

\section{A Review of K{\"a}hler Differentials for Rings and Schemes}\label{Section: Diffles}

Derivations, and the corresponding module of K{\"a}hler differentials, are some of the most fundamental tools in differential algebra and differential algebraic geometry. Because the theory of derivations both provides a groundwork for studying the differential geometric aspects of algebraic geometry and commutative algebra, it is quite important for us to have a firm grasp on this theory so that we may see how it controls and reflects the actual tangential geometry that schemes carry; it is additionally important to see how the differential algebra schemifies to differential algebraic geometry and the ways in which we can see differential algebraic geometry as glued-together affine-local differential algebra. Consequently, in this section we both review the basics of differential algebra and the theory of derivations for commutative ring as well as how to extend these notions to schemes.

\subsection{The Theory of Derivations for Rings and The Module of K{\"a}hler Differentials}\label{Subsection:Derivations for rings}

In this section we record some basic but important results regarding the module of relative K{\"a}hler differentials for a commutative algebra $R \to A$ as well its extension to the sheaf-theoretic and scheme-theoretic settings. To this end let us first recall the module $\Kah{A}{R}$  for a commutative $R$-algebra is. While we generally assume basic familiarity with derivations, let us recall at least the definitions inasmuch as we need them.

\begin{definition}[{\cite[Definition 4.1]{RickRobinRobertTim}}]
	Let $R$ be a cring, let $A$ be a commutative $R$-algebra, and let $M$ be an $A$-module. \emph{An $R$-derivation} from $A$ to $M$, $\partial:A \to M$, is an $R$-linear morphism for which the equation
	\[
	\partial(ab) = a\partial(b) + \partial(a)b
	\]
	holds for all $a, b \in A$. We write $\mathsf{Der}_R(A,M)$ for the set of all $R$-derivations from $A$ to the $A$-module $M$.
\end{definition}
\begin{remark}
	The identity
	\[
	\partial(ab) = a\partial(b) + \partial(a)b
	\]
	is often called the \emph{Leibniz Rule}. We will make use of said nomenclature in this paper.
\end{remark}
\begin{remark}
If one is interested in extending K{\"a}hler differentials to rigs, one is obliged to add the equation
\[
\partial(1) = 0
\]
to the Leibniz rule. While this is implied for rings, in rigs one simply has
\[
\partial(1) = \partial(1\cdot 1) = \partial(1) + \partial(1)
\]
which simply says that $\partial(1)$ is an additive idempotent when one does not have negatives.
\end{remark}
We now give an abstract definition of the module of relative K{\"a}hler differentials.
\begin{definition}\label{Defn: Section Diff: Kahler Diffs}
	Let $R$ be a cring and let $A$ be a commutative $R$-algebra. The \emph{$A$-module of relative K{\"a}hler differentials}, $\Kah{A}{R}$, is the $A$-module which corepresents the functor
	\[
	\mathsf{Der}_R(A,-):\Mod{A} \to \Mod{A}.
	\]
	That is, $\Kah{A}{R}$ is defined by natural isomorphisms
	\[
	\Mod{A}(\Kah{A}{R},M) \cong \mathsf{Der}_R(A,M)
	\]
	for all $A$-modules $M$.
\end{definition}
\begin{remark}
	A more familiar way to phrase the universal property defining $\Kah{A}{R}$ is to say first that there is an $R$-derivation $\mathrm{d}:A \to \Kah{A}{R}$ such that that for any $A$-module $M$ and any $R$-derivation $\partial:A \to M$, there exists a unique $A$-module morphism $\overline{\partial}:\Kah{A}{R} \to M$ making
	\[
	\begin{tikzcd}
		\Kah{A}{R} \ar[r, dashed]{}{\exists!\overline{\partial}} & M \\
		A \ar[u]{}{\mathrm{d}} \ar[ur, swap]{}{\partial}
	\end{tikzcd}
	\]
	commute. This leads to the familiar description of $\Kah{A}{R}$ in terms of the free $A$-module generated by symbols $\mathrm{d}a$ subject to the rules
	\[
	\begin{cases}
		\mathrm{d}(a+b) = \mathrm{d}(a) + \mathrm{d}(b) & \text{for all}\, a, b \in A; \\
		\mathrm{d}(ab) = \mathrm{d}(a)b + a\mathrm{d}(b) & \text{for all}\, a,b \in A.
	\end{cases}
	\]
\end{remark}

Given a morphism $f:A \to B$ of commutative $R$-algebras, we can induce an $A$-linear map  of the form $\Kah{A}{R} \to \res_f(\Kah{B}{R})$ as follows. For each $a \in A$, we define
\[
\mathrm{d}a \mapsto \mathrm{d}\bigg(f(a)\bigg)
\]
and extend $A$-linearly. This induces the map $\mathrm{d}f:\Kah{A}{R} \to \res_f(\Kah{B}{R})$ which lives entirely within $\Mod{A}$ and allows us to write the K{\"a}hler module construction as  functor from commutative $R$-algebras to the fibration of modules. There is an analogous  construction for semirings performed in \cite{RobinMePartialMapRigs},  but it is more involved and delicate.

\begin{proposition}\label{Prop: Section Kahlers: Kahler diffles are functors}
	The K{\"a}hler differentials induce  a functor
	\[
	\Kah{(-)}{R}:\CAlg{R} \to \MMod
	\]
	given on objects by $A \mapsto (A,\Kah{A}{R})$ and on morphisms $f:A \to B$ of $R$-algebras via
	\[
	\Kah{(-)}{R}(f) := (f,\mathrm{d}f):\left(A,\Kah{A}{R}\right) \to \left(B,\Kah{B}{R}\right).
	\]
\end{proposition}

The major benefit of taking this perspective is that it allows us to carefully and precisely phrase the relative cotangent sequence and then extend it to the sheaf-theoretic setting. 
\begin{definition}[Relative Cotangent Sequence; cf.\@ {\cite[Th{\'e}or{\`e}me 20.5.7.i]{EGA04}, \cite[Proposition 16.2]{Eisenbud}}]
	For any  any composable pair of morphisms of commutative rings $f:R \to A$ and $g:A \to B$ there is an exact sequence of $B$-modules:
	\[
	\begin{tikzcd}
		\Kah{A}{R} \otimes_A B \ar[r]{}{u_{f,g}} & \Kah{B}{R} \ar[r]{}{v_{f,g}} & \Kah{B}{A} \ar[r] & 0
	\end{tikzcd}
	\]
\end{definition}
\begin{remark}\label{Remark: Section Kahlers: Cotangent sequence}
	The map $u_{f,g}$ is defined on pure tensors by
	\[
	\mathrm{d}(a) \otimes b \mapsto b\mathrm{d}\bigg(f(a)\bigg).
	\]
	This is precisely the adjoint transpose of the $A$-module map $\mathrm{d}f$:
	\begin{prooftree}
		\AxiomC{$\mathrm{d}f:\Kah{A}{R} \to \res_{f}(\Kah{B}{R})$ in $\Mod{A}$}
		\UnaryInfC{$u_{f,g} = (\mathrm{d}f)^{\sharp}:\Kah{A}{R} \otimes_A B \to \Kah{B}{R}$ in $\Mod{B}$}
	\end{prooftree}
\end{remark}

In order to discuss how to extend the relative K{\"a}hler differentials $\Kah{A}{R}$ of a ring map $\nu:R \to A$ to the relative K{\"a}hler differentials $\Kah{X}{S}$ of a scheme map $f:X \to S$, we first require a sheaf of relative differentials. For the purpose of getting to the tangent structure on $\Sch_{/S}$, we will also need to know how to work with sheaves of relative differentials and the fact that said sheaf is quasi-coherent over $X$.

\subsection{Extending K{\"a}hler Differentials to Schemes}\label{Subsection: Scheme Diffles}

As detailed in \cite[Section 16.3]{EGA44} and \cite[Chapter II.8]{Hartshorne}, we can extend the module of relative K{\"a}hler differentials to a \emph{quasi-coherent sheaf} of relative K{\"a}hler differentials. While \cite[Section 16.3]{EGA44} defines the sheaf $\Kah{X}{S}$ in terms of the graded filtered algebra associated to the diagonal $\Delta_f:X \to X \times_S X$ of the structure map $f:X \to S$, we do not take that perspective here. Instead we apply the various machinery of K{\"a}hler differentials in order to translate from modules to quasicoherent sheaves.

To construct the sheaf $\Kah{X}{S}$ we use the theory of scheme derivations and a functorial/representable functor characterization of K{\"a}hler differentials. For this we require a short review of said derivations before giving a proof of the sheaf $\Omega_{X/S}$ in terms of its universal property as a (co)representing object for the functor of derivations.

\begin{definition}[{\cite[Section 16.5.1]{EGA44}}]
Let $f:X \to S$ be a map of schemes and let $\Fscr$ be a quasi-coherent sheaf on $X$. Regard $\Ocal_X$ and $\Fscr$ as $f^{-1}\Ocal_S$-modules via the sheaf of rings morphism $f^{\flat}:f^{-1}\Ocal_S \to \Ocal_X$. An \emph{$f$-derivation on $\Ocal_X$ with coefficients in $\Fscr$} (alternatively \emph{an $S$-derivation on $\Ocal_X$ with coefficients in $\Fscr$} if the map $f$ is clear from context) is a morphism
\[
\partial \in \Mod{\left(f^{-1}\Ocal_S\right)}\!\left(\Ocal_X,\Fscr\right)
\]
such that:
\begin{itemize}
	\item for all opens $U \subseteq \lvert X \rvert$ the $U$-component of $\partial$ is a derivation
	\[
	\partial_U \in \mathsf{Der}_{(f^{-1}\Ocal_S)(U)}\left(\Ocal_X(U),\Fscr(U)\right)
	\]
	\item Assume we have opens inclusions $U \subseteq \lvert X \rvert$ and $V \subseteq \lvert S \rvert$ open for which there is a commuting diagram of the form:
	\[
	\begin{tikzcd}
	U \ar[ddr, swap, bend right = 20, dashed]{}{\exists\,j_{UV}} \ar[drr, bend left = 20]{}{j_U} \ar[dr, dashed]{}{\exists!} \\
	 & V \times_S X \ar[r]{}{\pi_1} \ar[d, swap]{}{\pi_0} & X \ar[d]{}{f} \\
	 & V \ar[r, swap]{}{j_V} & S
	\end{tikzcd}
	\]
	That is,  $f(U) \subseteq V$. Additionally, write $\alpha_{UV}:\Ocal_S(V) \to (f^{-1}\Ocal_S)(U)$ for the induced morphism. We then ask that for all $t \in \Ocal_X(U)$ and for all $s \in \Ocal_S(V)$,
	\[
	\partial_U\big(\alpha_{UV}(s)t\big) = \alpha_{UV}(s)\partial_U(t).
	\]
\end{itemize} 
\end{definition}
%

As in the usual ring-theoretic setting, we can rephrase giving $S$-derivations in terms of giving sections of an infinitesimal extension of $\Ocal_X$ by the quasi-coherent sheaf. That is, if $\Ocal_X\ltimes \Fscr$ denotes the sheaf of rings for which
\[
\left(\Ocal_X\ltimes \Fscr\right)(U) := \Ocal_X(U) \ltimes \Fscr(U) = \left((s,f) \; | \; s \in \Ocal_X(U), f \in \Fscr(U)\right)
\]
with the ``ring of dual numbers'' multiplication $(s,f)(x,g) = (sx, sg + xf)$ then the first projection $\pr_0:\Ocal_X \ltimes \Fscr \to \Ocal_X$ is a sheaf map. Moreover, both sheaves $\Ocal_X$ and $\Ocal_X\ltimes \Fscr$ are $f^{-1}\Ocal_S$-algebras. Running the usual argument  gives the following lemma.

\begin{lemma}[{\cite[Section 16.5.1]{EGA44}}]\label{Lemma: Section Background Scheme: S derivations and sections}
Let $f:X \to S$ be a morphism of schemes and let $\Fscr$ be a quasi-coherent sheaf on $X$. Then to give an $f$-derivation $\partial:\Ocal_X \to \Fscr$ is equivalent to giving a section to the projection $\pr_0:\Ocal_X \ltimes \Fscr \to \Ocal_X$ in the category $\CAlg{f^{-1}\Ocal_S}$.
\end{lemma}

A straightforward calculation shows that for any scheme map $f:X \to S$ and any quasi-coherent sheaf $\Fscr$, the set $\mathsf{Der}_S(\Ocal,\Fscr)$ is an $\Ocal_X(\lvert X\rvert)$-module. If there is a morphism $\varphi:\Fscr \to \Gscr$ of quasi-coherent sheaves, post-composition by $\varphi$ gives an induced map
\[
\varphi_{\circ}:\mathsf{Der}_S(\Ocal_X,\Fscr) \to \mathsf{Der}_S(\Ocal_X,\Gscr)
\]
and shows that, as in the ring-theoretic case, there is a functor
\[
\mathsf{Der}_S(\Ocal_X,-):\QCoh(X) \to \Mod{\Ocal_X(\lvert X\rvert)}.
\]
We will prove that this is corepresentable below by \emph{constructing} the sheaf $\Kah{X}{S}$ as the corepresenting object of the derivation functor. We first show that we can sheafify the K{\"a}hlers first and then conclude that this sheaf corepresents the derivations. Note that our construction of $\Kah{X}{S}$ is distinct from what is conventionally found in the algebraic geometry literature: there (such as in \cite{EGA44}, \cite{Hartshorne}, \cite{stacks-project}) the sheaf $\Omega_{X/S}$ on $X$ is defined as the pullback against the diagonal $\Delta:X \to X \times_S X$ of a certain quasi-coherent $\Iscr/\Iscr^2$ where $\Iscr$ is a sheaf of ideals on $X \times_S X$ locally generated by the kernel of the multiplication map $B \otimes_A B \to B$ over affine patches\footnote{The yoga here is twofold. First, the pullback $X \times_S X$ is generated by gluing open affines pullbacks $U_i \times_{V_j} U_k$ where $U_i, U_k$ are open affines of $X$ and where $V_j$ is an open affine of $S$. Second, given a ring map $A \to B$, the diagonal map $\Spec B \to \Spec B \times_A \Spec B$ is the spectrum of the multiplication map $\mu:B \otimes_A B \to B$. Putting these together, we find that we can generate a quasi-coherent sheaf $\Iscr$ by ensuring that the values of $\Iscr(U_i  \times_{V_j} U_i)$ are isomorphic to $\operatorname{Ker}(\Gamma(U_i) \otimes_{\Gamma(V_j)} \Gamma(U_i) \to \Gamma(U_i)$ and then consider its quotient sheaf $\Iscr/\Iscr^2$. This ideal is well-known locally to be related to $\Kah{\Gamma(U_i)}{\Gamma(V_j)}$ and pulling this sheaf back to the original scheme gives us our K{\"a}hler differentials.}. However, there is an arguably more conceptually straightforward structural argument we can make which only involves the existence of the K{\"a}hlers on each open subscheme of $X$. In fact, the argument we make can conjecturally be massaged to work at the level of K{\"a}hler categories which need not invoke kernels nor negatives in its construction. Because of the author's interest in rig-theoretic algebraic geometry, we have chosen to take this additive inverse/kernel agnostic\footnote{In the sense that it is indifferent to the existence or lack thereof of both additive inverses or kernels.} approach towards building $\Kah{X}{S}$. The cost, however, is a requirement to build this from the ground up.

In order to construct the sheaf $\Kah{X}{S}$ and show that it is quasi-coherent over $X$, we proceed in three steps:
\begin{enumerate}
	\item First we show $\Kah{X}{S}$ is a presheaf on $X$.
	\item Second we show that for a map of affine schemes $\Spec f:\Spec B \to \Spec A$, there is an isomorphism of presheaves $\Kah{\Spec B/\Spec A} \cong \widetilde{\Kah{B}{A}}$.
	\item Third we show that $\Kah{X}{S}$ is a quasi-coherent sheaf whose affine open restrictions are those constructed in Step $(2)$.
\end{enumerate}
By virtue of combining these results, we will be able to deduce that $\Kah{X}{S}$ is a quasi-coherent sheaf \emph{and} that $\Kah{X}{S}$ corepresents the functor $\mathsf{Der}_S(\Ocal_X,-)$.

\begin{proposition}\label{Prop: Section Kahlers: Presheaf of relative Kahlers}
Let $f:X \to S$ be a morphism of schemes. Then there is a presheaf $\Kah{X}{S}$ of $\Ocal_X$-modules on $X$ with the property that for all affine opens $U$ of $X$,
\[
\left(\Kah{X}{S}\right)(U) = \Kah{\Ocal_X(U)}{(f^{-1}\Ocal_S)(U)}.
\]
\end{proposition}
\begin{proof}
Given the statement of the proposition we must first construct the restriction maps
\[
\left(\Kah{X}{S}\right)(U) \to \left(\Kah{X}{S}\right)(V)
\]
for $U \supseteq V$ open; second prove that these assemble to a sheaf on $X$; and, third, prove that the given sheaf is quasi-coherent, i.e., that $\Omega_{X/S}$ is given by gluing affine locally.

In the first case write $\rho_{UV}:=\Ocal_X(U \supseteq V):\Ocal_X(U) \to \Ocal_X(V)$ for the restriction morphism of $\Ocal_X$. By Proposition \ref{Prop: Section Kahlers: Kahler diffles are functors} we have that $\mathrm{d}\rho_{UV}$ is a map
\[
\mathrm{d}\left(\rho_{UV}\right):\Kah{\Ocal_X(U)}{(f^{-1}\Ocal_S)(U)} \to \Kah{\Ocal_X(V)}{(f^{-1}\Ocal_S)(U)}
\]
in $\Mod{\Ocal_X(U)}$; note that we have omitted the $\res_{\rho_{UV}}$ for the sake of readability and to streamline some arguments. Now as these module maps restrict, via the ring maps from $(f^{-1}\Ocal_S)(U)$, to maps of $(f^{-1}\Ocal_S)(U)$-modules. Thus we obtain via an extension of scalars a module map
\[
\Kah{\Ocal_X(U)}{(f^{-1}\Ocal_S)(U)} \underset{{(f^{-1}\Ocal_S)(U)}}{\otimes} \left(f^{-1}\Ocal_S\right)(V) \xrightarrow{\mathrm{d}(\rho_{UV}) \otimes \id} \Kah{\Ocal_X(V)}{(f^{-1}\Ocal_S)(U)} \underset{(f^{-1}\Ocal_S)(U)}{\otimes} \left(f^{-1}\Ocal_S\right)(V).
\]
On one hand, from basic results on K{\"a}hler differentials (cf.\@, for instance, \cite[Proposition 16.4]{Eisenbud}) we have a natural isomorphism
\[
\Kah{\Ocal_X(V)}{(f^{-1}\Ocal_S)(U)} \underset{(f^{-1}\Ocal_S)(U)}{\otimes} \left(f^{-1}\Ocal_S\right)(V) \cong \Kah{(\Ocal_X(V) \otimes_{(f^{-1}\Ocal_S)(U)} \left(f^{-1}\Ocal_S\right)(V))}{(f^{-1}\Ocal_S)(V)}.
\]
Using the fact that colimits commute, that $f^{-1}(\Ocal_S)$ is computed via sheafification (which is a colimit) of a filtered colimit, and the fact that $V \subseteq U$ we get a natural isomorphism
\[
\Ocal_X(V) \otimes_{(f^{-1}\Ocal_S)(U)} \left(f^{-1}\Ocal_S\right)(V) \cong \Ocal_X(V) \otimes_{(f^{-1}\Ocal_S)(V)} f^{-1}\left(\Ocal_X(V)\right) \cong \Ocal_X(V)
\]
and hence natural isomorphisms, given via composition,
\[
\Kah{(\Ocal_X(V) \otimes_{(f^{-1}\Ocal_S)(U)} \left(f^{-1}\Ocal_S\right)(V))}{(f^{-1}\Ocal_S)(V)} \cong \Kah{\Ocal_X(V)}{(f^{-1}\Ocal_S)(V)}
\]
and
\[
\Kah{\Ocal_X(V)}{(f^{-1}\Ocal_S)(U)} \underset{(f^{-1}\Ocal_S)(U)}{\otimes} \left(f^{-1}\Ocal_S\right)(V) \cong \Kah{\Ocal_X(V)}{(f^{-1}\Ocal_S)(V)}.
\]
By pre-composing with the tensor product inclusion
\[
\Kah{\Ocal_X(U)}{(f^{-1}\Ocal_S)(U)} \xrightarrow{\iota_0} \Kah{\Ocal_X(U)}{(f^{-1}\Ocal_S)(U)} \underset{{(f^{-1}\Ocal_S)(U)}}{\otimes} \left(f^{-1}\Ocal_S\right)(V)
\]
we obtain the restriction morphism $(\Kah{X}{S})(U \supseteq V)$ defined below:
\[
\begin{tikzcd}
\Kah{\Ocal_X(U)}{(f^{-1}\Ocal_S)(U)} \ar[rrr]{}{\iota_0} \ar[d, swap]{}{(\Kah{X}{S})(U \supseteq V)} & & &  \Kah{\Ocal_X(U)}{(f^{-1}\Ocal_S)(U)} \underset{{(f^{-1}\Ocal_S)(U)}}{\otimes} \left(f^{-1}\Ocal_S\right)(V) \ar[d]{}{\mathrm{d}(\rho_{UV}) \otimes \id} \\
\Kah{\Ocal_X(V)}{(f^{-1}\Ocal_S)(V)} & & & \Kah{\Ocal_X(V)}{(f^{-1}\Ocal_S)(U)} ]\underset{(f^{-1}\Ocal_S)(U)}{\otimes} \left(f^{-1}\Ocal_S\right)(V) \ar[lll]{}{\cong}
\end{tikzcd}
\]
Checking the functoriality of these maps is straightforward but tedious and omitted. Thus $\Kah{X}{S}$ is a presheaf of $\Ocal_X$-modules on $X$. 
\end{proof}

In order to prove that $\Kah{X}{S}$ is a quasi-coherent sheaf, we first will show that for affine scheme morphisms $\Spec B \to \Spec A$, the presheaf $\Kah{\Spec B}{\Spec A}$ is quasi-coherent by exhibiting an isomorphism $\Kah{X}{S} \cong \tilde{\Kah{B}{A}}$. We will need this when we prove that for general scheme maps $X \to S$ the sheaf $\Kah{X}{S}$ is quasi-coherent, as we will show that on affine covers $\Kah{X}{S}$ restricts to the sheaves $\tilde{\Kah{B_i}{A_j}}$.

\begin{lemma}\label{Lemma: Section Kahlers: Relative Kahlers are qcoh over Affine Schemes}
For a map $f:\Spec B \to \Spec A$ of affine schemes, the presheaf $\Kah{\Spec B}{\Spec A}$ is a quasi-coherent sheaf on $\Spec B$ with
\[
\Kah{\Spec B}{\Spec A} \cong \widetilde{\Kah{B}{A}}.
\]
\end{lemma}
\begin{proof}
Let $\varphi:A \to B$ be the map corresponding to $f = \Spec \varphi$. Begin by recalling that if $b \in B$ with $D(b) := \lbrace \pfrak \in \lvert \Spec B \rvert \; : \; b \notin \pfrak \rbrace$ the set $\lbrace D(b) \; | \; b \in B \rbrace$ is an open basis for $\Spec B$. Moreover, if we write
\[
\Sigma_b := \varphi^{-1}(b)  = \lbrace a \in A \; | \; \varphi(a) = b \rbrace \subseteq A
\]
then $D(\Sigma_b) = \cup_{a \in \Sigma_b} D(a)$ is also an open of $\lvert \Spec A \rvert$; it in fact is an open for which we find $f(D(b)) = \Sigma_b$. Furthermore, by construction we have a chain of natural isomorphisms
\[
A[\Sigma_{b}^{-1}] \cong \lim_{\substack{\longrightarrow \\ a \in A.\, \varphi(a) = b}}\left(A[a^{-1}]\right) \cong \lim_{\substack{\longrightarrow \\ D(a).\,D(b) \subseteq f^{-1}(D(a))}}\left( \Ocal_A(D(a))\right) = \left(f^{-1}\Ocal_A\right)\big(D(b)\big)
\]
so we have a natural isomorphism
\[
A[\Sigma_{b}^{-1}] \cong \left(f^{-1}\Ocal_A\right)\big(D(b)\big).
\]

With the natural isomorphism involving $A[\Sigma_b^{-1}] \cong (f^{-1}\Ocal_A)(D(b))$ in mind, observe from the fact that localizations are stable under base change by construction we get that the diagram
\[
\begin{tikzcd}
A \ar[r]{}{} \ar[d, swap]{}{\varphi} & A[\Sigma_b^{-1}] \ar[d]{}{\varphi[\Sigma_b^{-1}]} \\
B \ar[r, swap]{}{} & B[b^{-1}]
\end{tikzcd}
\]
is a pushout in $\CAlg{A}$. Using that the module of K{\"a}hler differentials is stable under base change by \cite[Proposition 16.4]{Eisenbud} and stable under localizations by \cite[Proposition 16.9]{Eisenbud}, we derive the following chain of natural isomorphisms:
\begin{align*}
\left(\Kah{\Spec B}{\Spec A}\right)\big(D(b)\big) &= \Kah{B[b^{-1}]}{A[\Sigma_b^{-1}]} \cong \Kah{\left(B \otimes_A A[\Sigma_b^{-1}]\right)}{A[\Sigma_b^{-1}]} \cong \Kah{B}{A} \otimes_A A[\Sigma_b^{-1}] \cong \Kah{B}{A} \otimes_B B[b^{-1}] \\
&= \left(\Kah{B}{A}\right)[b^{-1}] = \left(\widetilde{\Kah{B}{A}}\right)\big(D(b)\big).
\end{align*}
As such, we have natural isomorphisms
\[
\left(\Kah{\Spec B}{\Spec A}\right)\big(D(b)\big) \cong  \left(\Kah{B}{A}\right)[b^{-1}] = \left(\widetilde{\Kah{B}{A}}\right)\big(D(b)\big)
\]
for all $b \in B$, i.e., for all basic opens $D(b)$ of $\lvert \Spec B \rvert$. Because topological bases induced Grothendieck pretopolgies and the fact that $\widetilde{\Kah{B}{A}}$ is a Zariski sheaf on $\Spec B$, by passing through the natural isomorphisms above we deduce that $\Kah{\Spec B}{\Spec A}$ is a Zariski-sheaf as well. Finally, because $\Kah{\Spec B}{\Spec A}$ is isomorphic to $\widetilde{\Kah{B}{A}}$ on a pretopology for the open Grothendieck topology on $\Spec B$, it follows that $\Kah{\Spec B}{\Spec A} \cong \widetilde{\Kah{B}{A}}$ in $\Mod{\Ocal_B}$. Thus $\Kah{\Spec B}{\Spec A}$ is a quasi-coherent sheaf on $\Spec B$.
\end{proof}

\begin{proposition}\label{Prop: Section Kahlers: Relative Kahlers are qcoh on X}
Let $f:X \to S$ be a map of schemes. Then the presheaf of relative K{\"a}hler differentials is a quasi-coherent sheaf on $X$.
\end{proposition}
\begin{proof}
Because we know that $\Kah{X}{S}$ is a presheaf on $X$ by Proposition \ref{Prop: Section Kahlers: Presheaf of relative Kahlers}, it suffices to prove that $\Kah{X}{S}$ is a quasi-coherent sheaf. Give $S$ an affine open cover $\lbrace S_i \; | \; i \in I \rbrace$, pull this cover back to an open cover $\lbrace f^{-1}(S_i) \; | \; i \in I \rbrace$ of $X$ (we can without loss of generality use the model of the pullback where $f^{-1}(S_i) = S_i \times_S X$ has underlying topological space the preimage of $S_i$ under $f$), and for each $i \in I$ give $f^{-1}(S_i)$ an affine open cover $\lbrace U_{ij} \; | \; j \in J_i \rbrace$. Then $\lbrace U_{ij} \; | \; i \in I, j \in J_i \rbrace$ is an affine open cover of $X$ and by construction $f(U_{ij})  \subseteq S_i$ for all $i \in I$ and all $j \in J_i$, i.e., the diagram
\[
\begin{tikzcd}
U_{ij} \ar[r]{}{\gamma_{ij}} \ar[d, swap]{}{f} & X \ar[d]{}{f} \\
S_i \ar[r, swap]{}{\gamma_i} & S
\end{tikzcd}
\]
commutes in $\Sch$. By construction of this cover and an application of the Gluing Lemma, for all $i \in I$ and for all $j \in J_i$, if $\gamma_{ij}:U_{ij} \to X$ is the open immersion fitting into the commuting diagram
\[
\begin{tikzcd}
U_{ij} \ar[r]{}{\gamma_{ij}} \ar[d, swap]{}{f} & X \ar[d]{}{f} \\
S_i \ar[r, swap] & S
\end{tikzcd}
\]
then we have $\Kah{U_{ij}}{S_i} \cong \gamma^{\ast}_{ij}\Kah{X}{S}$ in each category $\QCoh(U_{ij})$. Because we have isomorphisms $U_{ij} \cong \Spec B_{ij}$ and $S_i \cong \Spec A_i$ for some rings $A_i, B_{ij}$ by virtue of each cover being an affine cover, we compute that
\[
\Kah{U_{ij}}{S_i} \cong \Kah{\Spec B_{ij}}{\Spec A_i} \cong \widetilde{\Kah{B_{ij}}{A_i}}
\]
by invoking Lemma \ref{Lemma: Section Kahlers: Relative Kahlers are qcoh over Affine Schemes} for each index $i \in I$ and fore each $j \in J_i$. This implies that in particular,
\[
\gamma_{ij}^{\ast}\left(\Kah{X}{S}\right) \cong \Kah{U_{ij}}{S_i} \cong  \left(\widetilde{\Kah{B_{ij}}{A_i}}\right)
\] 
in $\QCoh(U_{ij})$. Using that the $U_{ij}$ cover $X$ and appealing to Theorem \ref{Thm: Section Background Scheme: Checking QCoh affine locally} gives that  $\Kah{X}{S}$ is a quasi-coherent sheaf\footnote{It is a standard lemma in site theory that if $(\Cscr,J)$ is a site generated by a Grothendieck pretopology $(\Cscr,\tau)$, and if $\Fscr$ is a presheaf on $\Cscr_{/X}$ in the slice topology $J_{/X}$ with a cover $\lbrace f_i:U_i \to X \; | \; i \in I \rbrace \in \tau(U)$ for which each presheaf $f_i^{\ast}\Fscr$ is a $J_{/U_i}$-sheaf with isomorphisms $(\pi_0^{ij})^{\ast}(f_i^{\ast}\Fscr) \cong (\pi_1^{ij})^{\ast}(f_j^{\ast}\Fscr)$ then $\Fscr$ is a $J_{/X}$-sheaf. That is, if $\Fscr$ is a presheaf on $X$ for which each restriction $f_i^{\ast}\Fscr$ is a $U_i$-sheaf (for each $f_i$ a map appearing in a cover $\lbrace f_i:U_i \to X \; | \; i \in I \rbrace \in \tau(X)$) and if there are isomorphisms of the restrionctions $f_i^{\ast}\Fscr$ and $f_j^{\ast}\Fscr$ as they restrict to the pullbacks $U_i \times_X U_j$, then $\Fscr$ is a sheaf on $X$.} on $X$.
\end{proof}

We now show (finally) that the sheaf of relative K{\"a}hler differentials is a corepresenting object for the $S$-derivations functor. It is in this sense that the quasi-coherent sheaf $\Kah{X}{S}$ of K{\"a}hler differentials play the same role for schemes what the module $\Kah{A}{B}$ of K{\"a}hler differentials plays for commutative rings.

\begin{proposition}\label{Prop: Section Kahlers: Kahler Diffles are Corepresenting}
Let $f:X \to S$ be a morphism of schemes. Then the quasi-coherent sheaf $\Kah{X}{S}$ corepresents the functor $\mathsf{Der}_S(\Ocal_X,-):\QCoh(X) \to \Mod{\Ocal_X(\lvert X\rvert)}$ in the sense that there is a natural isomorphism
\[
\mathsf{Der}_S\left(\Ocal_X, -\right) \cong \QCoh(X)\left(\Kah{X}{S},-\right).
\]
\end{proposition}
\begin{proof}
Begin by fixing a quasi-coherent sheaf $\Fscr$ on $X$. Then we derive the following chain of equivalences:
\begin{prooftree}
	\AxiomC{$\partial \in \operatorname{Der}_S(\Ocal_X,\Fscr)$}
	\UnaryInfC{$\partial \in \Mod{f^{-1}\Ocal_S}(\Ocal_X,\Fscr)$ and $\partial_U \in \operatorname{Der}_{\Ocal_X(U)}(\Ocal_X(U),\Fscr(U))$ for $U$ open}
	\UnaryInfC{$\partial \in \Mod{f^{-1}\Ocal_S}(\Ocal_X,\Fscr)$ and $\partial_U^{\sharp} \in \Mod{\Ocal_X(U)}(\Kah{\Ocal_X(U)}{(f^{-1}\Ocal_S)(U)}, \Fscr(U))$ for $U$ open}
	\UnaryInfC{$\partial \in \Mod{f^{-1}\Ocal_S}(\Ocal_X,\Fscr)$ and $\partial_U^{\sharp} \in \Mod{\Ocal_X(U)}((\Kah{X}{S})(U), \Fscr(U))$ for $U$ open}
	\UnaryInfC{$\partial^{\sharp} \in \Mod{\Ocal_X}(\Kah{X}{S},\Fscr)$ and $\partial_U^{\sharp} \in \Mod{\Ocal_X(U)}((\Kah{X}{S})(U), \Fscr(U))$ for $U$ open}
	\UnaryInfC{$\partial^{\sharp} \in \Mod{\Ocal_X}(\Kah{X}{S},\Fscr)$}
\end{prooftree}
Because each derivation above is a natural equivalence this shows that we have our natural isomorphism $\mathsf{Der}_S(\Ocal_X,-) \cong \QCoh(X)(\Kah{X}{S},-)$, as claimed.
\end{proof}
\begin{remark}[{cf.~ \cite[Section 16.5.4, Corollaire 16.5.5]{EGA44}}]
One may be interested in sheafifying the scheme-theoretic derivations we have been studying. That is, it is natural to construct a quasi-coherent sheaf $\mathcal{Der}_S(\Ocal_X,\Fscr)$ associated to a quasi-coherent sheaf $\Fscr$. The idea here is to define, for opens $U \subseteq \lvert X \rvert$ with associated inclusion $j:U \to X$,
\[
\mathcal{Der}_S(\Ocal_X,\Fscr)(U) := \mathsf{Der}_S(\Ocal_U,j^{\ast}\Fscr).
\]
Then it is relatively straightforward to show that $\mathcal{Der}_S(\Ocal_X,\Fscr)$ is a quasi-coherent sheaf. The main difference between this sheaf and the $\Ocal_X(X)$-module $\mathsf{Der}_S(\Ocal_X,\Fscr)$ lie in the fact that if we write $[-,-]$ to denote the internal hom functor in $\QCoh(X)$, then
\[
\mathcal{Der}_S(\Ocal_X,\Fscr) \cong [\Kah{X}{S},\Fscr]
\]
in $\QCoh(X)$. In particular, we get that this also recaptures Proposition \ref{Prop: Section Kahlers: Kahler Diffles are Corepresenting} by taking global sections of the isomorphism of sheaves above.
\end{remark}

One benefit of the approach to developing the sheaf of relative K{\"a}hler differentials we took is that we can deduce many of the universal properties and compatibility conditions that $\Kah{X}{S}$ satisfies from arguing locally on rings by the sheaf-theoretic natural of $\Kah{X}{S}$ while also, by making use of the relative spectrum and relative symmetric algebra functors, argue with $\Kah{X}{S}$ geometrically via the affine $X$-scheme $\RelSpec{X}(\RelSym{\Ocal_X}(\Kah{X}{S}))$. In fact, as we will see below, the scheme $\RelSpec{X}(\RelSym{\Ocal_X}(\Kah{X}{S})) =: T_{X/S}$ is the tangent scheme of $X$ in $\Sch_{/S}$ and also Grothendieck's ``{fibr{\'e} tangent de $X$ r{\'e}lativement {\`a} $S$}'' (cf.\@ \cite[Section 16.5.12, Line 3]{EGA44}). 

To conclude this section we simply need to record the flavour in which the relative K{\"a}hler differentials are functorial in $\Sch_{/S}$ so that we can later define the tangent functor on $S$-schemes. To this end we need to know how, given an $S$-scheme morphism $f:X \to Y$, in what way is there are map between relative K{\"a}hler differentials $\Kah{X}{S}$ and $\Kah{Y}{S}$.

We now recall from the affine case that by Proposition \ref{Prop: Section Kahlers: Kahler diffles are functors}, the K{\"a}hler differentials $\Kah{(-)}{R}$ induce a functor
\[
\Kah{(-)}{R}:\CAlg{R} \to \MMod
\]
because we have access to a fibration $\MMod \to \Cring$ whose associated pseudofunctor $\Mod{(-)}:\Cring^{\op} \to \fCat$ has translation functors given by restriction of scalars $\res_{(-)}$. Because in the case $\QCoh(-):\Sch_{/S}^{\op} \to \fCat$ is given by \emph{pullback} (which plays the role of extension of scalars in the ring-theoretic setting), we want to be looking for an alternative perspective. However we are guided in this case precisely by Remark \ref{Remark: Section Kahlers: Cotangent sequence}. 

In Remark \ref{Remark: Section Kahlers: Cotangent sequence}, we saw that given a morphism $f:A \to B$ of rings, the adjunct $(\mathrm{d}f)^{\sharp}:\Kah{A}{R} \otimes_A B \to \Kah{B}{R}$ gave rise to a morphism from the extension of scalars along $f:A \to B$ of $\Kah{A}{R}$ to the module $\Kah{B}{R}$.In the scheme-theoretic setting, assume that we have a pair of morphisms $f:X \to Y$ and $g:Y \to S$. By schemifying the relative cotangent sequence we obtain a scheme-theoretic relative cotangent sequence: an exact sequence of quasicoherent sheaves (cf.\@ \cite[Corollaire 16.4.19]{EGA44}):
\[
\begin{tikzcd}
f^{\ast}\Kah{Y}{S} \ar[r]{}{u_{f,g}} &  \Kah{X}{S} \ar[r] & \Kah{X}{Y} \ar[r] & 0
\end{tikzcd}
\]
This indicates that the functor $\Sch_{/S}^{\op} \to \QQCoh$ is \emph{op}fibrational in nature as opposed to fibrational. As such, we must use the category $\mathbb{QC}$ and Proposition \ref{Prop: Section Background Scheme: QCoh cofibration} instead of the fibration $\QQCoh$.
\begin{proposition}\label{Prop: Section Kahlers: Functoriality of Kahlers for QCoh}
For any scheme $S$ there is a functor 
\[
\Kah{(-)}{S}:\Sch_{/S} \to \mathbb{QC}
\]
which sends schemes $X$ to $(X,\Ocal_X)$ and which sends morphisms $f:X \to Y$ to $(f,u_f):(X,\Kah{X}{S}) \to (Y,\Kah{Y}{S})$ where $u_f$ is the map in the Relative Cotangent Complex indicated below:
\[
\begin{tikzcd}
f^{\ast}\Kah{Y}{S} \ar[r]{}{u_f}  & \Kah{X}{S} \ar[r] & \Kah{X}{Y} \ar[r] & 0
\end{tikzcd}
\]
\end{proposition}

\section{The Tangent Categories of Commutative Algebras and Schemes}\label{Section: Tan Cats}
In this section provide a new\footnote{In the sense that it has not appeared in this manner nor at this level of detail to my knowledge.} perspective on the tangent structure on the category $\Sch_{/S}$ of $S$-schemes. This description gives a relatively clean explanation of how to use the affine scheme tangent categories $\CAlg{R}^{\op} \simeq \mathbf{AffSch}_{/R}$ and glue them up to a tangent structure on $\Sch_{/S}$ for arbitrary (not-necessarily-affine) schemes $S$. The basic strategy is to use Theorem \ref{Thm: Section Background Scheme: Hopf algebras for RelSym} to glue the ways in which we induce the tangent structure on $\mathbf{AffSch}_{/R}$ via the Hopf algebra structure on each tangent scheme $\Spec(\Sym{A}(\Kah{A}{R}))$ for commutative $R$-algebras $A$. 

While much of what we do below is known in the tangent-category theory literature, we present it very explicitly below to provide a bridge between the tangent-categorical and algebraic-geometric perspectives. In particular, we provide it to show how tangent-categorical information may be carefully and precisely glued sheaf-theoretically. As such, we first recall what tangent categories and their differential bundles are (as well as how each object is incarnated in the affine scheme setting).

\subsection{Tangent Categories, Differential Bundles, and the Tangent Category of Affine Schemes}\label{Subsection: Tan Cat Exposition}

We begin this section by giving a quick recollection of what it means to be a tangent category in the sense of \cite{GeoffRobinDiffStruct} as well as what it means to be a differential bundle in a tangent category (in the sense of \cite{GeoffRobinBundle}). While we will give some intuition towards what the axioms are encoding, we will not spend too much time presenting the various examples of these gadgets which occur in nature aside from pointing out that they exist.

\begin{definition}[{\cite[Definition 2.3]{GeoffRobinDiffStruct}, \cite{GeoffRobinBundle}}]\label{Defn: Section Tangent: Tangent category}
	Let $\Cscr$ be a category. A \emph{tangent structure} $\Tbb$ on $\Cscr$ consists of the following information:
	\begin{enumerate}
		\item There is a functor $T:\Cscr \to \Cscr$ together with a natural transformation
		\[
		\begin{tikzcd}
			\Cscr \ar[rr, bend left = 20, ""{name = U}]{}{T} \ar[rr, bend right = 20, swap, ""{name = D}]{}{\id} & & \Cscr
			\ar[from = U, to = D, Rightarrow, shorten <= 4pt, shorten >= 4pt]{}{p}
		\end{tikzcd}
		\]
		for which for all objects $X \in \Cscr_0$ and for all $n \in \N$, the wide pullback
		\[
		T_nX := \lim\left(\begin{tikzcd}
			TX \ar[dr, swap]{}{p_X} & \cdots \ar[d]{}[description]{n\,\text{copies}} & TX \ar[dl]{}{p_X} \\
			& X
		\end{tikzcd}\right)
		\]
		exists and is preserved by the functors $T^m$ for all $m \in \N$.
		\item There are natural transformations
		\[
		\begin{tikzcd}
			\Cscr \ar[rr, bend left = 20, ""{name = U}]{}{\id} \ar[rr, bend right = 20, swap, ""{name = D}]{}{T} & & \Cscr
			\ar[from = U, to = D, Rightarrow, shorten <= 4pt, shorten >= 4pt]{}{0}
		\end{tikzcd}\quad \begin{tikzcd}
			\Cscr \ar[rr, bend left = 20, ""{name = U}]{}{T_2} \ar[rr, bend right = 20, swap, ""{name = D}]{}{T} & & \Cscr
			\ar[from = U, to = D, Rightarrow, shorten <= 4pt, shorten >= 4pt]{}{\operatorname{add}}
		\end{tikzcd}
		\]
		which make $(p_X, 0_X, \operatorname{add}_X)$ into an internal commutative monoid in $\Cscr_{/X}$ for all objects $X \in \Cscr_0$.
		\item There is a natural transformation
		\[
		\begin{tikzcd}
			\Cscr \ar[rr, bend left = 20, ""{name = U}]{}{T} \ar[rr, bend right = 20, swap, ""{name = D}]{}{T^2} & & \Cscr
			\ar[from = U, to = D, Rightarrow, shorten <= 4pt, shorten >= 4pt]{}{\ell}
		\end{tikzcd}
		\]
		called the \emph{vertical lift} which renders the diagrams
		\[
		\begin{tikzcd}
			TX \ar[r]{}{\ell_X} \ar[d, swap]{}{p_X} & T^2X \ar[d]{}{(T \ast p)_X} \\
			X \ar[r, swap]{}{0_X} & TX
		\end{tikzcd}\quad
		\begin{tikzcd}
			T_2X \ar[rr]{}{\langle \ell_X \circ \pi_0, \ell_X \circ \pi_1\rangle} \ar[d, swap]{}{\operatorname{add}_X} & & T(T_2X) \ar[d]{}{(T \ast \operatorname{add})_X} \\
			TX \ar[rr, swap]{}{\ell_X} & & T^2X
		\end{tikzcd}\quad
		\begin{tikzcd}
			X \ar[r]{}{0_X} \ar[d, swap]{}{0_X} & TX \ar[d]{}{(T \ast 0)_X} \\
			TX \ar[r, swap]{}{\ell_X} & T^2X
		\end{tikzcd}
		\]
		commutative. In the language of \cite{GeoffRobinDiffStruct}, $(\ell_X,0_X)$ is a morphism of additive bundles.
		\item There is a natural transformation
		\[
		\begin{tikzcd}
			\Cscr \ar[rr, bend left = 20, ""{name = U}]{}{T^2} \ar[rr, bend right = 20, swap, ""{name = D}]{}{T^2} & & \Cscr
			\ar[from = U, to = D, Rightarrow, shorten <= 4pt, shorten >= 4pt]{}{c}
		\end{tikzcd}
		\]
		called the \emph{canonical flip} which renders the diagrams
		\[
		\begin{tikzcd}
			T^2X \ar[r]{}{c_X} \ar[d, swap]{}{(T \ast p)_X} & T^2X \ar[d]{}{(p \ast T)_X} \\
			TX \ar[r, equals] & TX
		\end{tikzcd}\quad
		\begin{tikzcd}
			(T^2)_2X \ar[rr]{}{\langle c_X \circ \pi_0, c_X \circ \pi_1\rangle} \ar[d, swap]{}{(T \ast \operatorname{add})_X} & & (T^2)_2X \ar[d]{}{(\operatorname{add} \ast T)_X} \\
			T^2X \ar[rr, equals] & & T^2X
		\end{tikzcd}\quad
		\begin{tikzcd}
			TX \ar[r, equals] \ar[d, swap]{}{(T \ast 0)_X} & TX \ar[d]{}{(0 \ast T)_X} \\
			T^2X \ar[r, swap]{}{c_X} & T^2X
		\end{tikzcd}
		\]
		commutative. In the language of \cite{GeoffRobinDiffStruct}, $(c_X, \id_{TX})$ is a morphism of additive bundles.
		\item The identities $c^2 = \id, \ell = c \circ \ell$, and
		\[
		\begin{tikzcd}
			T^3 \ar[d, swap]{}{T \ast c} \ar[r]{}{c \ast T} & T^3 \ar[r]{}{T \ast c} & T^3 \ar[d]{}{c \ast T} \\
			T^3 \ar[r, swap]{}{c \ast T} & T^3 \ar[r, swap]{}{T \ast c} & T^3
		\end{tikzcd}\quad
		\begin{tikzcd}
			T \ar[r]{}{\ell} \ar[d, swap]{}{\ell} & T^2 \ar[d]{}{T \ast \ell} \\
			T^2 \ar[r, swap]{}{\ell \ast T} & T^3
		\end{tikzcd}\quad
		\begin{tikzcd}
			T^2 \ar[d, swap]{}{c} \ar[r]{}{\ell \ast T} & T^3 \ar[r]{}{T \ast c} & T^3 \ar[d]{}{c \ast T} \\
			T^2 \ar[rr, swap]{}{T \ast \ell} & & T^3
		\end{tikzcd}
		\]
		all hold in $[\Cscr,\Cscr]$.
		\item For all objects $X \in \Cscr_0$, the diagram
		\[
		\begin{tikzcd}
		T_2X \ar[rr]{}{p_X\circ \pi_0} \ar[d, swap]{}{v} & & X \ar[d]{}{0_X} \\
		T^2X \ar[rr, swap]{}{(T \ast p)_X} & & TX
		\end{tikzcd}
		\]
		is an equalizer and where $v$ is the morphism
		\[
		v := (T \ast \operatorname{add})_X \circ \left\langle (0 \ast T)_X, \ell_X \right\rangle.
		\] 
		This pullback expresses the \emph{universality of the vertical lift}.
	\end{enumerate}
	A \emph{tangent category} is a pair $(\Cscr,\Tbb)$ where $\Cscr$ is a category and $\Tbb$ is a tangent structure.
\end{definition}
\begin{remark}
	Here are some interpretations of the axioms above.
	\begin{enumerate}
		\item The wide pullbacks $T_nX$ encode the space of $n$-tangent vectors of $X$ which have the same anchor vector in $X$. In this case $T^m$ preserving $T_nX$ says that if you differentiate $m$ times and had $n$ tangent vectors anchored at the same point, then you now have precisely the space of $n$ degree $m$-vectors tangent to $T^mX$ anchored at the same base point.
		\item The vertical lift $\ell:T \Rightarrow T^2$ encodes a differentiation operation, as it says that if you take a tangent vector and differentiate it, you get a tangent vector which is in a degree $2$ infinitesimal neighbourhood of the base point.
		\item The canonical flip $c$ encodes the symmetry of mixed partials
		\[
		\frac{\partial^2f}{\partial x\,\partial y} = \frac{\partial^2f}{\partial y\,\partial x}
		\]
		for $\mathcal{C}^{\infty}$-morphisms $f$. Equivalently, it encodes the symmetry of the Hessian matrix for $\Ccal^{\infty}$-morphisms of smooth manifolds.
	\end{enumerate}
\end{remark}
\begin{example}
Every category $\Cscr$ is a tangent category with tangent structure $\Ibb = (\id, \id, \id, \id, \id, \id)$.
\end{example}
\begin{example}
The category $\SMan$ of smooth real manifolds is a tangent category with tangent functor $T$ given by the tangent bundle functor, $p:T \Rightarrow \id$ given by the bundle projection, $0:\id \Rightarrow T$ given by selecting the zero vector, and $\operatorname{add}:T_2 \Rightarrow T$ given by adding two tangent vectors which are anchored at the same base point.
\end{example}
\begin{example}\label{Example: Tangent Structure on CAlgR}
If $R$ is a commutative rig then $\CAlg{R}$ is a tangent category with $T(A) := A[\epsilon] \cong A[x]/(x^2) \cong A \otimes_R R[x]/(x^2)$, $p_A:TA \to A$ the augmentation map $a + x\epsilon \mapsto a$, $0_A:A \to TA$ the algebra map, and with $\operatorname{add}:T_2(A) \to TA$ given by adding nilpotent indeterminates, i.e., by $(a+x\epsilon,a+y\epsilon) \mapsto (a+(x+y)\epsilon)$.
\end{example}
\begin{example}
Let $\Cscr$ be a category with finite biproducts, i.e., a semiadditive category. Then $\Cscr$ is a tangent category with tangent functor defined by $TX := X \oplus X$ and $Tf := f \oplus f$, with bundle projection $p$ defined to be the first projection $\pi_0:X \oplus X \to X$, and with zero section $0$ defined to be the first inclusion $\iota_0:X \to X \oplus X$. A routine calculation shows that $T_2(X) \cong X \oplus (X \oplus X)$ and that the addition map $\operatorname{add}:T_2X \to TX$ is given by $\id_X \oplus (\id_X + \id_X)$. 
The verical lift $\ell_X:TX \to T^2X$ is then given by
\[
\begin{tikzcd}
TX \ar[r, equals] & X \oplus X \ar[r]{}{\cong} &  X \oplus (0 \oplus 0) \oplus X \ar[r]{}{} & (X \oplus X) \oplus (X \oplus X) \ar[r, equals] & T^2X
\end{tikzcd}
\]
while the canonical flip $c_X:T^2X \to T^2X$ is the map
\[
\begin{tikzcd}
T^2X \ar[rrr]{}{c_X} \ar[d, equals] & & & T^2X \\
(X \oplus X) \oplus (X \oplus X) \ar[d, swap]{}{\cong} & & & (X \oplus X) \oplus (X \oplus X) \ar[u, equals] \\
(X \oplus (X \oplus X)) \oplus X \ar[rrr, swap]{}{(\id_X \oplus s_{X,X})\oplus \id_X} & & & (X \oplus (X \oplus X)) \oplus X \ar[u, swap]{}{\cong}
\end{tikzcd}
\]
where $s_{X,X}$ is the map which switches the order of the terms in the product and the unlabeled vertical isomorphisms are simply reassociation isomorphisms.
\end{example}

We now discuss differential bundles for tangent categories. These play the same role for tangent categories that vector bundles play for smooth manifolds. However, it is worth noting that the definition of a differential bundle does not involve any intrinsic notion of saying that the bundle map $q:E \to X$ of a differential bundle is locally trivializable; instead it uses the tangent structure directly to say that the bundle map $q$ is locally linear (in a sense encoded by the map $\lambda$ below) and then encodes the vector bundle nature of an object through said local linearity.

\begin{definition}[{\cite[Definition 2.3]{GeoffRobinBundle}}]\label{Defn: Section Tangent: Differential Bundle}
	Let $\Cscr$ be a tangent category with $X \in \Cscr_0$. A \emph{differential bundle} 
	\[
	\mathsf{q} = \bigg(q:E \to X, \zeta:X \to E, \sigma:E \times_X E \to E, \lambda:E \to TE\bigg)
	\]
	over $X$ consists of the maps $q, \zeta, \sigma, \lambda$ above, subject to the following conditions:
	\begin{itemize}
		\item For all $n \in \N$, the wide pullbacks
		\[
		E_n := \lim\left(\begin{tikzcd}
			E \ar[dr, swap]{}{q} & \cdots \ar[d]{}[description]{n\,\text{copies}} & E \ar[dl]{}{q} \\
			& X
		\end{tikzcd}\right)
		\]
		exist and are preserved by $T^m$ for all $m \in \N$.
		\item The maps $\zeta:X \to E$ and $\sigma:E_2 \to E$ make $(q,\zeta,\sigma)$ into an internal commutative monoid in $\Cscr_{/X}$.
		\item The lift map $\lambda:E \to TE$ make the diagrams
		\[
		\begin{tikzcd}
			E \ar[r]{}{\lambda} \ar[d, swap]{}{q} & TE \ar[d]{}{Tq} \\
			X \ar[r, swap]{}{0_X} & TX
		\end{tikzcd}\quad
		\begin{tikzcd}
			X \ar[d, swap]{}{\zeta} \ar[r]{}{0_X} & TX \ar[d]{}{T\zeta} \\
			E \ar[r,swap]{}{\lambda} & TE
		\end{tikzcd}\quad
		\begin{tikzcd}
			E_2 \ar[rr]{}{\langle \lambda \circ \pi_0, \lambda \circ \pi_1\rangle} \ar[d, swap]{}{\sigma} & & T(E_2) \ar[d]{}{T\sigma} \\
			E \ar[rr, swap]{}{\lambda} & & E
		\end{tikzcd}
		\]
		commute. In the language of \cite{GeoffRobinDiffStruct}, $(\lambda,0_X)$ is a morphism of additive bundles in $\Cscr$.
		\item The map $\lambda$ makes the diagrams
		\[
		\begin{tikzcd}
			E \ar[r]{}{\lambda} \ar[d, swap]{}{q} & TE \ar[d]{}{p_E} \\
			X \ar[r, swap]{}{\zeta} & E
		\end{tikzcd}\quad
		\begin{tikzcd}
			X \ar[d, swap]{}{\zeta} \ar[r]{}{\zeta} & E \ar[d]{}{0_E} \\
			E \ar[r,swap]{}{\lambda} & TE
		\end{tikzcd}\quad
		\begin{tikzcd}
			E_2 \ar[rr]{}{\langle \lambda \circ \pi_0, \lambda \circ \pi_1\rangle} \ar[d, swap]{}{\sigma} & & (TE)_2 \ar[d]{}{\operatorname{add}_E} \\
			E \ar[rr, swap]{}{\lambda} & & E
		\end{tikzcd}
		\]
		commute. That is, in the language of \cite{GeoffRobinDiffStruct} $(\lambda,\zeta)$ is a morphism of additive bundles in $\Cscr$.
		\item The diagram
		\[
		\begin{tikzcd}
		E \ar[r]{}{\lambda} \ar[d, swap]{}{\lambda} & TE \ar[d]{}{\ell_E} \\
		TE \ar[r, swap]{}{T\lambda} & T^2E
		\end{tikzcd}
		\]
		commutes.
		\item The lift $\lambda$ is universal and locally linear in the sense that the diagram
		\[
		\begin{tikzcd}
			E_2 \ar[d, swap]{}[description]{q \circ \pi_0 = q \circ \pi_1} \ar[rrrr]{}{T(\sigma) \circ \langle \ell_E \circ \pi_0, 0_E \circ \pi_1 \rangle} & & & & TE \ar[d]{}{Tq} \\
			X \ar[rrrr, swap]{}{0_X} & & & & TX
		\end{tikzcd}
		\]
		is a pullback which is preserved by $T^m$ for all $m \in \N$.
	\end{itemize}
\end{definition}

\begin{definition}\label{Defn: Section Tangent: Morphism of dbundles}
	Let $\Cscr$ be a tangent category with $q:E \to X$ and $r:F \to Y$ differential bundles. A \emph{morphism of differential bundles} from $q$ to $r$ is a pair of maps $(g,f)$ where $g:E \to F$ and $f:X \to Y$ are morphisms such that the diagram
	\[
	\begin{tikzcd}
		E \ar[r]{}{g} \ar[d, swap]{}{q} & F \ar[d]{}{r} \\
		X \ar[r, swap]{}{f} & Y
	\end{tikzcd}
	\]
	commutes. We say that $(g,f)$ is a \emph{linear morphism of differential bundles} if it is a morphism of differential bundles and the diagram
	\[
	\begin{tikzcd}
		E \ar[r]{}{g} \ar[d, swap]{}{\lambda_E} & F \ar[d]{}{\lambda_F} \\
		TE \ar[r, swap]{}{Tg} & TF
	\end{tikzcd}
	\]
	commutes as well.
\end{definition}
\begin{definition}
	If $\Cscr$ is a tangent category and $X \in \Cscr_0$ is an object then we will write $\DBun(X)$ for the category of differential bundles $q:E \to X$ over $X$ with linear morphisms $(f,\id_X)$ of differential bundles over $X$. If we need to explicitly emphasize the linearity of the maps, we write $\DBun_{\operatorname{lin}}(X)$ instead. Additionally, we write $\DBun_{\operatorname{bun}}(X)$ for the category of differential bundles over $X$ with not-necessarily-linear maps between bundles.
\end{definition}

Let us see some basic examples of categories of differential bundles for tangent categories.
\begin{example}
For the tangent category $\SMan$, by \cite[Theorem 1]{BenVectorBundles} there is an equivalence of categories $\DBun(M) \simeq \mathbf{VecBun}(M)$ for any smooth manifold $M$.
\end{example}
\begin{example}
Let $R$ be a commutative rig. For the tangent category $\CAlg{R}$, there is an equivalence of categories $\DBun(A) \simeq \Mod{A}$ for any commutative $A$-algebra by \cite[Theorem 3.13]{GeoffJSDiffBunComAlg} for the ring-theoretic case and by going through the argument of \cite[Remark 3.15]{GeoffJSDiffBunComAlg} for the rig-theoretic case; see also \cite[Example 2.2.3]{JSMeMapFlavoursInTanCats} for details.
\end{example}
\begin{example}
If $\Cscr$ is a category with finite biproducts equipped with the tangent structure described prior, then for any object $X$, $\DBun(X) \simeq \Cscr$ via the observation that a differential bundle $E \to X$ takes the form $E \cong X \oplus Y$ for some object $Y$ of $\Cscr$ while bundle morphisms $\varphi:E \to F$ precisely take the form $\id_X \oplus g:X \oplus Y \to X \oplus Z$.
\end{example}

The next example is of crucial importance to us,  so we will go through it in a high level of detail. Most of the descriptions here are derived from \cite{GeoffJSDiffBunComAlg} and can/will also be found in various expressions and degrees of detail in both works \cite{RobinMePartialMapRigs} and \cite{JSMeMapFlavoursInTanCats}. The reader who is uninterested in seeing the explicit combinatorial details checked and performed may want to skip ahead to Remark \ref{Remark: Dual Tangent Structure} for a purely formal way to construct the tangent structure on $\CAlg{R}^{\op}$ for any commutative rig $R$.

\begin{example}[{cf.\@ \cite[Example 4.2]{GeoffJSDiffBunComAlg}, \cite{RobinMePartialMapRigs}, \cite{JSMeMapFlavoursInTanCats}}]\label{Example: The tangent category of affine R schemes}
	Let $R$ be a commutative rig and let $\CAlg{R}^{\op}$ denote the opposite category of commutative $R$-algebras --- when $R$ is a commutative ring this is of course equivalent to the category of affine $R$-schemes. Then $\CAlg{R}^{\op}$ is a tangent category where the tangent functor is defined by sending a commutative $R$-algebra $A$ to the symmetric $A$-algebra of relative K{\"a}hler differentials\footnote{If $R$ is a rig and if $A$ is a commutative $R$ algebra, an $R$-derivation on $A$ with coefficients in an $A$-module $M$ is an $R$-linear map $\partial:A \to M$ such that for all $a, b \in A$, $\partial(a+b) = \partial(a)+\partial(b)$, $\partial(ab) = \partial(a)b + a\partial(b),$ and $\partial(1) = 0$. Only in the presence of negatives can we omit the equation $\partial(1) = 0$, as the identity $\partial(1) = \partial(1^2) = \partial(1)+\partial(1)$ simply says that $\partial(1)$ is additively idempotent otherwise.} 
	\begin{equation}\label{Eqn: Section Tangent: Tangent Structure of Affine Tangent Functor}
	T_{A/R} := \Sym{A}\left(\Kah{A}{R}\right).
	\end{equation}
	The projection $p:T_{(-)/R} \Rightarrow \id$ in $\CAlg{R}^{\op}$ is defined by taking $p_A:T_{A/R} \to A$ to be the opposite of the map
	\begin{equation}\label{Eqn: Section Tangent: Tangent Structure of Affine Projection Map}
	A \xrightarrow{\cong} \Sym{A}\left(0\right) \xrightarrow{\Sym{A}(\gnab_{\Omega})} \Sym{A}\left(\Kah{A}{R}\right)
	\end{equation}
	where $\gnab_{\Omega}:0 \to \Kah{A}{R}$ is the zero map from $0$ to $\Kah{A}{R}$ in $\Mod{A}$. 
	
	The zero map $0:\id \Rightarrow T_{(-)/R}$ in $\CAlg{R}^{\op}$ is the opposite of the map
	\begin{equation}\label{Eqn: Section Tangent: Tangent Structure of Affine Zero Map}
	\Sym{A}\left(\Kah{A}{R}\right) \xrightarrow{\Sym{A}(\bang_{\Omega})} \Sym{A}\left(0\right) \xrightarrow{\cong} A
	\end{equation}
	where $\bang_{\Omega}:\Kah{A}{R} \to 0$ is the zero map from $\Kah{A}{R}$ to $0$ in $\Mod{A}$.
	
	The addition $\operatorname{add}_A:T_2(-) \Rightarrow T_{(-)/R}$ in $\CAlg{R}^{\op}$ is the map opposite to the bialgebra addition morphism
	\[
	\nabla_{\Omega}:\Sym{A}\left(\Kah{A}{R}\right) \to \Sym{A}\left(\Kah{A}{R}\right)  \otimes_A \Sym{A}\left(\Kah{A}{R}\right), \quad x \mapsto x \otimes 1 + 1 \otimes x
	\]
	which  recognizes $\Sym{A}(M)$ as an $A$-bialgebra for the $A$-module $M = \Kah{A}{R}$. When $R$ is a commutative ring, each commutative monoid $(p_A, 0_A, \operatorname{add}_A)$ is an internal Abelian group by Proposition \ref{Prop: Section Background Alg: Symmetric Alg is Coalg}.
	
	The vertical lift $\ell:T \Rightarrow T^2$ and canonical flip are defined as follows. Given a commutative $R$-algebra $A$, the underlying $A$-algebra which describes $T^2(A)$ is the algebra
	\[
	T^2(A) := \Sym{\Sym{A}(\Kah{A}{R})}\left(\Kah{\big[\Sym{A}(\Kah{A}{R})\big]}{R}\right).
	\] 
	The rig $T^2(A)$ admits the following explicit description (cf.\@ \cite{RobinMePartialMapRigs} for the genesis of our description below in the rig-theoretic setting and \cite{GeoffRobinDiffStruct} for a description in the ring-theoretic setting). First, we can always write the rig $\Sym{A}(\Kah{A}{R})$ with the description
	\[
	\Sym{A}(\Kah{A}{R}) \cong \frac{A[\mathrm{d}a:a \in A]}{\mathsf{cong}\langle\mathrm{d}(a+a^{\prime}) \simeq \mathrm{d}(a) + \mathrm{d}(a^{\prime}), \mathrm{d}(aa^{\prime}) \simeq \mathrm{d}(a)a^{\prime} + a\mathrm{d}(a^{\prime}), \mathrm{d}(1_R) \simeq 0: a, a^{\prime} \in A\rangle}.
	\]
	where $\mathsf{cong}\langle \operatorname{Relations}\rangle$ denotes the smallest congruence on the algebra generated by the relations given. Applying the tangent functor $T_{(-)/R}$ to this algebra gives yet another algebra
	\begin{align*}
		\Sym{\Sym{A}(\Kah{A}{R})}\left(\Kah{\big[\Sym{A}(\Kah{A}{R})\big]}{R}\right) \cong \frac{A[\mathrm{d}a, \delta(a), (\mathrm{d}\delta)(a):a \in A]}{C}
	\end{align*}
	where $C$ is the congruence generated by the relations:
	\begin{align}
	\label{Eqn: First of the T2 rels}	\mathrm{d}(1_R) &= 0, & \delta(1_R) &= 0, \\
	\label{Eqn: Second of the T2 rels}	\mathrm{d}(a+b) &= \mathrm{d}(a) + \mathrm{d}(b), & \delta(a+b) &= \delta(a) + \delta(b), \\
	\label{Eqn: Third of the T2 rels}	\mathrm{d}(ab) &= \mathrm{d}(a)b + a\mathrm{d}(b), & \delta(ab) &= \delta(a)b + a\delta(b) \\
	\label{Eqn: Fourth of the T2 rels}	(\mathrm{d}\delta)(1_R) &= 0 & (\mathrm{d}\delta)(a+b) &= (\mathrm{d}\delta)(a) + (\mathrm{d}\delta)(b) \\
	\label{Eqn: Fifth of the T2 rels}	(\mathrm{d}\delta)(ab) &= \mathrm{d}(a)\delta(b) + (\mathrm{d}\delta)(a)b + \mathrm{d}(b)\delta(a) + a(\mathrm{d}\delta)(b)
	\end{align}
	On one hand the vertical lift $\ell_A:T_{A/R} \to T^2(A)$ is defined to be the opposite of the rig-map
	\begin{equation}\label{Eqn: Section Tangent: Tangent Structure of Affine VertLift Map}
	\rotatebox[origin = c]{180}{$\ell$}_A:\Sym{\Sym{A}(\Kah{A}{R})}\left(\Kah{\big[\Sym{A}(\Kah{A}{R})\big]}{R}\right) \longrightarrow \Sym{A}\left(\Kah{A}{R}\right)
	\end{equation}
	given by
	\[
	a \mapsto a,\quad \mathrm{d}a \mapsto 0,\quad  \delta(a) \mapsto 0,\quad (\mathrm{d}\delta)(a) \mapsto \mathrm{d}a
	\]
	on generators. On the other hand, the canonical flip $c_{A}:T^2_{A/R} \to T^2_{A/R}$ is the opposite map of the morphism $\hat{c}$ defined on generators by
	\begin{equation}\label{Eqn: Section Tangent: Tangent Structure of Affine CanonFlip Map}
	c_A(a) = a,\quad c_A(\mathrm{d}(a)) = \delta(a),\quad c_A(\delta(a)) = \mathrm{d}a,\quad c_A(\mathrm{d}\delta)(a) = (\mathrm{d}\delta)(a).
	\end{equation}
\end{example}
\begin{example}
	Let us now examine the category of linear differential bundle maps over a base $R$-algebra $A$ in $\CAlg{R}^{\op}$. The arguments of \cite[Theorem 4.17]{GeoffJSDiffBunComAlg} prove for the ring-theoretic case (\cite[Remark 4.20]{GeoffJSDiffBunComAlg} argues for the rig-theoretic case; cf.\@ also \cite[Example 2.2.4]{JSMeMapFlavoursInTanCats}), we have that $\DBun(A)^{\op} \simeq \Mod{A}$. When $R$ is a commutative ring then for any affine $R$-scheme $\Spec A$, we obtain an opposite equivalence
	\[
	\mathbf{QCoh}(\Spec A) \simeq \Mod{A} \simeq \DBun(A)^{\op}.
	\]
	An important take-away from this equivalence is that for any map $f:E \to F$ of differential bundles over $A$ in $\CAlg{R}^{\op}$, $f$ is linear if and only if $f$ is isomorphic to a map of the form $\Sym{A}(g)$ for some module map $g:N \to M$.
\end{example}
\begin{remark}\label{Remark: Dual Tangent Structure}
While we constructed it very explicitly, the existence of the tangent structure on $\CAlg{R}^{\op}$ can be seen as a formal consequence of the tangent structure on $\CAlg{R}$ of Example \ref{Example: Tangent Structure on CAlgR} --- it is the so-called dual tangent structure of \cite[Proposition 5.17]{GeoffRobinDiffStruct}. In \cite[Proposition 5.17]{GeoffRobinDiffStruct}, R.~ Cockett and G.~ Cruttwell showed that when one has a tangent category $(\Cscr,\Tbb)$ for which there are left adjoints
\[
\begin{tikzcd}
\Cscr \ar[rr, swap, bend right = 20, ""{name =R}]{}{T} & & \Cscr \ar[ll, swap, bend right = 20, ""{name = L}]{}{L} 
\ar[from = L, to = R, symbol = \dashv]
\end{tikzcd}\quad
\begin{tikzcd}
	\Cscr \ar[rr, swap, bend right = 20, ""{name =R}]{}{T_n} & & \Cscr \ar[ll, swap, bend right = 20, ""{name = L}]{}{L_n} 
	\ar[from = L, to = R, symbol = \dashv]
\end{tikzcd}
\]
for all $n \in \N$, then $\Cscr^{\op}$ has a tangent structure with $L^{\op}:\Cscr^{\op} \to \Cscr^{\op}$ as the tangent functor. This happens for commutative $R$-algebras, as for any $R$-algebras $A$ and $B$ the natural isomorphisms
\begin{align*}
\CAlg{R}(A,B[\epsilon]) &\cong \mathsf{Der}_{R}(A,B) \cong \Mod{A}(\Kah{A}{R},\Und{A}(B)) \cong \CAlg{R}\left(\Sym{A}(\Kah{A}{R}),B\right)
\end{align*}
show that the functor $\Sym{A}(\Kah{(-)}{R}) \dashv (-)[\epsilon]$ in $\CAlg{R}$. Noting further that because the category $\CAlg{R}$ is coextensive (as the category $\CAlg{R}^{\op} \simeq \mathbf{AffSch}_{/R}$ is extensive) gives that pushouts of finite products of commutative algebras commute and allows one to deduce in a straightforward manner that there are natural isomorphisms
\[
(-)[\epsilon]_n \cong \left((-)_n\right)[\epsilon]
\]
and adjunctions
\[
\begin{tikzcd}
	\CAlg{R} \ar[rrrr, swap, bend right = 20, ""{name =R}]{}{(-)[\epsilon]_n} & & & & \Cscr \ar[llll, swap, bend right = 20, ""{name = L}]{}{\Sym{(-)}(\Kah{(-)}{R})_n} 
	\ar[from = L, to = R, symbol = \dashv]
\end{tikzcd}
\]
for all $n \in \N$. Consequently by applying \cite[Proposition 5.17]{GeoffRobinDiffStruct} we see formally that the functor 
\[
\Sym{(-)}(\Kah{(-)}{R})^{\op}:\CAlg{R}^{\op} \to \CAlg{R}^{\op}
\]
induces a tangent structure on $\CAlg{R}^{\op}$ with each structure transformation $p, 0, \operatorname{add}, \ell, c$ given by the mate of the corresponding structure transformation of the tangent structure on $\CAlg{R}$. While this is formally much more straightforward to build, seeing the explicit combinatorial construction in Example \ref{Example: The tangent category of affine R schemes} gives one an idea of how to actually compute such maps and transformations.
\end{remark}

\subsection{The Tangent Structure on Categories of (Relative) Schemes}\label{Subsection: Tangent Category of Schemes}

Let us  now examine, prove, and study the tangent structure on $\Sch_{/S}$ for an arbitrary base scheme $S$ using the technology we have built thus far. Our first tasks are to define the tangent functor, its bundle projection, its zero section, and its addition. However, in light of how the tangent structure is defined in Example \ref{Example: The tangent category of affine R schemes}, how Theorem \ref{Thm: Section Background Scheme: Hopf algebras for RelSym} lifts the cocommutative Hopf algebra structure enjoyed by the symmetric algebra functor to quasi-coherent sheaves, and by both \cite[Line 16.5.12.1]{EGA44} and \cite[Theorem 4.27]{GeoffJSDiffBunComAlg}, we are led to the following definitions.

First, the object assignment for our tangent functor $T$ is forced essentially by what we described above: we require
\[
T_{X/S} := \RelSpec{X}\left(\RelSym{\Ocal_X}\left(\Kah{X}{S}\right)\right).
\]
The definition of $T$ on morphisms is a little more delicate. Applying the morphism $u_f$ of Proposition \ref{Prop: Section Background Scheme: QCoh cofibration} recall by Corollary \ref{Cor: Section Background Scheme: The pullback Sym isos for qcoh} that for any morphism $f:X \to Y$ of $S$-schemes we have a morphism
\[
f^{\ast}\left(\RelSym{\Ocal_Y}\left(\Kah{Y}{S}\right)\right) \xrightarrow{\cong} \RelSym{\Ocal_X}\left(f^{\ast}\Kah{Y}{S}\right) \xrightarrow{\RelSym{\Ocal_X}(u_f)} \RelSym{\Ocal_X}\left(\Kah{X}{S}\right)
\]
of quasi-coherent sheaves on $X$. Applying the relative spectrum functor at $X$ gives the map
\[
\RelSpec{X}\left(\RelSym{\Ocal_X}\left(\Kah{X}{S}\right)\right) \xrightarrow{\RelSpec{X}(u_f)} \RelSpec{X}\left(\RelSym{\Ocal_X}\left(f^{\ast}\Kah{Y}{S}\right)\right) \xrightarrow{\cong} \RelSpec{X}\left(f^{\ast}\left(\RelSym{\Ocal_Y}\left(\Kah{Y}{S}\right)\right)\right).
\]
Post-composing by the pullback preservation isomorphism of Corollary \ref{Cor: Section Background Scheme: The pullback iso for relspec} gives rise to a morphism of schemes
\[
\RelSpec{X}\left(\RelSym{\Ocal_X}\left(\Kah{X}{S}\right)\right) \xrightarrow{\RelSpec{X}(u_f)} \RelSpec{X}\left(\RelSym{\Ocal_X}\left(f^{\ast}\Kah{Y}{S}\right)\right) \xrightarrow{\cong} f^{\ast}\left(\RelSpec{Y}\left(\RelSym{\Ocal_Y}\left(\Kah{Y}{S}\right)\right)\right).
\]
Call this map $\theta_f$. Post-composing $\theta_f$ by the projection
\[
f^{\ast}\left(\RelSpec{Y}\left(\RelSym{\Ocal_Y}\left(\Kah{Y}{S}\right)\right)\right) = \RelSpec{Y}\left(\RelSym{\Ocal_Y}\left(\Kah{Y}{S}\right)\right) \times_Y X \longrightarrow \RelSpec{Y}\left(\RelSym{\Ocal_Y}\left(\Kah{Y}{S}\right)\right)
\]
allows us to define $Tf$ as the composite below:
\begin{equation}\label{Eqn: Section Tangent: Tangent Functor Assignment on Morphisms}
	\begin{tikzcd}
	\RelSpec{X}\left(\RelSym{\Ocal_X}\left(\Kah{X}{S}\right)\right) \ar[rr]{}{\theta_f} \ar[drr, swap]{}{Tf} & & f^{\ast}\left(\RelSpec{Y}\left(\RelSym{\Ocal_Y}\left(\Kah{Y}{S}\right)\right)\right) \ar[d]{}{\pr_0} \\
	 & & \RelSpec{Y}\left(\RelSym{\Ocal_Y}\left(\Kah{Y}{S}\right)\right)
	\end{tikzcd}
\end{equation}

\begin{definition}\label{Defn: Section Tangent: Tangent functor for schemes}
Let $S$ be a scheme. The tangent functor $T:\Sch_{/S} \to \Sch_{/S}$ is defined on objects by
\[
T_{X/S} := \RelSpec{X}\left(\RelSym{\Ocal_X}\left(\Kah{X}{S}\right)\right)
\]
and defined on morphisms $f:X \to Y$ by sending $f$ to
\[
Tf:\RelSpec{X}\left(\RelSym{\Ocal_X}\left(\Kah{X}{S}\right)\right) \longrightarrow \RelSpec{Y}\left(\RelSym{\Ocal_Y}\left(\Kah{Y}{S}\right)\right)
\]
as defined in Diagram (\ref{Eqn: Section Tangent: Tangent Functor Assignment on Morphisms}).
\end{definition}

An important observation regarding the tangent functor on schemes is that it is a right adjoint to a product functor and hence is an internal hom functor; that is, $T_{(-)/S} \cong [W_S,-]:\Sch_{/S} \to \Sch_{/S}$ for an infinitesimal object we now describe. Begin by observing that the presheaf $\Ocal_{S}[\epsilon]$ on $S$ defined by declaring, for $U \subseteq \lvert S \rvert$ open,
\[
\big(\Ocal_S[\epsilon]\big)(U) := \Ocal_S(U)[\epsilon] \cong \frac{\Ocal_S(U)[x]}{(x^2)}
\]
and extending the restriction maps by pushing out, is a quasi-coherent sheaf of $\Ocal_S$-algebras. As such, we define 
\[
W_S := \RelSpec{S}(\Ocal_S[\epsilon]).
\]
Because it is built in terms of the relative spectrum functor and hence has affine structure map to $S$, the scheme $W_S$ has the property that if $\lbrace U_i \; | \; i \in I \rbrace$ is an affine open cover of $S$ then the collection $\lbrace W_{U_i} \; | \; i \in I \rbrace$ is an affine open cover of $W_S$. By using the fact that affine-locally maps of the form $W_{U_i} \to V_j$ look like ring-theoretic maps $\Ocal_X(V_j) \to \Ocal_S(U_i)[\epsilon]$ and hence give derivations of $\Ocal_X(V_j)$ valued in $\Ocal_S(U_i)$, we are able to prove that the functor $(-) \times_S W_S$ has right adjoint given by the tangent functor $T_{(-)/S}$.

\begin{proposition}\label{Prop: Section Tangent: Tangent fucntor is a right adjoint}
Let $S$ be a scheme. Then there is an adjunction:
\[
\begin{tikzcd}
\Sch_{/S} \ar[rr, bend right = 20, swap, ""{name = R}]{}{T_{(-)/S}} & & \Sch_{/S} \ar[ll, bend right = 20, swap, ""{name = L}]{}{(-) \times_S W_S}
\ar[from = L, to = R, symbol = \dashv]
\end{tikzcd}
\]
\end{proposition}
\begin{proof}
Let $X$ and $Y$ be $S$-schemes with structure maps
\[
\nu_X:X \to S, \quad \nu_Y:Y \to S
\] 
and assume that we have a morphism
\[
f:X \times_S W_S \longrightarrow Y
\]
of $S$-schemes. As $W_S$ is affine over $S$, the projection $\pr_0:X \times_S W_S \to X$ is affine over $X$ by \cite[Proposition 1.5.1]{EGA2}. In particular, we have that there is a natural isomorphism
\[
X \times_S W_S = X \times_S \RelSpec{S}(\Ocal_S[\epsilon]) \cong \RelSpec{X}\left(\Ocal_X \otimes_{\nu_X^{-1}\Ocal_S} \nu_X^{-1}(\Ocal_S[\epsilon])\right) = \RelSpec{X}\left(\nu_X^{\ast}\Ocal_S[\epsilon]\right).
\]
by applying Corollary \ref{Cor: Section Background Scheme: The pullback iso for relspec}. Thus the map $f:X \times_S W_S \to S$ is equivalently given by a morphism of $S$-schemes
\[
\tilde{f}:\RelSpec{X}\left(\nu_X^{\ast}\Ocal_S[\epsilon]\right) \to Y
\]
Because $\tilde{f}$ has a scheme affine over $X$ as its domain, to define $\tilde{f}$ is equivalent to defining affine open covers $\lbrace V_j \; | \; j \in J \rbrace$ of $Y$, $\lbrace U_i \; | \; i \in I \rbrace$ of $X$, and $\lbrace S_k \; | \; k \in K \rbrace$ of $S$ for which if $\pr_0^{-1}U_i$ is the corresponding open affine subscheme in $X \times_S W_S$, then for all $i \in I$ there exists a $j \in J$ and a $k \in K$ for which the diagram
\[
\begin{tikzcd}
\pr_0^{-1}U_i \ar[rr] \ar[dd, swap]{}{\tilde{f}} \ar[dr] & & X \times_S W_S \ar[dr] \\
& S_k \ar[rr] & & S \\
V_j \ar[ur, swap]{}{\nu_Y} \ar[rr, swap] & & Y \ar[ur, swap]{}{\nu_Y}
\ar[from = 1-3, to = 3-3, crossing over, near start]{}{f}
\end{tikzcd}
\]
commutes. But by construction we have that
\[
\pr_0^{-1}U_i \cong U_i \times_{S_k} W_{S_k}.
\]
If $S_k \cong \Spec C_k$, $W_{S_k} \cong \Spec C_k[\epsilon]$ and if $U_i \cong \Spec A_i$ then
\[
\pr_0^{-1}U_i \cong \Spec A_i \times_{C_k} \Spec C_k[\epsilon] \cong \Spec\left(A_i \otimes_{C_k} C_k[\epsilon]\right) \cong \Spec A_i[\epsilon].
\]
Thus to give $\tilde{f}$ is to give maps $\tilde{f}_{ijk}:U_i \times_{S_k} W_{S_k} \to V_j$. Now if we additionally write $V_j \cong \Spec B_j$ then we derive the following chain of natural equivalences:
\begin{prooftree}
	\AxiomC{$\tilde{f}:X \times_S W_S \to Y$}
	\UnaryInfC{For all compatible $(i,j,k)$.\, $\tilde{f}_{ijk}:U_i \times_{S_k} W_{S_k} \to V_j$}
	\UnaryInfC{For all compatible $(i,j,k)$.\,$\Spec \varphi_{ijk}: \Spec A_i[\epsilon] \to \Spec B_j$}
	\UnaryInfC{For all compatible $(i,j,k)$.\,$\varphi_{ijk}:B_j \to A_i[\epsilon]$}
	\UnaryInfC{For all compatible $(i,j,k)$.\,$\partial_{ijk} \in \mathsf{Der}_{C_k}(B_j,A_i)$}
	\UnaryInfC{For all compatible $(i,j,k)$.\,$\tilde{\partial}_{ijk} \in \Mod{B_j}(\Kah{B_j}{C_k},\Und{B_j}(A_i))$}
	\UnaryInfC{For all compatible $(i,j,k)$.\,$\tilde{\partial}_{ijk}^{\sharp} \in \CAlg{B_j}(\Sym{B_j}(\Kah{B_j}{C_k}),A_i)$}
	\UnaryInfC{For all compatible $(i,j,k)$.\,$\Spec \tilde{\partial}_{ijk}^{\sharp}:U_i \to T_{V_j/S_k}$}\RightLabel{Glue the $\Spec \tilde{\partial}_{ijk}^{\sharp}$ over $i,j,k$}
	\UnaryInfC{$\hat{f}:X \to T_{Y/S}$}
\end{prooftree}
By running this argument in reverse we find that $(-) \times_S W_S \dashv T_{(-)/S}$, as was desired.
\end{proof}
\begin{corollary}\label{Cor: Section Tangent: Tangent functor pres limits}
Let $S$ be a scheme. Then $T_{(-)/S}$ is continuous.
\end{corollary}
\begin{proof}
This is immediate from $T_{(-)/S}$ being a right adjoint.
\end{proof}

Luckily for us, defining the natural transformations which give the bundle projection and zero sections of the tangent functor follow much as in Example \ref{Example: The tangent category of affine R schemes} and in particular in Lines (\ref{Eqn: Section Tangent: Tangent Structure of Affine Projection Map}) and (\ref{Eqn: Section Tangent: Tangent Structure of Affine Zero Map}). Because $\QCoh(X)$ is an Abelian category and $\RelSym{\Ocal_X}$ is a left adjoint, we again have that using the natural isomorphism $\Ocal_X \cong \RelSym{\Ocal_X}(0)$ together with the zero inclusion $\gnab$ at the sheaf of relative K{\"a}hlers gives a map
\[
\begin{tikzcd}
\Ocal_X \ar[r]{}{\cong} & \RelSym{\Ocal_X}(0) \ar[rr]{}{\RelSym{\Ocal_X}(\gnab_{\Omega})} & & \RelSym{\Ocal_X}\left(\Kah{X}{S}\right)
\end{tikzcd}
\]
in the category $\QCoh(X,\CAlg{\Ocal_X})$ of quasi-coherent sheaves of $\Ocal_X$-algebras; call this map $\rho_X$. Dually, writing the zero projection $\bang_{\Omega}:\Kah{X}{S} \to 0$, we also obtain a morphism
\[
\begin{tikzcd}
\RelSym{\Ocal_X}\left(\Kah{X}{S}\right) \ar[rr]{}{\RelSym{\Ocal_X}(\bang_{\Omega})} & & \RelSym{\Ocal_X}(0) \ar[r]{}{\cong} & \Ocal_X
\end{tikzcd}
\] 
in $\QCoh(X,\CAlg{\Ocal_X})$; call this map $z_X$. Applying the relative spectrum functor at $X$ gives us the morphisms which define the $X$-component of the projection and zero section transformations.
\begin{definition}\label{Defn: Section Tangent: Zero and Projection Transformations}
Let $S$ be a base scheme. The projection and zero section natural transformations
\[
\begin{tikzcd}
\Sch_{/S} \ar[rr, swap, bend right = 30, ""{name = D}]{}{\id} \ar[rr, bend left = 30, ""{name = U}]{}{T_{(-)/S}} & & \Sch_{/S}
\ar[from = U, to = D, Rightarrow, shorten <= 4pt, shorten >= 4pt]{}{p}
\end{tikzcd}\qquad
\begin{tikzcd}
	\Sch_{/S} \ar[rr, swap, bend right = 30, ""{name = D}]{}{T_{(-)/S}} \ar[rr, bend left = 30, ""{name = U}]{}{\id} & & \Sch_{/S}
	\ar[from = U, to = D, Rightarrow, shorten <= 4pt, shorten >= 4pt]{}{0}
\end{tikzcd}
\]
are defined by setting $p_X := \RelSpec{X}(\rho_X)$ and $0_X := \RelSpec{X}(z_X)$, respectively.
\end{definition}
\begin{proposition}\label{Prop: Section Tangent: Projection and zero are affine}
For any $S$-scheme $X$, the morphisms $p_X:T_{X/S} \to X$ and $0_X:X \to T_{X/S}$ are affine.
\end{proposition}
\begin{proof}
Definition \ref{Defn: Section Tangent: Zero and Projection Transformations} defines $p$ and $0$ by way of a relative spectrum functor, so an appeal to Corollary \ref{Cor: Section Background Scheme: Equiv of Cats for Relspec} gives the result.
\end{proof}

Recall from Corollary \ref{Cor: Section Background Scheme: Relative Spec of Relative Sym is functor into Ab Sch} that for any quasi-coherent sheaf $\Fscr$ on $X$, the scheme
\[
\mathbb{V}(\Fscr) := \RelSpec{X}\left(\widehat{\RelSym{\Ocal_X}}(\Fscr)\right)
\]
is a group scheme with unit section $0_{\Fscr}:X \to \mathbb{V}(\Fscr)$ given by applying $\RelSpec{X} \circ \RelSym{\Ocal_X}$ to the zero projection map $\bang_{\Fscr}:\Fscr \to 0$ in $\QCoh(X)$. Additionally, when we set $\Fscr = \Kah{X}{S}$ we find that we have natural isomorphisms
\[
\mathbb{V}(\Kah{X}{S}) \times_X \mathbb{V}(\Kah{X}{S}) \cong T_{X/S} \times_X T_{X/S} = (T_{X/S})_2
\]
of schemes in $\mathbf{Aff}_{/X}$. This essentially forces our description of the addition on $(T_{X/S})_2$: we use the group scheme structure on $\mathbb{V}(\Kah{X}{S}) = T_{X/S}$ as asserted by Corollary \ref{Cor: Section Background Scheme: Relative Spec of Relative Sym is functor into Ab Sch}.
\begin{definition}\label{Defn: Section Tangent: The addition on the tangent bundle}
Let $S$ be a scheme. For any scheme $X$ over $S$ we define the addition transformation
\[
\begin{tikzcd}
\Sch_{/S} \ar[rr, bend left = 30, ""{name = U}]{}{(T_{(-)/S})_2} \ar[rr, swap, bend right = 30, ""{name = D}]{}{T_{(-)/S}} & & \Sch_{/S}
\ar[from = U, to = D, Rightarrow, shorten <= 4pt, shorten >= 4pt]{}{\operatorname{add}}
\end{tikzcd}
\]
by defining the $X$-component of $\operatorname{add}$ to be given by taking the $X$-relative spectrum of the comultiplication map 
\[
\nabla_{\Omega}:\RelSym{\Ocal_X}\left(\Kah{X}{S}\right) \longrightarrow \RelSym{\Ocal_X}\left(\Kah{X}{S}\right) \otimes_{\Ocal_X} \RelSym{\Ocal_X}\left(\Kah{X}{S}\right)
\] 
given in Theorem \ref{Thm: Section Background Scheme: Hopf algebras for RelSym}, i.e., by setting
\[
\operatorname{add}_X := \RelSpec{X}(\nabla_{\Omega}).
\]
\end{definition}
\begin{proposition}\label{Prop: Section Tangent: Affine addition map}
The for any scheme $S$ and for any $S$-scheme $X$, the morphism $\operatorname{add}_X$ is affine.
\end{proposition}
\begin{proof}
This is proved mutatis mutandis to Proposition \ref{Prop: Section Tangent: Projection and zero are affine}.
\end{proof}

All that remains to define are the vertical lift $\ell$ and the canonical flip $c$. To go about this we take a much more explicit and ``computational'' approach than what has been taken so far in this paper. Recall from the map expressed in Line (\ref{Eqn: Section Tangent: Tangent Structure of Affine VertLift Map}) that when $S = \Spec R$ is affine, the vertical lift on affine schemes $\ell_A:T_{A/R} \to T^2_{A/R}$ is given by the spectrum of the map
\[
\hat{\ell}_A:\Sym{\Sym{A}(\Kah{A}{R})}\left(\Kah{\big[\Sym{A}(\Kah{A}{R})\big]}{R}\right) \longrightarrow \Sym{A}\left(\Kah{A}{R}\right)
\]
which is the $A$-algebra map defined by $\mathrm{d}a, \delta a \mapsto 0$ and $(\mathrm{d}\delta)a \mapsto \mathrm{d}a$ for all $a \in A$. To extend this to schemes carefully, we will use the structure of affine morphisms of schemes together with Lemmas \ref{Lemma: Background Scheme: Affine morphisms stable under composition} and \ref{Lemma: Section Background Scheme: Forward functoriality of relspec for affine maps}. In particular, we will use the observation that to define our morphism $\ell:T_{X/S} \to T^2_{X/S}$ it suffices to define a morphism
\[
\rotatebox[origin = c]{180}{$\ell$}:Q_{X}(T^2_{X/S}) \to Q_{X}(T_{X/S})
\]
of quasi-coherent sheaves of $\Ocal_X$-algebras; this is where Lemma \ref{Lemma: Section Background Scheme: Forward functoriality of relspec for affine maps} comes in.

Let us examine how the lemma to which we just referred may be used to aid us. For any scheme map $f:X \to S$ the structure map 
\[
p_{T_{X/S}} = (p \ast T)_{X}:T_{X/S}^2 \to T_{X/S}
\]
is affine by Proposition \ref{Prop: Section Tangent: Projection and zero are affine}, 
so since affine morphisms compose by Lemma \ref{Lemma: Background Scheme: Affine morphisms stable under composition}  with $p_X:T_{X/S} \to X$ also affine by Proposition \ref{Prop: Section Tangent: Projection and zero are affine}, every map in the diagram
\[
\begin{tikzcd}
T^2_{X/S} \ar[dr] \ar[rr]{}{(p \ast T)_X} & & T_{X/S} \ar[dl]{}{p_X} \\
 & X
\end{tikzcd}
\] 
is affine. Thus we can write $T_{X/S}$ as
\[
T_{X/S} \cong \RelSpec{X}\left(\RelSym{\Ocal_X}(\Kah{X}{S})\right).
\]
In particular, the quasi-coherent algebras $\Tscr$ on $X$ which corresponds to $T_{X/S}$ through the equivalences $Q_X$ and $\RelSpec{X}$ is the sheaf
\[
Q_X(T_{X/S}) = (p_X)_{\ast}\Ocal_{T_{X/S}} \cong \RelSym{\Ocal_X}\left(\Kah{X}{S}\right)
\]
and has the property that for all $U \subseteq X$
\[
\RelSym{\Ocal_X}\left(\Kah{X}{S}\right)(U) \cong (p_X)_{\ast}\Ocal_{T_{X/S}}(U) \cong \frac{\Ocal_X(U)[\mathrm{d}a:a \in \Ocal_X(U)]}{(\mathrm{d}(a+b) - \mathrm{d}a - \mathrm{d}b, \mathrm{d}(ab) - \mathrm{d}(a)b - a\mathrm{d}b:a,b \in \Ocal_X(U))}.
\]

To write $T_{X/S}^2$ as a scheme affine over $X$ in terms of quasi-coherent sheaves of $\Ocal_X$-algebras requires a little more effort and the use of Lemma \ref{Lemma: Section Background Scheme: Forward functoriality of relspec for affine maps}. Recall that by Lemma \ref{Lemma: Section Background Scheme: Forward functoriality of relspec for affine maps} applied to the map $p_X:T_{X/S} \to X$, we have a strictly commuting diagram
\[
\begin{tikzcd}
\mathbf{Aff}_{/T_{X/S}} \ar[rrr]{}{p_{X} \circ (-)} \ar[d, swap]{}{Q_{T_{X/S}}} & & & \mathbf{Aff}_{/X} \ar[d]{}{Q_X} \\
\QCoh(T_{X/S},\CAlg{\Ocal_{T_{X/S}}})^{\op} \ar[rrr, swap]{}{(p_X)_{\ast}} & & & \QCoh(X,\CAlg{\Ocal_X})^{\op}
\end{tikzcd}
\]
where, as before, $Q_{T_{X/S}}$ and $Q_{X}$ are the inverse equivalences of $\RelSpec{T_{X/S}}$ and $\RelSpec{X}$, respectively. Under this diagram, we trace that the scheme $T^2_{X/S}$ gets sent to the sheaf
\[
\left((p_X)_{\ast} \circ Q_{T_{X/S}}\right)\big(T_{X/S}^{2}\big) = \left(p_X \circ \left(p \ast T\right)_{X}\right)_{\ast}\Ocal_{T_{X/S}^{2}}.
\]
On opens $U \subseteq X$, we calculate that
\[
\left(\left(p_X \circ \left(p \ast T\right)_{X}\right)_{\ast}\Ocal_{T_{X/S}^{2}}\right)(U) \cong \frac{\Ocal_X(U)[\mathrm{d}a, \delta(a), (\mathrm{d}\delta)(a):a \in \Ocal_X(U)]}{I}
\]
where $I$ is the ideal generated by the equational identities expressed in Lines (\ref{Eqn: First of the T2 rels})  -- (\ref{Eqn: Fifth of the T2 rels}).

We now define $\rotatebox[origin = c]{180}{$\ell$}:(p_X \circ (p \ast T)_X)_{\ast}\Ocal_{T^2_{X/S}} \to (p_{X})_{\ast}\Ocal_{T_{X/S}}$ by setting
\begin{prooftree}
	\AxiomC{$\rotatebox[origin = c]{180}{$\ell$}_U:((p_X \circ (p \ast T)_X)_{\ast}\Ocal_{T^2_{X/S}})(U) \to (p_{X})_{\ast}\Ocal_{T_{X/S}}(U)$}\doubleLine
	\UnaryInfC{$\hat{\ell}_{\Ocal_X(U)}:\Ocal_X(U)[\mathrm{d}a,\delta(a), (\mathrm{d}\delta)a]/I \to \Sym{\Ocal_X(U)}(\Kah{X}{S}(U))$}
\end{prooftree}
for all opens $U \subseteq \lvert X \rvert$. Because the morphisms $\hat{\ell}_A$ are natural in commutative rings $A$, it follows that the map $\rotatebox[origin = c]{180}{$\ell$}$ is a morphism of sheaves of $\Ocal_X$-algebras on $X$. This in turn implies that $\rotatebox[origin = c]{180}{$\ell$}$ is a map in $\QCoh(X,\CAlg{\Ocal_X})^{\op}$. Applying the relative spectrum functor allows us to define the vertical lift (once we prove that $\ell$ is natural in $X$).
\begin{lemma}\label{Lemma: Section Tangent: Vertical lift is natural}
Let $S$ be a scheme and let $\ell_X:T^{2}_{X/S} \to T_{X/S}$ be the morphism of schemes defined by
\[
\ell_X := \RelSpec{X}\left(\rotatebox[origin = c]{180}{$\ell$}_X\right)
\]
where $\rotatebox[origin = c]{180}{$\ell$}_X$ is the morphism of sheaves defined directly prior to the statement of the lemma. Then $\ell$ is a natural transformation:
\[
\begin{tikzcd}
\Sch_{/S} \ar[rr, bend left = 30, ""{name=  U}]{}{T_{(-)/S}} \ar[rr, swap, bend right = 30, ""{name = D}]{}{T_{(-)/S}^2} & & \Sch_{/S}
\ar[from = U, to = D, Rightarrow, shorten <= 4pt, shorten >= 4pt]{}{\ell}
\end{tikzcd}
\]
\end{lemma}
\begin{proof}
Let $f:X \to Y$ be a morphism of $S$-schemes with structure maps $\nu_X:X \to S$ and $\nu_Y:Y \to S$. Give $S = \lbrace S_i \; | \; i \in I \rbrace$ an open affine cover, $\lbrace V_j \; | \; j \in J_i \rbrace$ an open affine cover of open subscheme $\nu_Y^{-1}(S_i)$ of $Y$, and let $\lbrace U_k \; | \; k \in K_j \rbrace$ be an open affine cover of the open subscheme $f^{-1}(V_j)$ of $X$. Then we have that $Y$ admits open affine cover $\lbrace V_j \; | \; j \in J_i \rbrace$, $X$ admits open affine cover $\lbrace U_k \; | \; i \in I, j \in J_i, k \in K_j \rbrace$, and the open subscheme $f^{-1}(\nu_Y^{-1}(S_i)) \cong \nu_X^{-1}(S_i)$ of $X$ has affine open cover $\lbrace U_k \; | \; j \in J_i, k \in J_k \rbrace$. Moreover, based on how these covers are assembled, we also have that
\[
\begin{tikzcd}
U_k \ar[dr]{}{\nu_X} \ar[rr] \ar[dd, swap]{}{f|_{U_k}} & & X  \ar[dr]{}{\nu_X} \\
 & S_i \ar[rr] & & S \\
V_j \ar[rr] \ar[ur, swap]{}{\nu_Y} & & Y \ar[ur, swap]{}{\nu_Y}
\ar[from = 1-3, to = 3-3, crossing over, near start]{}{f}
\end{tikzcd}
\]
commutes for every compatible choice of $i \in I$, $j \in J_i$, and $k \in K_j$. Now, to show that $\ell$ is natural it suffices to prove that over the affines $S_i, V_j$, and $U_k$ we have an induced commuting diagram
\[
\begin{tikzcd}
T_{U_k/S_i} \ar[r]{}{Tf|_{U_k}} \ar[d, swap]{}{\ell_{U_k}} & T_{V_j/S_i} \ar[d]{}{\ell_{V_j}} \\
T^2_{U_k/S_i} \ar[r, swap]{}{T^2f|_{U_k}} & T^2_{V_j/S_i}
\end{tikzcd}
\]
as taking the filtered colimit\footnote{Which exists in $\Sch_{/S}$ because all maps $S_i \to S$, $V_j \to Y, U_k \to X$, are affine and open immersions while $U_k \to V_j$, $V_j \to S_i$ affine for all $i, j ,k$.} gives the desired naturality diagram for $\ell$. However, by construction the square above commutes if and only if the corresponding square
\[
\begin{tikzcd}
\frac{\Ocal_Y(V_j)[\mathrm{d}y, \delta y, (\mathrm{d}\delta)y:y \in \Ocal_Y(V_j)]}{I_{2,Y}} \ar[d, swap]{}{\rotatebox[origin = c]{180}{$\ell$}_{\Ocal_Y(V_j)}} \ar[rrr]{}{(T^2f|_{U_k})^{\sharp}} & & & \frac{\Ocal_X(U_k)[\mathrm{d}x, \delta x, (\mathrm{d}\delta)x:x \in \Ocal_X(U_k)]}{I_{2,X}}\ar[d]{}{\hat{\ell}_{\Ocal_X(U_k)}} \\
\frac{\Ocal_Y(V_j)[\mathrm{d}y:y \in \Ocal_Y(V_j)]}{I_Y} \ar[rrr, swap]{}{(Tf|_{U_k})^{\sharp}} & & & \frac{\Ocal_X(U_k)[\mathrm{d}x:x \in \Ocal_X(U_k)]}{I_{X}}
\end{tikzcd}
\]
commutes in $\CAlg{\Ocal_S(S_i)}$ for all $i \in I$, $j \in J_i$, and $k \in K_j$; note that we write $I_{2,X}$ and $I_{2,Y}$ for the corresponding ideals defining the rings $\Gamma(T_{U_k/S_i}^2)$ and $\Gamma(T_{V_j/S_i}^2)$, respectively, while $I_X$ and $I_Y$ are the ideals defining the rings $\Gamma(T_{U_k/S_i})$ and $\Gamma(T_{V_j/S_i})$. However, the commutativity of the square above follows is precisely from the naturality of $\hat{\ell}$ on $\CAlg{\Ocal_S(S_i)}$ by Example \ref{Example: The tangent category of affine R schemes}.
\end{proof}

\begin{definition}\label{Defn: Section Tangent: Vertical Lift over Schemes}
Let $S$ be a scheme. The vertical lift $\ell:T_{(-)/S} \Rightarrow  T^2_{(-)/S}$ is the natural transformation constructed in Lemma \ref{Lemma: Section Tangent: Vertical lift is natural}
\end{definition}
Immediate from definition and construction is the following lemma which indicates that the vertical lift $\ell$ is an affine morphism.
\begin{corollary}\label{Cor: Section Tangent: Vertical lift is affine}
For any scheme $X$, the vertical lift $\ell_X:T_{X/S} \to T^2_{X/S}$ is affine.
\end{corollary}
\begin{proof}
This is immediate from Corollary \ref{Cor: Section Background Scheme: Equiv of Cats for Relspec} and that by definition $\ell_X = \RelSpec{X}(\rotatebox[origin = c]{180}{$\ell$}_X)$.
\end{proof}

We now construct the canonical flip $c:T^{2}_{(-)/S} \Rightarrow T^{2}_{(-)/S}$ of the tangent structure on $\Sch_{/S}$. Begin by recalling from Example \ref{Example: The tangent category of affine R schemes} that in the category $\CAlg{R}^{\op} \simeq \mathbf{AffSch}_{/R}$ and for any commutative $R$-algebra $A$, the canonical flip $c_A:T^2_{A/R} \to T^2_{A/R}$ is the opposite of the map
\[
\hat{c}_A:\frac{A[\mathrm{d}a, \delta a, (\mathrm{d}\delta)a:a \in A]}{I_{2,A}} \longrightarrow \frac{A[\mathrm{d}a, \delta a, (\mathrm{d}\delta)a]}{I_{2,A}}
\]
defined by $a \mapsto a, (\mathrm{d}\delta)a \mapsto (\mathrm{d}\delta)a, \mathrm{d}a \mapsto \delta a$, and $\delta a \mapsto \mathrm{d}a$ for all $a \in A$. Our goal now is to schemify these maps in the same way as we schemified the map $\hat{\ell}$: by first, for a base $S$-scheme $X$, extending the $\hat{c}$ to a map $\reflectbox{$c$}:\Tscr^2 \to \Tscr^2$ of quasi-coherent sheaves of $\Ocal_X$-algebras and then applying the relative spectrum functor.

Let us make the last sentence, namely about the way in which we extend the algebra maps $\hat{c}$ into the scheme maps $c$, more explicit. Recall from our discussion prior to the statement of Lemma \ref{Lemma: Section Tangent: Vertical lift is natural} that the structure sheaf $(p_X \circ (p \circ T)_X)_{\ast}\Ocal_{T^2_{X/S}}$ is a quasi-coherent sheaf of $\Ocal_X$-algebras whose corresponding scheme affine over $X$ satifies
\[
\RelSpec{X}\left((p_X \circ (p \circ T)_X)_{\ast}\Ocal_{T^2_{X/S}}\right) \cong T^{2}_{X/S}.
\]
Additionally, the structure sheaf $(p_X \circ (p \circ T)_X)_{\ast}\Ocal_{T^2_{X/S}}$ also ha the property that for all open subschemes $U$ of $X$ there is a natural isomorphism
\[
\left((p_X \circ (p \circ T)_X)_{\ast}\Ocal_{T^2_{X/S}}\right)(U) \cong \frac{\Ocal_X(U)[\mathrm{d}u, \delta u, (\mathrm{d}\delta)u:u \in \Ocal_X(U)]}{I_{2,\Ocal_X(U)}}.
\]
Let $\gamma_U$ denote the isomorphisms
\[
\gamma_U:\left((p_X \circ (p \circ T)_X)_{\ast}\Ocal_{T^2_{X/S}}\right)(U) \xrightarrow{\cong} \frac{\Ocal_X(U)[\mathrm{d}u, \delta u, (\mathrm{d}\delta)u:u \in \Ocal_X(U)]}{I_{2,\Ocal_X(U)}}
\] 
We use these natural isomorphisms to make the assignment
\begin{prooftree}
	\AxiomC{$\reflectbox{$c$}_X:(p_X \circ (p \circ T)_X)_{\ast}\Ocal_{T^2_{X/S}} \Rightarrow (p_X \circ (p \circ T)_X)_{\ast}\Ocal_{T^2_{X/S}}$}\doubleLine
	\UnaryInfC{For $U$ open.\, $(\reflectbox{$c$}_X)_U := \gamma_U^{-1} \circ \hat{c}_{\Ocal_X(U)} \circ \gamma_U$}
\end{prooftree}
and now prove that this indeed gives a morphism of quasi-coherent sheaves over $X$.
\begin{lemma}\label{Lemma: Section Tangent: Canonical flip locally is qcoh morphism}
Let $S$ be a scheme and let $X$ be an $S$-scheme. Then the map $\reflectbox{$c$}_X$ is a morphism of qusi-coherent sheaves on $X$.
\end{lemma}
\begin{proof}
Because the natural transformations $\gamma$ are natural in $X$, to verify the commutativity of
\[
\begin{tikzcd}
((p_X \circ (p \circ T)_X)_{\ast}\Ocal_{T^2_{X/S}})(U) \ar[d, swap]{}{((p_X \circ (p \circ T)_X)_{\ast}\Ocal_{T^2_{X/S}})(V \subseteq U)} \ar[rrr]{}{(\reflectbox{$c$}_X)_U} & & & ((p_X \circ (p \circ T)_X)_{\ast}\Ocal_{T^2_{X/S}})(U) \ar[d]{}{((p_X \circ (p \circ T)_X)_{\ast}\Ocal_{T^2_{X/S}})(V \subseteq U)} \\
((p_X \circ (p \circ T)_X)_{\ast}\Ocal_{T^2_{X/S}})(V) \ar[rrr, swap]{}{(\reflectbox{$c$}_X)_V} & & & ((p_X \circ (p \circ T)_X)_{\ast}\Ocal_{T^2_{X/S}})(V)
\end{tikzcd}
\] 
it suffices to prove that for all opens $V \subseteq U$ of $X$ there is a corresponding commuting square
\[
\begin{tikzcd}
\frac{\Ocal_X(U)[\mathrm{d}u, \delta u, (\mathrm{d}\delta)u:u \in \Ocal_X(U)]}{I_{2,\Ocal_X(U)}} \ar[d,swap]{}{T^2\rho_{UV}} \ar[rrr]{}{\hat{c}_{\Ocal_X(U)}} & & & \frac{\Ocal_X(U)[\mathrm{d}u, \delta u, (\mathrm{d}\delta)u:u \in \Ocal_X(U)]}{I_{2,\Ocal_X(U)}} \ar[d]{}{T^2\rho_{UV}} \\
\frac{\Ocal_X(V)[\mathrm{d}v, \delta v, (\mathrm{d}\delta)v:v \in \Ocal_X(V)]}{I_{2,\Ocal_X(V)}} \ar[rrr, swap]{}{\hat{c}_{\Ocal_X(V)}} & & & \frac{\Ocal_X(V)[\mathrm{d}v, \delta v, (\mathrm{d}\delta)v:v \in \Ocal_X(V)]}{I_{2,\Ocal_X(V)}}
\end{tikzcd}
\]
where the maps $T^2\rho_{UV}$ are those which are given on generates by $u \mapsto \rho_{UV}(u)$, $\mathrm{d}u \mapsto \mathrm{d}(\rho_{UV}u)$, $\delta u \mapsto \delta(\rho_{UV}u)$, and $(\mathrm{d}\delta)u \mapsto (\mathrm{d}\delta)(\rho_{UV}u)$. However, the naturality of the $\hat{c}$ as given in Example \ref{Example: The tangent category of affine R schemes} imply that the diagram does commute and hence that $\reflectbox{$c$}_X$ is a morphism of quasi-coherent sheaves of $\Ocal_X$-algebras. 
\end{proof}

We now are in position to define $c$ as a natural transformation on $\Sch_{/S}$. As such, we first prove that the assignment $X \mapsto \RelSpec{X}(\reflectbox{$c$}_X)$ is indeed a natural transformation on $\Sch_{/S}$.

\begin{lemma}\label{Lemma: Section Tangent: Canonical flip is natural transf on schemes}
Let $S$ be a scheme. Then the assignment
\[
X \mapsto \RelSpec{X}(\reflectbox{$c$}_X)
\]
on $S$-schemes $X$ determines a natural transformation:
\[
\begin{tikzcd}
\Sch_{/S}\ar[rr, bend left = 30, ""{name = U}]{}{T^{2}_{(-)/S}} \ar[rr, swap, bend right = 30, ""{name =D}]{}{T^{2}_{(-)/S}} & & \Sch_{/S}
\ar[from = U, to = D, Rightarrow, shorten <= 4pt, shorten >= 4pt]{}{c}
\end{tikzcd}
\]
\end{lemma}
\begin{proof}
First, one needs to check that $c$ types as a scheme morphism $T^2_{X/S} \to T^2_{X/S}$; however, this is immediate by construction of $\reflectbox{$c$}_X$ as an endomorphism defined on $(p_X \circ (p \ast T)_X)\Ocal_{T^2_{X/S}}$. Second one needs to show that for any morphism of $S$-schemes $f:X \to Y$ that the diagram
\[
\begin{tikzcd}
T^{2}_{X/S} \ar[r]{}{c_X} \ar[d, swap]{}{T^2f} & T^{2}_{X/S} \ar[d]{}{T^2f} \\
T^{2}_{Y/S} \ar[r, swap]{}{c_Y} & T^2_{Y/S}
\end{tikzcd}
\]
commutes. 

To prove the commutativity of the diagram above let $\lbrace S_i \to S \; | \; i \in I \rbrace,$ $\lbrace V_j \to Y \; | \; i \in I, j \in J_i\rbrace$, and $\lbrace U_k \to X \; | \; i \in I, j \in J_i, k \in K_j \rbrace$ be open affine covers of $S$, $Y$, and $X$, repsectively, with the properties that given commuting triangle
\[
\begin{tikzcd}
X \ar[rr]{}{f} \ar[dr, swap]{}{\nu_X} & & Y \ar[dl]{}{\nu_Y} \\
 & S
\end{tikzcd}
\]
we have induced commuting triangular prism:
\[
\begin{tikzcd}
	U_k \ar[dr]{}{\nu_X} \ar[rr] \ar[dd, swap]{}{f|_{U_k}} & & X  \ar[dr]{}{\nu_X} \\
	& S_i \ar[rr] & & S \\
	V_j \ar[rr] \ar[ur, swap]{}{\nu_Y} & & Y \ar[ur, swap]{}{\nu_Y}
	\ar[from = 1-3, to = 3-3, crossing over, near start]{}{f}
\end{tikzcd}
\]
Note that such covers may be constructed as in the proof of Lemma \ref{Lemma: Section Tangent: Vertical lift is natural}. Arguing by taking the filtered colimit over all $i, j, k$ as in the proof of Lemma \ref{Lemma: Section Tangent: Vertical lift is natural} we are thus reduced to proving the commutativity of the diagram
\[
\begin{tikzcd}
	T^{2}_{U_k/S_i} \ar[r]{}{c_{U_k}} \ar[d, swap]{}{T^2f|_{U_k}} & T^{2}_{U_k/S_i} \ar[d]{}{T^2f|_{U_k}} \\
	T^{2}_{V_j/S_i} \ar[r, swap]{}{c_{V_j}} & T^{2}_{V_j/S_i}
\end{tikzcd}
\]
of $S_i$-schemes. However, as every map in sight is affine and since the maps $\reflectbox{$c$}_{(-)}$ invoke the transformations $\gamma$, this diagram is equivalent to the corresponding diagram
\[
\begin{tikzcd}
\frac{\Ocal_Y(V_j)[\mathrm{d}v, \delta v, (\mathrm{d}\delta)v:v \in \Ocal_Y(V_j)]}{I_{2,\Ocal_Y(V_j)}} \ar[d, swap]{}{(T^2f|_{U_k})^{\sharp}} \ar[rrr]{}{\hat{c}_{\Ocal_Y(V_j)}} & & & \frac{\Ocal_Y(V_j)[\mathrm{d}v, \delta v, (\mathrm{d}\delta)v:v \in \Ocal_Y(V_j)]}{I_{2,\Ocal_Y(V_j)}} \ar[d]{}{(T^2f|_{U_k})^{\sharp}} \\
\frac{\Ocal_X(U_k)[\mathrm{d}u, \delta u, (\mathrm{d}\delta)u:u \in \Ocal_X(U_k)]}{I_{2,\Ocal_X(U_k)}} \ar[rrr, swap]{}{\hat{c}_{\Ocal_X(U_k)}} & & & \frac{\Ocal_X(U_k)[\mathrm{d}u, \delta u, (\mathrm{d}\delta)u:u \in \Ocal_X(U_k)]}{I_{2,\Ocal_X(U_k)}}
\end{tikzcd}
\]
of $\Ocal_S(S_i)$-algebras. This in turn follows from Example \ref{Example: The tangent category of affine R schemes} (in the case of $R = \Ocal_{S}(S_i)$) and hence allows us to deduce that $c$ is a natural transformation.
\end{proof}

\begin{definition}\label{Defn: Section Tangent: Canonical Flip}
The canonical flip $c$ on $\Sch_{/S}$ is the natural transformation $c$ constructed in Lemma \ref{Lemma: Section Tangent: Canonical flip is natural transf on schemes}.
\end{definition}
We, as in the case of the other structure maps we've introduced, may deduce that the canonical flip $c_X$ is an affine morphism for every $S$-scheme $X$.
\begin{corollary}\label{Cor: Section Tangent: Canonical flip is affine}
Let $S$ be a scheme and let $X$ be an $S$-scheme. Then the morphism $c_X:T^{2}_{X/S} \to T^{2}_{X/S}$ is affine.
\end{corollary}
\begin{proof}
This is immediate from the fact that $c = \RelSpec{X}(\reflectbox{$c$}_X)$ and from Corollary \ref{Cor: Section Background Scheme: Equiv of Cats for Relspec}.
\end{proof}

We now have all the necessary ingredients to prove that $\Sch_{/S}$ is a tangent category. We will also be able to deduce immediately following this that every structure morphism which appears in the definition of $\Sch_{/S}$ is also an affine morphism of schemes.

\begin{Theorem}\label{Thm: Section Tangents: The Tangent Category of Schemes}
The category $\Sch_{/S}$ with tangent functor $T_{(-)/S}$ given in Definition \ref{Defn: Section Tangent: Tangent functor for schemes}, with bundle projection $p$ and zero section $0$ given in Definition \ref{Defn: Section Tangent: Zero and Projection Transformations}, addition $\operatorname{add}$ as given in Definition \ref{Defn: Section Tangent: The addition on the tangent bundle}, with vertical lift $\ell$ given in Definition \ref{Defn: Section Tangent: Vertical Lift over Schemes}, and with canonical flip given as in Definition \ref{Defn: Section Tangent: Canonical Flip} is a tangent category.
\end{Theorem}
\begin{proof}
We give an axiom-by-axiom verification of Definition \ref{Defn: Section Tangent: Tangent category} (which defines tangent categories).
\begin{enumerate}
	\item The existence of the tangent functor and the tangent bundle projection are precisely described in Definitions \ref{Defn: Section Tangent: Tangent functor for schemes} and \ref{Defn: Section Tangent: Zero and Projection Transformations}. Furthermore, as $\Sch_{/S}$ is finitely continuous and as $T_{(-)/S}$ is continuous by Corollary \ref{Cor: Section Tangent: Tangent functor pres limits} we also know that every finite wide pullback of $p$ against itself exists and is preserved by all powers of $T$.
	\item The existence of the natural transformations $0$ and $\operatorname{add}$ are given in Definitions \ref{Defn: Section Tangent: Zero and Projection Transformations} and \ref{Defn: Section Tangent: The addition on the tangent bundle}. That this makes each bundle $T_{X/S}$ into an internal Abelian group in $\Sch_{/S}$ follows from Corollary \ref{Cor: Section Background Scheme: Hopf algebra result in opposite land}.
	\item The existence of the vertical lift is Definition \ref{Defn: Section Tangent: Vertical Lift over Schemes}. To establish the necessary coherences it must satisfy we recognize that because $\ell$ is affine, the given identities may be established affine-locally along an affine open cover of $X$. However, this then immediately follows from Example \ref{Example: The tangent category of affine R schemes}.
	\item The existence of the canonical flip is Definition \ref{Defn: Section Tangent: Canonical Flip}. To establish the necessary coherences it must satisfy we recognize that because $c$ is affine, the given identities may be established affine-locally along an affine open cover of $X$. However, this then immediately follows from Example \ref{Example: The tangent category of affine R schemes}.
	\item The identities required of Item (5) follow from the fact that each such identity holds affine-locally over every $S$-scheme $X$.
	\item Once again, because every map in sight is affine it suffices to prove that the diagram is an equalizer affine-locally over $X$ and glue. However, this holds also because of Example \ref{Example: The tangent category of affine R schemes}.
\end{enumerate}
\end{proof}

It is worth noting that these tangent structures on $\Sch_{/S}$ are pseudofunctorial in the schemes $S$ in the sense that in the relevant $2$-category $\fTan_{\operatorname{str}}$ of tangent categories, morphisms $f:X \to Y$ of schemes lift to (strong) morphisms of tangent categories $\Sch_{/Y} \to \Sch_{/X}$. To do so, we need to define morphisms of tangent categories as well as the $2$-category $\fTan$ of tangent categories.

\begin{definition}[{\cite[Definition 2.7]{GeoffRobinDiffStruct}}]\label{Defn: Tangent morphism}
	A \emph{morphism of tangent categories} $(\Cscr,\Tbb) \to (\Dscr,\Sbb)$ is given by a pair $(F,\alpha)$ where $F:\Cscr \to \Dscr$ is a functor and $\alpha$ is a natural transformation
	\[
	\begin{tikzcd}
		\Cscr \ar[r, ""{name = U}]{}{T} \ar[d, swap]{}{F} & \Cscr \ar[d]{}{F} \\
		\Dscr \ar[r, swap, ""{name = D}]{}{S} & \Dscr
		\ar[from = U, to = D, Rightarrow, shorten <= 4pt, shorten >= 4pt]{}{\alpha}
	\end{tikzcd}
	\]
	which makes five diagrams (omitted here) which indicate the ways in which $\alpha$ mediates between the structure transformations $(p,0,\operatorname{add},\ell,c)$ for $(\Cscr,\Tbb)$ and $(\Dscr,\Sbb)$ commute. If $\alpha$ is an isomorphism, then $(F,\alpha)$ is said to be a \emph{strong tangent morphism} while if $\alpha$ is the identity then $(F,\alpha)$ is said to be a \emph{strict tangent morphism}.
\end{definition}
Before discussing tangent morphisms directly, let us define the $2$-categories of tangent categories we consider directly in this paper. While we will not discuss the $2$-cells meaningfully, it is worth noting that they are not simply natural transformations $\rho:F \Rightarrow G$ between the functor components of tangent morphisms $(F,\alpha)$ and $(G,\beta)$. Instead they are natural transformations $\rho:F \Rightarrow G$ \emph{which also} interact well with the distributors via a whiskering condition spelled out in \cite{GarnerEmbeddingTanCat}.
\begin{definition}
Define the $2$-category $\fTan$ of tangent categories as follows:
\begin{itemize}
	\item Objects: Tangent categories $(\Cscr,\Tbb)$.
	\item $1$-cells: Tangent morphisms $(F,\alpha)$.
	\item $2$-cells: Tangent transformations of \cite{GarnerEmbeddingTanCat}.
	\item $1$-cell Composition: $(G,\beta) \circ (F,\alpha) := (G \circ F, (\beta \ast \id_F) \circ (\id_G \ast \alpha))$.
	\item $2$-cell Vertical Composition: As in $\fCat$.
	\item $2$-cell Horizontal Composition: As in $\fCat$.
\end{itemize}
\end{definition}
\begin{definition}
Define the full sub-$2$-category $\fTan_{\operatorname{str}}$ of $\fTan$ to be the one generated by taking $1$-cells to be only strong tangent morphisms, i.e., by taking only tangent morphisms $(F,\alpha)$ for which $\alpha$ is a natural isomorphism.
\end{definition}

It is straightforward to find examples of tangent morphisms in the literature and so we omit a general investigation into such a topic. Instead we will focus our attention to schemes themselves for the sake of this paper and show that for any morphism $f:X \to Y$ of $S$-schemes, the pullback functor $f^{\ast}$ induces a strong tangent morphism between slice categories. While we show this for schemes only, the interested reader should consult \cite[Sections 6, 7]{PronkVooys} for many examples of strong tangent morphisms of tangent categories.

\begin{proposition}[{cf.~ \cite[Proposition 6.15]{PronkVooys}}]\label{Prop: Pullback functor strong tangent for schemes}
For any morphism $f:X \to Y$ of schemes, the functor
\[
f^{\ast}:\Sch_{/Y} \to \Sch_{/X}
\]
induces a strong tangent morphism whose distributor is the natural isomorphism
\[
T_f:T_{(X \times_Y Z)/X} \xrightarrow{\cong} T_{Z/Y} \times_Y X
\]
of \cite[Line 16.5.12.2]{EGA44}.
\end{proposition}
\begin{proof}[Sketch]
This is proved in \cite[Proposition 6.15]{PronkVooys}. The main idea is to use the fact that the isomorphisms $T_f$ exist (a conceptual proof of the existence of this isomorphism in terms of only the tangent-categorical data is given in \cite[Proposition 4.1.14]{JSMeMapFlavoursInTanCats}) while the commutativity of $T_f$ with the various structure transformations may be checked affine-locally as in \cite[Lemmas 6.10 -- 6.14]{PronkVooys}.
\end{proof}

It is somewhat tedious to prove, but important, that these tangent morphisms assemble to a pseudofunctor defined on $\Sch_{/S}^{\op}$ which takes values not merely in the $2$-category $\fCat$, but instead in the $2$-category of tangent categories and strong tangent morphisms.

\begin{proposition}[{cf.~ \cite[Proposition 6.16]{PronkVooys}}]\label{Prop: Section Tangent: PSeudofunctor in Tan}
For any base scheme $S$, there is a slice category pseudofunctor
\[
\mathsf{Sl}_{S}:\Sch_{/S}^{\op} \to \fTan_{\operatorname{str}}
\]
given by sending schemes $X$ to their slice category tangent structures $(\Sch_{/X}, T_{(-)/X})$ and given by sending morphisms $f:X \to Y$ to the strong tangent morphism $(f^{\ast},T_f)$ of Proposition \ref{Prop: Pullback functor strong tangent for schemes}.
\end{proposition}
\begin{proof}[Sketch]
This is proved in \cite[Proposition 6.16]{PronkVooys}. 
\end{proof}

\subsection{A Tangent Structure for the Opposite Category of Schemes}
With the fact that the tangent functor $T_{(-)/S} = [S[\epsilon],-]$ is a right adjoint, we can use \cite[Proposition 5.17]{GeoffRobinDiffStruct} to construct the dual tangent structure on $\Sch_{/S}^{\op}$ in formally the same way that Remark \ref{Remark: Dual Tangent Structure} built the tangent structure on $\CAlg{R}^{\op}$. We will see also that this tangent structure on $\Sch_{/S}^{\op}$ looks like a scheme-ified version of the ring of dual numbers tangent structure on $\Sch_{/S}$.

\begin{proposition}\label{Prop: Section Tangents: Dual Numbers for Schemes Tangents}
The functor
\[
\left((-) \times_S W_S\right):\Sch_{/S}^{\op} \to \Sch_{/S}^{\op}
\]
induces a tangent structure on $\Sch_{/S}^{\op}$.
\end{proposition}
\begin{proof}
Begin by observing that we already know that $T_{(-)/S} = [W_S,-]:\Sch_{/S} \to \Sch_{/S}$ is a right adjoint with
\[
\left((-) \times_S W_S\right):\Sch_{/S} \to \Sch_{/S}
\]
the corresponding left adjoint. A straightforward computation shows that for any $S$-scheme $\nu_X:X \to S$ there is a natural isomorphism
\begin{align*}
X \times_S W_S &\cong X \times_S \underline{\Spec}_S(\Ocal_S[\epsilon]) \cong \underline{\Spec}_X\left(\Ocal_X \otimes_{\nu_X^{-1}\Ocal_S} \nu_X^{-1}\Ocal_S[\epsilon] \right) \\
&\cong \underline{\Spec}_X(\Ocal_X[\epsilon]) =: W_X
\end{align*}
and so $(-) \times_S W_S \cong W_{(-)}$. Now write
\[
W_{X,n} := \underline{\Spec}_X\left(\underbrace{\Ocal_X[\epsilon] \times_{\Ocal_X} \cdots \times_{\Ocal_X} \times_{\Ocal_X} \Ocal_X[\epsilon]}_{n\,\text{copies of}\,\Ocal_X[\epsilon]}\right)
\]
and note that for any $S$-scheme $Y$,
\[
\underbrace{\Sch_{/S}(W_X,Y) \times_{\Sch_{/S}(X,Y)} \cdots \times_{\Sch_{/S}(X,Y)} \Sch_{/S}(W_X,Y)}_{n\,\text{copies of}\,\Sch_{/S}(W_X,Y)} \cong \Sch_{/S}(W_{X,n},Y).
\]
Becasue the $W_{X,n}$ vary functorially in $X$, there is a functor $W_{(-),n}:\Sch_{/S} \to \Sch_{/S}$ which admits natural isomorphisms
\begin{align*}
\Sch_{/S}(X,T_nY) &\cong \Sch_{/S}(X,T_{Y/S}) \times_{\Sch_{/S}(X,Y)} \cdots \times_{\Sch_{/S}(X,Y)} \Sch_{/S}(X,T_{Y/S}) \\
&\cong \Sch_{/S}(X \times_S W_S,Y) \times_{\Sch_{/S}(X,Y)} \cdots \times_{\Sch_{/S}(X,Y)} \Sch_{/S}(X \times_S W_S, Y) \\
&\cong \Sch_{/S}(W_X,Y) \times_{\Sch_{/S}(X,Y)} \cdots \times_{\Sch_{/S}(X,Y)} \Sch_{/S}(W_X,Y) \\
&\cong \Sch_{/S}(W_{X,n},Y)
\end{align*}
so indeed there are adjoints
\[
\begin{tikzcd}
\Sch_{/S} \ar[rr, bend right = 20, swap, ""{name = R}]{}{T_n} & & \Sch_{/S} \ar[ll, bend right = 20, swap, ""{name = L}]{}{W_{(-),n}}
\ar[from = L, to = R, symbol = \dashv]
\end{tikzcd}
\]
for all $n \in \N$. Consequently a routine application of \cite[Proposition 5.17]{GeoffRobinDiffStruct} shows that the functor $W_{(-)}^{\op}:\Sch_{/S}^{\op} \to \Sch_{/S}^{\op}$ induces a tangent structure on $\Sch_{/S}^{\op}$.
\end{proof}
Further examination of this dual tangent structure on $\Sch_{/S}^{\op}$ is left for future work by the author in a different paper. For instance, it should be the case that there is an equivalence
\[
\DBun_{\Sch^{\op}}(X) \overset{!}{\simeq} \QCoh(X),
\]
but to establish this is quite technical\footnote{One major issue which arises is that when one looks for pullbacks in $\Sch_{/S}^{\op}$ one equivalently looks for pushouts in $\Sch_{/S}$. The core problem is that the category $\Sch_{/S}$ fails to admit coequalizers (and hence pushouts) in full generality --- a standard example is that one cannot glue two copies of the affine lines $\Abb_K^1 = \Spec K[t]$ at only the generic point $\eta = \Spec K(t)$. As such, great care must be taken in order to ensure that the diagrams one draws actually exist in $\Sch_{/S}$ (and hence in $\Sch_{/S}^{\op}$).} and worth an in-depth discussion in its own right. Additionally, proving such a result requires in-depth manipulations of nil-square extensions of rings, first order deformations of schemes, and the ways in which they interact with the results of \cite[Section 3]{GeoffJSDiffBunComAlg}.

\section{Reconstructing Quasi-Separated Schemes from Differential Bundles}\label{Section: Appies ofTan on Sch}

In this final section of the paper we will make some observations regarding how one may use the tangent sturucture on $\Sch_{/S}$, or some related tangent-sub-categories, to deduce both new observations provide new perspecitves on existing topics in (differential) algebraic geometry. In the first subsection below, we will show that the pseudofunctor $\DBun(-)$ provides a complete invariant for the quasi-separated schemes in the sense that for said schemes, $\DBun(X) \simeq \DBun(Y)$ if and only if $X \cong Y$; while the statement of this result is new to this article, it is essentially a restatement of the Gabriel-Rosenberg Reconstruction Theorem. As such, we will first recall that result before providing the background necessary to describe the category $\qsSch_{/S}$ of quasi-separated $S$-schemes. Afterwords, we will prove that the category $\qsSch_{/S}$ is a strict tangent subcategory of $\Sch_{/S}$ and that there is an equality $\DBun_{\qsSch}(X) = \DBun_{\Sch}(X)$ whenever $X \to S$ is quasi-separated.

\subsection{Differential Bundles as a Complete Invariant for Quasi-Separated Schemes}\label{Subsection: DBun as invariant}

A result of paramount importance for noncommutative algebraic geometry is the much-celebrated Gabriel-Rosenberg Reconstruction Theorem which says that equivalences between categories of quasi-coherent sheaves give rise to isomorphisms of quasi-separated schemes and vice-versa. In Morita-theortic terms, it says that Morita equivalence for quasi-separated schemes coincides with isomorphisms of quasi-separated schemes\footnote{The Gabriel-Rosenberg Reconstruction Theorem also provides a wonderful mosquito nuke of a proof that, for commutative rings $A$ and $B$, $\Mod{A} \simeq \Mod{B}$ if and only if $A \cong B$: simply apply the Reconstruction Theorem to the categories $\QCoh(\Spec A)$ and $\QCoh(\Spec B)$ and use the equivalence of categories $\mathbf{AffSch} \simeq \Cring^{\op}$.}. However, the importance of the Reconstruction Theorem ultimately comes down to the fact that it provides a jumping-off point to begin noncommutative algebraic geometry: since equivalences $\QCoh(X) \simeq \QCoh(Y)$ imply the existence of isomorphisms $X \cong Y$, one can ``replace'' the geometric object $X$ by its category of quasi-coherent sheaves $\QCoh(X)$. By then treating arbitrary\footnote{For many practical considerations, one may want to restrict their attention to  arbitrary AB5 Abelian categories (Abelian categories with infinite coproducts and with exact filtered colimits) with a generator, as those are precisely the Abelian categories which best resemble the structural behaviour of the categories $\QCoh(X)$.} Abelian categories $\Ascr$ as some representative for a model of ``sheaves on a noncommutative space'' and use the homological algebraic study of said categories as a way to begin examining the tip of the noncommutative geometric iceberg.

The Gabriel-Rosenberg Reconstruction Theorem is proved by way of studying the prime (Gabriel) spectrum of an Abelian category and then by showing that for a quasi-separated scheme $X$, there is an equivalence
\[
\QCoh(X) \simeq \QCoh(\Spec(\QCoh(X)))
\]
if and only if there is an isomorphism of quasi-separated schemes $X \cong \Spec(\QCoh(X))$. For details regarding this argument, the Gabriel spectrum, and more, see \cite{BrandenbergReconstruction}. What is important for the scope of this article is the statement of the theorem.

\begin{Theorem}[{Reconstruction Theorem; cf.~ \cite{RosenbergSpectra}, \cite{BrandenbergReconstruction}}]\label{Thm: Reconstruction Theorem}
For any quasi-separated schemes $X$ and $Y$, $\QCoh(X) \simeq \QCoh(Y)$ if and only if there is an isomorphism of schemes $X \cong Y$.
\end{Theorem}
What is also remarkable about  this theorem is that it remains true even over relative quasi-separated schemes over affine base schemes. This means that if one knows that $X$ and $Y$ are quasi-separated of some affine scheme $\Spec A$ (such as in the case that $X$ and $Y$ are varieties, for instance) then the Reconstruction Theorem holds at the level of $A$-schemes. 
\begin{corollary}\label{Cor: Relative Reconstruction Theorem}
For any commutative ring $A$ and any quasi-separated $A$-schemes $X$ and $Y$, there is an equivalence $\QCoh(X) \simeq \QCoh(Y)$ if and only if there is an isomorphism of schemes $X \cong Y$ in $\Sch_{/A} = \Sch_{/\Spec A}$.
\end{corollary}
%

Our goal now is to show that the Gabriel-Rosenberg Reconstruction Theorem admits a tangent-categorical rephrasing by showing that it is a complete invariant on the tangent category of quasi-separated schemes. However, this means that we first necessitate introducing quasi-separated morphisms of schemes and then providing an explicit proof of the fact that the category $\qsSch_{/S}$ of quasi-separated $S$-schemes is a tangent subcategory of $\Sch_{/S}$. Additionally, we must also indicate then that the subcategory inclusion $\qsSch_{/S} \to \Sch_{/S}$ is full on bundles in the sense that if $X$ is a quasi-separated $S$-scheme and if $E \to X$ is a differential bundle in $\Sch_{/S}$ then $E \to X$ is also a differential bundle in $\qsSch_{/S}$. That is, we must show that differential bundles of a quasi-separated scheme $X$ as computed in the category of $S$-schemes coincide exactly with the differential bundles as computed in $\qsSch_{/S}$.

Separated morphisms of schemes are a scheme-theoretic recasting of what it means to be Hausdorff in structural terms. Because the underlying topological space $\lvert X \rvert$ of a scheme $X$ is very rarely Hausdorff\footnote{In fact, for a commutative ring $A$ the underlying space of the spectrum $\Spec A$ is Hausdorff if and only if the reduction $A_{\red} = A/\mathsf{nil}(A)$ is a von Neumann regular ring.}, one must take a more relative and structure-based approach towards what it means to be Hausdorff. The approach that Grothendieck defined is based on when the diagonal $\Delta:X \to X\times_S X$ is a closed immersion, mimicking the classical theorem that a space $X$ is Hausdorff if and only if the diagonal $\lbrace (x,x) \; | \; x \in X \rbrace$ is closed in $X \times X$.
\begin{definition}[{cf. \cite[D{\'e}finition 5.4.1]{EGA1}, \cite[Definition II.4.1]{Hartshorne}}]\label{Defn: Separated morphism}
Let $f:X \to S$ be a morphism of schemes. We say that $f$ is a \emph{separated morphism} if the diagonal morphism
\[
\Delta_{X\vert S}:X\to X \times_S X
\]
is a closed immersion.
\end{definition}
A major benefit of working with schemes $X \to \Spec A$ separated over an affine scheme is that for any two affine subschemes $U$ and $V$ of $X$, the intersection $U \times_X V$ is again an affine scheme.

In this light, quasi-separated morphisms of schemes are weakening of this perspective on separation. Again defined by Grothendieck in , quasi-separated morphisms are defined by asking for the diagonal $\Delta_{X|S}:X \to X \times_S X$ to be quasi-compact instead of a closed immersion. This weakening, while slightly inconvenient in many ways, is well-suited to working with quasi-coherent sheaves and provides a particularly natural and nearly minimal finiteness condition which allows for the direct image of a quasi-coherent sheaf to remain quasi-coherent.
\begin{definition}[{cf. \cite[D{\'e}finition 1.2.1]{EGA41}}]\label{Defn: qsMorphism}
A morphism $f:X \to S$ of schemes is \emph{quasi-separated} if the diagonal
\[
\Delta_{X\vert S}:X \to X \times_S X
\]
is a quasi-compact morphism of schemes. We also say that a \emph{scheme $X$ is quasi-separated} if the terminal map $\bang_X:X \to \Spec \Z$ is quasi-separated.
\end{definition}
The analogue of the affine intersection property of separated morphisms is as follows. If $X$ is a scheme with a quasi-separated morphism $X \to \Spec A$ for some ring $A$, then for any two affine open subschemes $U$ and $V$ of $X$, their intersection $U \times_X V$ may be covered by \emph{finitely many} affine open subschemes.

Below we record some of the ways in which separated and quasi-separated morphisms are stable under composition, base change, and more. These properties are all well-documented in introductory references on scheme theory, and so we defer to those resources; instead we will just focus on \emph{how} these morphisms interact and the ways in which they relate to each other and to affine morphisms. For what is below, it is worth noting that in the sketch giving the locations each relevant statement may be found each citation comes from one of Grothendieck and Dieudonn{\'e}'s {\'E}l{\'e}ments de G{\'e}om{\'e}trie Alg{\'e}brique I (\cite{EGA1}), II (\cite{EGA2}), or {\'E}l{\'e}ments de G{\'e}om{\'e}trie IV, Premi{\`e}re Partie (\cite{EGA41}) or a combination thereof.

\begin{proposition}\label{Prop: Section Gabriel: Properties of separation stuffs}
The following hold for morphisms of schemes.
\begin{enumerate}
	\item If $f:X \to S$ is separated then $f$ is quasi-separated.
	\item If $f:X\to S$ and $g:S \to T$ are separated, then so too is $g \circ f$.
	\item If $f:X\to S$ and $g:S \to T$ are quasi-separated, then so too is $g \circ f$.
	\item In a pullback diagram
	\[
	\begin{tikzcd}
		X \times_S Y \ar[r]{}{\pi_1} \ar[d, swap]{}{\pi_0} & Y \ar[d]{}{g} \\
		X \ar[r, swap]{}{f} & S
	\end{tikzcd}
	\]
	if $f$ is separated, then so too is $\pi_1$. Similarly, if $f$ is quasi-separated then so too is $\pi_1$.
	\item If $f:X \to S$ and $g:S \to T$ are morphisms for which $g \circ f$ is separated, then so too is $f$.
	\item If $f:X \to S$ and $g:S \to T$ are morphisms for which $g \circ f$ is quasi-separated, then so too is $f$.
	\item If $f:X \to S$ is an affine morphism, then $f$ is separated. In particular, $f$ is quasi-separated.
\end{enumerate}
\end{proposition}
\begin{proof}[Sketch]
Statement $(1)$ is a scheme-theoretic fact which follows from the fact that closed immersions are quasi-compact by \cite[Proposition 1.1.2.i]{EGA41}, so if $\Delta_{X|S}$ is a closed immersion it is necessarily quasi-compact as well; hence, by definition, when $f$ is separated it is necessarily quasi-separated. Statement $(2)$ is \cite[Proposition 5.5.1.ii]{EGA1}. Statement $(3)$ is \cite[Proposition I.2.2.ii]{EGA41}. The separated portion of Statement $(4)$ is \cite[Corollaire 5.5.3]{EGA1} while the quasi-separated portion is \cite[Proposition I.2.2iii]{EGA41}. Statment $(5)$ is \cite[Proposition 5.5.1.v]{EGA1}. Statement $(6)$ is \cite[Proposition I.2.2.v]{EGA41}. Finally, the fact that $f$ is affine then $f$ is separated is exactly \cite[Proposition I.2.4]{EGA2}. Thus the entirety of Statement $(7)$ follows from applying Statement $(1)$ to the observation that affine morphisms $f$ are separated.
\end{proof}

This proposition allows us to deduce that the tangent scheme $T_{X/S}$ of $X$ is itself quasi-separated over any $S$-scheme $S$. We will also be able to deduce that for any $n \in \N$ that both the iterated powers $T^n_{X/S}$ and pullbacks $T_{n}X$ are quasi-separated over $X$. As such, let us introduce the category $\qsSch_{/S}$ of quasi-separated $S$-schemes, before proceeding to prove these lemmas and then proving that $\qsSch_{/S}$ is a strict tangent subcategory of $\Sch_{/S}$.

\begin{definition}\label{Defn: Cat of qs schemes}
Let $S$ be a scheme. We define the category $\qsSch_{/S}$ to be the category:
\begin{itemize}
	\item Objects: $S$-schemes $f:X \to S$ for which $f$ is quasi-separated.
	\item Morphisms: As in $\Sch_{/S}$.
	\item Composition and Identities: As in $\Sch_{/S}$.
\end{itemize}
\end{definition}
\begin{remark}\label{Remark: Section Gabriel: All maps in qsSch are qs}
It may seem strange that at a first glance, the morphisms in $\qsSch_{/S}$ are not required to be quasi-separated. However, if we have quasi-separated $S$-schemes $f:X \to S$ and $h:Y \to S$ together with a morphism of schemes $g:X \to Y$ rendering
\[
\begin{tikzcd}
X \ar[dr, swap]{}{f} \ar[r]{}{g} & Y \ar[d]{}{h} \\
 & S
\end{tikzcd}
\]
then by Statement $(6)$ of Proposition \ref{Prop: Section Gabriel: Properties of separation stuffs} it follows that $g$ is quasi-separated. As such, for any quasi-separated $S$-schemes $X$ and $Y$, we get that $\qsSch_{/S}(X,Y) = \Sch_{/S}(X,Y)$ and hence that $\qsSch_{/S}$ is a full subcategory of $\Sch_{/S}$.
\end{remark}

\begin{lemma}\label{Lemma: Section Gabriel: Tangent scheme is qs over scheme}
Let $f:X \to S$ be a morphism of schemes. Then the bundle projection $p_{X}:T_{X/S} \to X$ is separated and in particular quasi-separated. Finally, if $f:X \to S$ is quasi-separated then so too is $f \circ p_X:T_{X/S} \to S$.
\end{lemma}
\begin{proof}
Because the morphism $p_{X}$ is affine, by Statement $(7)$ of Proposition \ref{Prop: Section Gabriel: Properties of separation stuffs} we know that $p_X$ is separated. Thus by Statement $(1)$ of Proposition \ref{Prop: Section Gabriel: Properties of separation stuffs} the map $p_X$ is quasi-separated. The final statement of the lemma follows from Statement $(3)$ of Proposition \ref{Prop: Section Gabriel: Properties of separation stuffs}.
\end{proof}
\begin{lemma}\label{Lemma: Section Gabriel: Tangent scheme pullbacks is qs over scheme}
Let $f:X \to S$ be a morphism of schemes and let $n \in \N$. Then each iterated pullback $T_{n}X$ is separated over $X$. In particular, if $f$ is quasi-separated then each $T_nX \to S$ is quasi-separated as well.
\end{lemma}
\begin{proof}
Because each iterated pullback $T_nX$ is the limit of the diagram
\[
\begin{tikzcd}
 & T_nX \ar[dr]{}{\pi_{n-1}} \ar[d]{}{\pi_i} \ar[dl, swap]{}{\pi_0} \\
T_{X/S} \ar[dr, swap]{}{p_X} & \underbrace{\cdots}_{n-2\,\text{terms}} \ar[d]{}{p_X} & T_{X/S} \ar[dl]{}{p_X} \\
 & X
\end{tikzcd}
\]
and because each morphism $p_X$ is affine, so too are all the projections $\pi_i$ for $0\leq i \leq n-1$. Thus each map $\pi_i$ is separated and hence quasi-separated. From here the fact that the composition $T_{n}X \to S$ is quasi-separated is immediate.
\end{proof}
\begin{lemma}\label{Lemma: Section Gabriel: Powers of tangents are affine}
Let $f:X \to S$ be a morphism of schemes and let $n \in \N$. Then each power of the tangent of $X$, $T^n_{X/S}$, is quasi-separated over $X$. In particular, if $f:X \to S$ is quasi-separated then so too is $T_{X/S}^n \to S$.
\end{lemma}
\begin{proof}
Because each morphism $p_{T^{n-1}X}:T^n_{X/S} \to T^{n-1}_{X/S}$  is affine, a routine induction gives that the ultimate map $T^n_{X/S} \to X$ is affine. Because affine morphisms are separated and hence quasi-separated, the map $T^n_{X/S} \to X$ is separated and hence quasi-separated as well. Finally, that the map $T^n_{X/S} \to X$ is uniquely defined (and in particular does not depend on any of the different whiskerings of $p$ by various powers of $T$), note that as
\[
\begin{tikzcd}
T_{X/S}^2 \ar[rr, shift left = 1]{}{(p \ast T)_X} \ar[rr, shift right = 1, swap]{}{(T \ast p)_X} & & T_{X/S} \ar[r]{}{p_X} & X
\end{tikzcd}
\] 
commutes, a routine induction shows that the induced (truncated) \v{C}ech nerve of the functor $T$ and the various whiskerings of the transformations $p$ displayed below
\[
\begin{tikzcd}
T^n_{X/S} \ar[rrrr, shift left = 10]{}{(p \ast T^{n-1})_X} \ar[rrrr, shift left = 5]{}[description]{(T \ast p \ast T^{n-2})_X} \ar[rrrr]{}[description]{\vdots} \ar[rrrr, shift right = 5]{}[description]{(T^{n-2} \ast p \ast T)_X} \ar[rrrr, shift right = 10, swap]{}{(T^{n-1} \ast p)_X} & & & & T^{n-1}_{X/S} \ar[rr, shift left = 5]{}{(p \ast T^{n-2})_X} \ar[rr, shift right = 5, swap]{}{(T^{n-2} \ast p)_X} \ar[rr]{}[description]{\vdots} & & \cdots \ar[rr, shift left = 1]{}{(p \ast T)_X} \ar[rr, shift right = 1, swap]{}{(T \ast p)_X}  & & T_{X/S} \ar[r]{}{p_X} & X
\end{tikzcd}
\]
has the property that the total composite $T^n_{X/S} \to X$ is independent of the path chosen. Thus, the composite $T_{X/S}^n \to S$ is uniquely defined and has the property that if $X \to S$ is quasi-separated, so too is the composite map $T^n_{X/S} \to S$.
\end{proof}

We now are nearly ready to show that the category $\qsSch_{/S}$ is a strict tangent subcategory of $\Sch_{/S}$; for this it will be helpful to keep Definition \ref{Defn: Tangent morphism} in mind. However, before proceeding we present two examples (and one remark) below which indicate why we are emphasizing the word ``strict'' when we say that $\qsSch_{/S}$ is a tangent subcategory of $\Sch_{/S}$ --- Example \ref{Example: Inclusion of varieties is not strong tangent mor} shows that a non-isomorphism subcategory inclusion need not induce a strong tangent morphism while Example \ref{Example: Identity functor not strong tangent morphism} shows that that even if a tangent morphism takes the form $A = (\id_{\Cscr},\alpha)$ it need not be the case that $A$ is strong.

\begin{example}\label{Example: Inclusion of varieties is not strong tangent mor}
	Recall that if we have a morphism $f:S \to S^{\prime}$ of schemes, there is a subcategory inclusion $f_{\ast}:\Sch_{/S} \to \Sch_{/S^{\prime}}$ of  $S$-schemes to $S^{\prime}$-schemes given by:
	\[
	\begin{tikzcd}
		X \ar[d]{}{\nu} \\
		S
	\end{tikzcd} \mapsto \begin{tikzcd}
		X \ar[d]{}{\nu} \ar[dr] \\
		S \ar[r, swap]{}{f} & S^{\prime}
	\end{tikzcd}
	\]
	We claim that the inclusion $\Sch_{/S} \to \Sch_{/S^{\prime}}$ is not always a strong tangent morphism. To see this we let $p \in \N$ be an integer prime and set $S = \Spec \Fbb_p(t^{1/p})$ and $S^{\prime} = \Spec \Fbb_p(t)$ so that the map $S \to S^{\prime}$ witnesses the alebraic field extension $\Fbb_p(t) \to \Fbb_p(t^{1/p})$. Now consider that for any finite type $\Fbb_p(t^{1/p})$-algebra $A$, if
	\[
	A \cong \frac{\Fbb_p(t^{1/p})[x_0, \cdots, x_n]}{(f_0, \cdots, f_m)}
	\]
	then 
	\[
	\Sym{A}\left(\Kah{A}{\Fbb_p(t^{1/p})}\right) \cong \frac{\Fbb_p(t^{1/p})[x_0, \cdots, x_n, \mathrm{d}x_0, \cdots, \mathrm{d}x_n]}{(f_0, \cdots, f_m, \mathrm{d}f_0, \cdots, \mathrm{d}f_m)}.
	\]
	On the other hand, such an $\Fbb_p(t^{1/p})$-algebra has the description
	\[
	A \cong \frac{\Fbb_p(t^{1/p})[x_0, \cdots, x_n]}{(f_0, \cdots, f_m)} \cong \frac{\big(\Fbb_p(t)[y]/(y^p-t)\big)[x_0, \cdots, x_n]}{(f_0, \cdots, f_m)} \cong \frac{\Fbb_p(t)[y,x_0, \cdots, x_n]}{(y^p-t, f_0, \cdots, f_m)}
	\]
	as a $\Fbb_p(t)$-algebra\footnote{With the convention that under this isomorphism, if any of the $f_j$ had coefficients in $\Fbb_p(t^{1/p})$ which are not in $\Fbb_p(t)$, we replace them with their representative coset in $\Fbb_p[y]/(y^p-t)$.} In this case, as $\mathrm{d}(y^p-t) = py^{p-1}\mathrm{d}y = 0$, we get
	\begin{align*}
		\Sym{A}\left(\Kah{A}{\Fbb_p(t)}\right) &\cong \frac{\Fbb_p(t)[y, \mathrm{d}y, x_i, \mathrm{d}x_i : 0 \leq i \leq n]}{(y^p-t,\mathrm{d}(y^p-t),f_i, \mathrm{d}f_i : 0 \leq i \leq m)} \cong \frac{\Fbb_p(t)[y,\mathrm{d}y, x_i,\mathrm{d}x_i:0 \leq i\leq n]}{(y^p-t,f_j,\mathrm{d}f_j:0 \leq j \leq m)} \\
		&\cong \frac{\Fbb_p(t^{1/p})[\mathrm{d}y,x_i,\mathrm{d}x_i:0 \leq i \leq n]}{(f_j,\mathrm{d}f_j:0 \leq j \leq m)}.
	\end{align*}
	By comparing presentations as $\Fbb_p(t^{1/p})$-algebras, we see that there is a chain of isomorphisms
	\begin{align*}
		\Sym{A}\left(\Kah{A}{\Fbb_p(t^{1/p})}\right) &\cong \frac{\Fbb_p(t^{1/p})[x_0, \cdots, x_n, \mathrm{d}x_0, \cdots, \mathrm{d}x_n]}{(f_0, \cdots, f_m, \mathrm{d}f_0, \cdots, \mathrm{d}f_m)} \\
		&\cong \left(\frac{\Fbb_p(t^{1/p})[x_0, \cdots, x_n, \mathrm{d}x_0, \cdots, \mathrm{d}x_n]}{(f_0, \cdots, f_m, \mathrm{d}f_0, \cdots, \mathrm{d}f_m)}\right) \otimes_{\Fbb_{p}(t^{1/p})} \frac{\Fbb_p(t^{1/p})[\mathrm{d}y]}{(\mathrm{d}y)} \ \\
		&\cong \frac{\Fbb_p(t^{1/p})[\mathrm{d}y, x_0, \cdots, x_n, \mathrm{d}x_0, \cdots, \mathrm{d}x_n]}{(\mathrm{d}y, f_0, \cdots, f_m, \mathrm{d}f_0, \cdots, \mathrm{d}f_m)} \\
		&\cong \left(\frac{\Fbb_p(t^{1/p})[\mathrm{d}y, x_0, \cdots, x_n, \mathrm{d}x_0, \cdots, \mathrm{d}x_n]}{(f_0, \cdots, f_m, \mathrm{d}f_0, \cdots, \mathrm{d}f_m)}\right)/(\mathrm{d}y)
	\end{align*}
	and so $\Sym{A}(\Kah{A}{\Fbb_p(t^{1/p})})$ is a quotient algebra of $\Sym{A}(\Kah{A}{\Fbb_p(t)})$. This means that in $\Sch^{\operatorname{f.t.}}_{/\Fbb_{p}(t)}$, the induced maps of affine schemes
	\[
	T_{A/\Fbb_p(t^{1/p})} = \Spec\left(\Sym{A}\Kah{A}{\Fbb_p(t^{1/p})}\right) \longrightarrow \Spec\left(\Sym{A}\Kah{A}{\Fbb_p(t)}\right) = T_{A/\Fbb_p(t)}
	\]
	is a closed immersion which is not an isomorphism. These observations glue and so show that for any finite-type $\Fbb_p(t^{1/p})$-scheme $X$, there is an induced morphism
	\[
	\alpha_X:T_{X/\Fbb_p(t^{1/p})} \to T_{X/\Fbb_p(t)}
	\]
	which is a closed immersion of a codimension one subscheme. A straightforward check shows that the morphisms $\alpha_X$ are natural in $X$ and give rise to a tangent morphism:
	\[
	\begin{tikzcd}
		\Sch_{/\Fbb_p(t^{1/p})} \ar[rrr, ""{name = U}]{}{T_{(-)/\Fbb_p(t^{1/p})}} \ar[d, swap]{}{\operatorname{incl}} &  & &\Sch_{/\Fbb_p(t^{1/p})} \ar[d]{}{\operatorname{incl}} \\
		\Sch_{/\Fbb_p(t)} \ar[rrr, swap, ""{name = D}]{}{T_{(-)/\Fbb_p(t)}} & & & \Sch_{/\Fbb_p(t)}
		\ar[from = U, to = D, Rightarrow, shorten <= 4pt, shorten >= 4pt]{}{\alpha}
	\end{tikzcd}
	\]
	By construction the map $\alpha$ is not an isomorphism, as it fails to be an isomorphism on finite type schemes. This implies that $(\operatorname{incl},\alpha)$ is a non-strong tangent morphism.
	
\end{example}
\begin{remark}\label{Remark: Vertical tan structure}
	The example above is a special case of the following phenomenon. Assume that $\Cscr$ is a tangent category with the property that for all $f:X \to Y$ the pullbacks
	\[
	\begin{tikzcd}
		V_f(X) \ar[r]{}{\iota^f} \ar[d, swap]{}{\iota_f} & Y \ar[d]{}{0_Y} \\
		TX \ar[r, swap]{}{Tf} & TY
	\end{tikzcd}
	\]
	exist and are preserved by all powers $T^m$ of the tangent functor\footnote{The pullbacks $V_{f}(X)$, for morphisms $f:X \to Y$, are differential bundles over $X$ and are often called the \emph{relative tangent bundle} (in \cite{GeoffJSDiffBunComAlg}, \cite{JSMeMapFlavoursInTanCats}) or \emph{vertical tangent bundle} (in \cite{PronkVooys}) of $TX$. When the pullback of the cospan $X \xrightarrow{f} Y \xleftarrow{p_Y} TY$, $X \times_Y TY$, the map $\iota_f:V_f(X) \to TX$ is the kernel in $\DBun(X)$ of the pairing map $\theta_f = \langle p_X,Tf\rangle:TX \to X \times_Y TY$. For more details about this construction and some of its properties see \cite{JSMeMapFlavoursInTanCats}.}. Then each slice category $\Cscr_{/S}$ has a tangent structure with tangent functor $V_{(-)}:\Cscr_{/S} \to \Cscr_{/S}$ given on objects $f:X \to S$ by $f \mapsto V_{f}(X)$. Now let $f:X \to S$ be a morphism in $\Cscr$ and note once again that there is a subcategory inclusion $f_{\ast}:\Cscr_{/X} \to \Cscr_{\S}$. 
	
	Now let $\varphi:Y \to X$ be an object in $\Cscr_{/X}$. On one hand we have that the tangent functor in $\Cscr_{/X}$ sends $\varphi:Y \to X$ to the vertical tangent bundle $V_{\varphi}(Y)$ while the tangent functor in $\Cscr_{/S}$ sends $f \circ \varphi:Y \to S$ to the vertical tangent bundle $V_{f \circ \varphi}(Y)$. Because of the pullbacks which define each bundle, we find that the outer cell of the diagram below commutes and hence induces a unique map $\alpha_{\varphi}$ rendering the overall diagram
	\[
	\begin{tikzcd}
		V_{\varphi}(Y) \ar[drr, bend left = 20]{}{f \circ \iota^{f}} \ar[ddr, swap, bend right = 20]{}{\iota_{\varphi}} \ar[dr, dashed]{}{\exists!\alpha_{\varphi}} \\
		& V_{f \circ \varphi}(Y) \ar[r]{}{\iota^{f \circ \varphi}} \ar[d, swap]{}{\iota_{f \circ \varphi}} & S \ar[d]{}{0_S} \\
		& TY \ar[r, swap]{}{T(f \circ \varphi)} & TS
	\end{tikzcd}
	\]
	commutative in $\Cscr$. The maps $\alpha_{\varphi}$ assemble to a natural transformation
	\[
	\begin{tikzcd}
		\Cscr_{/X} \ar[r, ""{name = U}]{}{V_{(-)}} \ar[d, swap]{}{f_{\ast}} & \Cscr_{/X} \ar[d]{}{f_{\ast}} \\
		\Cscr_{/S} \ar[r, swap, ""{name = D}]{}{V_{(-)}} & \Cscr_{/S}
		\ar[from = U, to = D, Rightarrow, shorten <= 4pt, shorten >= 4pt]{}{\alpha}
	\end{tikzcd}
	\]
	which makes $(f_{\ast},\alpha)$ into a tangent morphism. However, as we see in Example \ref{Example: Inclusion of varieties is not strong tangent mor} it need not be the case that $\alpha$ is an isomorphism.
\end{remark}
We now provide an example of a non-strong tangent morphism whose underlying functor is the identity functor. This example gives a ``worst-case'' situation indicating that the strength of a tangent morphism fundamentally encodes how one is moving between tangent structures and not just merely moving between categories.
\begin{example}\label{Example: Identity functor not strong tangent morphism}
Let $M$ be a smooth manifold and consider the category $\SMan_{/M}$. We now equip $\SMan_{/M}$ with two tangent structures. First define the tangent structure $\Tbb_M$ on $\SMan_{/M}$ by defining, for any manifold $f:X \to M$,
\[
T(f:X \to M) := TX \to M
\]
and adapting all structure transformations appropriately; details regarding this tangent structure may be found in \cite[Proposition 2.5]{GeoffRobinDiffStruct}. Second define the tangent structure $\Vbb_M$ on $\Sch_{/M}$ by taking the tangent functor to be the relative tangent bundle of Remark \ref{Remark: Vertical tan structure}, i.e.,
\[
V_f(X) := \left\lbrace (x,\mathbf{v}) \in TX \; | \; T_xf(\mathbf{v}) = \mathbf{0} \right\rbrace \cong TX \times_{TM} M
\]
and induce the structure transformations by pulling back appropriately. Details regarding this tangent structure may be found in \cite[Proposition 7.10]{PronkVooys}. Then the first projection $\iota_f$ of the pullback
\[
\begin{tikzcd}
V_f(X) \ar[r]{}{\iota^f} \ar[d, swap]{}{\iota_f} & M \ar[d]{}{0_M} \\
TX \ar[r, swap]{}{Tf} & TM
\end{tikzcd}
\]
is natural in objects $f:X \to M$ of the slice category $\SMan_{/M}$ and induces a tangent morphism $(\id,\iota_{(-)}):(\SMan_{/M},\Vbb_{/M}) \to (\SMan_{/M},\Tbb_{M}).$ Provided that $M$ is sufficiently structured so that there is a morphism of manifolds $f:X \to M$ for which $TX \not\cong V_f(X)$ (for instance, one may take take $M = \R$) then gives that $\iota_f$ need not be a natural isomorphism and hence that $(\id_{\SMan_{/M}},\iota_{(-)})$ is not a strong tangent morphism.
\end{example}

We now show that the category $\qsSch_{/S}$ of quasi-separated $S$-schemes is a strict tangent subcategory of $\Sch_{/S}$.

\begin{proposition}\label{Prop: Section Gabriel: qsSch is tangent sbucat}
Let $S$ be a scheme. The category $\qsSch_{/S}$ of quasi-separated $S$-schemes is a strict tangent subcategory of $\Sch_{/S}$.
\end{proposition}
\begin{proof}
Begin by noting that by Lemmas \ref{Lemma: Section Gabriel: Tangent scheme is qs over scheme}, \ref{Lemma: Section Gabriel: Tangent scheme pullbacks is qs over scheme}, and \ref{Lemma: Section Gabriel: Powers of tangents are affine} that all the induced structure morphisms $T^m_{X/S} \to S$ and $T_nX \to S$ are quasi-separated. Furthermore, by Statement $(4)$ we know that quasi-separated morphisms are stable under pullback; since the pullbacks of the form $X \times_S Y$ are the products in $\Sch_S$ and the pullbacks $X \times_Z Y$ of quasi-separated $S$-schemes remain quasi-separated, the standard categorical lemma implies that $\qsSch_{/S}$ is finitely complete and that the inclusion functor $\qsSch_{/S} \to \Sch_{/S}$ preserves and reflects finite limits. Combining all these observations imply that there is a factorization
\[
\begin{tikzcd}
\qsSch_{/S} \ar[r]{}{T_{(-)/S}} \ar[d, swap]{}{\operatorname{incl}} & \qsSch_{/S} \ar[d]{}{\operatorname{incl}} \\
\Sch_{/S} \ar[r, swap]{}{T_{(-)/S}} & \Sch_{/S}
\end{tikzcd}
\]
and that each natural transformation $p, 0, \operatorname{add}, \ell,$ and $c$ restrict to natural transformations:
\[
\begin{tikzcd}
\qsSch_{/S} \ar[rr, bend left = 20, ""{name = U}]{}{T} \ar[rr, bend right = 20, swap, ""{name = D}, equals]{}{} & & \qsSch_{/S}
\ar[from = U, to = D, Rightarrow, shorten <= 4pt, shorten >= 4pt]{}{p}
\end{tikzcd}\quad
\begin{tikzcd}
	\qsSch_{/S} \ar[rr, bend left = 20, ""{name = U}, equals]{}{} \ar[rr, bend right = 20, swap, ""{name = D}]{}{T} & & \qsSch_{/S}
	\ar[from = U, to = D, Rightarrow, shorten <= 4pt, shorten >= 4pt]{}{0}
\end{tikzcd}\quad
\begin{tikzcd}
	\qsSch_{/S} \ar[rr, bend left = 20, ""{name = U}]{}{T_2} \ar[rr, bend right = 20, swap, ""{name = D}]{}{T} & & \qsSch_{/S}
	\ar[from = U, to = D, Rightarrow, shorten <= 4pt, shorten >= 4pt]{}{\operatorname{add}}
\end{tikzcd}
\]
\[
\begin{tikzcd}
	\qsSch_{/S} \ar[rr, bend left = 20, ""{name = U}]{}{T} \ar[rr, bend right = 20, swap, ""{name = D}]{}{T^2} & & \qsSch_{/S}
	\ar[from = U, to = D, Rightarrow, shorten <= 4pt, shorten >= 4pt]{}{\ell}
\end{tikzcd}\quad
\begin{tikzcd}
	\qsSch_{/S} \ar[rr, bend left = 20, ""{name = U}]{}{T^2} \ar[rr, bend right = 20, swap, ""{name = D}]{}{T^2} & & \qsSch_{/S}
	\ar[from = U, to = D, Rightarrow, shorten <= 4pt, shorten >= 4pt]{}{c}
\end{tikzcd}
\]
Additionally, the fact that the inclusion $\operatorname{incl}:\qsSch_{/S} \to \Sch_{/S}$ reflects limits and the fact that for any quasi-separated $S$-scheme the diagram
\[
\begin{tikzcd}
T_2X \ar[r]{}{p_X \circ \pi_0} \ar[d, swap]{}{v} & X \ar[d]{}{0_X} \\
T^2X \ar[r, swap]{}{(T \ast p)_X} & TX
\end{tikzcd}
\]
is a pullback in $\Sch_{/S}$ implies that the same is true in $\qsSch_{/S}$. Thus the universality of the vertical lift remains true in $\qsSch_{/S}$ and so $\qsSch_{/S}$ is a tangent functor. Finally the fact that $\qsSch_{/S}$ is a strict tangent subcategory of $\Sch_{/S}$ is immediate from the fact that the diagram
\[
\begin{tikzcd}
	\qsSch_{/S} \ar[r]{}{T_{(-)/S}} \ar[d, swap]{}{\operatorname{incl}} & \qsSch_{/S} \ar[d]{}{\operatorname{incl}} \\
	\Sch_{/S} \ar[r, swap]{}{T_{(-)/S}} & \Sch_{/S}
\end{tikzcd}
\]
commutes strictly, so $(\operatorname{incl},\id)$ is a tangent morphism.
\end{proof}

Now that we have proved that $\qsSch_{/S}$ is a tangent category, we must prove that $\DBun_{\qsSch}(X) = \DBun_{\Sch}(X)$ for any quasi-separated $S$-scheme. In particular, this will allow us to deduce that there is an equivalence $\DBun_{\qsSch_{/S}}(X) \simeq \QCoh(X)^{\op}$, as in the non-quasi-separated case.

\begin{proposition}\label{Prop: Section Gabriel: DBun in qs is Dbun in sch}
Let $S$ be a scheme. Then for any quasi-separated scheme $X$, there is an equality of categories
\[
\DBun_{\qsSch}(X) = \DBun_{\Sch}(X).
\]
That is, $\mathsf{q}$ is a differential bundle in $\qsSch_{/S}$ if and only if $\mathsf{q}$ is a differential bundle in $\Sch_{/S}$.
\end{proposition}
\begin{proof}
The fact that if $\mathsf{q}$ is a differential bundle in $\qsSch_{/S}$ it remains the case that $\mathsf{q}$ is a differential bundle in $\Sch_{/S}$ is immediate from the fact that $\qsSch_{/S}$ is a strict tangent subcategory of $\Sch_{/S}$; as such, we need only prove that every differential bundle $\mathsf{q}$ in $\Sch_{/S}$ over a quasi-separated $S$-scheme is in turn a differential bundle in $\qsSch_{/S}$. Let $X$ be a quasi-separated $S$-scheme and assume that
\[
\mathsf{q} = \left(\begin{tikzcd}
	E \ar[d]{}{q} \\ X
\end{tikzcd}, \begin{tikzcd}
X \ar[d]{}{\zeta} \\
E
\end{tikzcd},
\begin{tikzcd}
E_2 \ar[d]{}{\sigma} \\
E
\end{tikzcd},
\begin{tikzcd}
E \ar[d]{}{\lambda} \\
TE
\end{tikzcd}\right)
\]
is a differential bundle over $X$ in $\Sch_{/S}$. By \cite[Proposition 4.25]{GeoffJSDiffBunComAlg} the projection $q:E \to X$ is affine and hence quasi-separated by Statement $(7)$ of Proposition \ref{Prop: Section Gabriel: Properties of separation stuffs}. But then the map $E \to S$ is quasi-separated, as
\[
\begin{tikzcd}
E \ar[r]{}{q} \ar[dr] & X \ar[d]{}{} \\
 & S
\end{tikzcd}
\]
commutes and the composition of quasi-separated maps is again quasi-separated by Statement $(3)$ of Proposition \ref{Prop: Section Gabriel: Properties of separation stuffs}. Arguing similarly as in the proof of Proposition \ref{Prop: Section Gabriel: qsSch is tangent sbucat} shows that $E_2, TE, \zeta, \sigma,$ and $\lambda$ are all morphisms in $\qsSch_{/S}$ and hence that $\mathsf{q}$ is a differential bundle in $\qsSch_{/S}$. While this only shows an object-wise equality
\[
\DBun_{\qsSch}(X)_0 = \DBun_{\Sch}(X)_0
\]
the fact that $\qsSch_{/S}$ is a full subcategory of $\Sch_{/S}$ allows us to complete the proof\footnote{If $\varphi:\mathsf{q} \to \mathsf{r}$ is a morphism of differential bundles in $\Sch_{/S}$ over a quasi-separated $S$-scheme, then in particular $\varphi:E \to F$ is a morphism in $\Sch_{/S}$. But as the projections $q:E \to X$ and $r:F \to X$ are quasi-separated, so too is $\varphi$.}.
\end{proof}
\begin{corollary}
For any scheme $S$ and for any quasi-separated $S$-scheme $X$, there is an equivalence of categories
\[
\DBun_{\qsSch}(X) \simeq \QCoh(X)^{\op}.
\]
\end{corollary}
\begin{proof}
This is immediate from the equality $\DBun_{\Sch}(X) = \DBun_{\qsSch}(X)$ and the corresponding equivalence $\DBun_{\Sch}(X) \simeq \QCoh(X)^{\op}$ of \cite{GeoffJSDiffBunComAlg}.
\end{proof}

With us finally having proved that $\qsSch_{/S}$ is a strict tangent subcategory of $\Sch_{/S}$ and that there is an equality of differential bundle categories $\DBun_{\qsSch}(X) = \DBun_{\Sch}(X)$ for quasi-separated schemes $X$, we can prove the tangent categorical framing of the Reconstruction Theorem. That is, we can prove that for any two quasi-separated schemes $X$ and $Y$, $X \cong Y$ if and only if $\DBun(X) \simeq \DBun(Y)$.

\begin{Theorem}[Tangent Categorical Reconstruction Theorem]\label{Thm: Tan Cat Recon Thm}
Let $X$ and $Y$ be quasi-separated schemes. Then there is an equivalence of categories of differential bundles
\[
\DBun(X) \simeq \DBun(Y)
\]
if and only if there is an isomorphism $X \cong Y$ of schemes.
\end{Theorem}
\begin{proof}
Begin by observing that by the Reconstruction Theorem, $X \cong Y$ if and only if there is an equivalence of categories $\QCoh(X) \simeq \QCoh(Y)$. However, as $\QCoh(X) \simeq \QCoh(Y)$ if and only if there is an equivalence $\QCoh(X)^{\op} \simeq \QCoh(Y)^{\op}$ of opposite categories and since $\QCoh(X)^{\op} \simeq \DBun(X)$ by \cite[Theorem 4.27]{GeoffJSDiffBunComAlg} it follows that $X \cong Y$ if and only if there is an equivalence of categories $\DBun(X) \simeq \DBun(Y)$.
\end{proof}
Applying the same chain of reasoning to Corollary \ref{Cor: Relative Reconstruction Theorem} allows us to derive the corresponding tangent-categorical analogue.
\begin{corollary}\label{Cor: Relative Tan Cat Recon Thm}
If $A$ is a commutative ring and if $X$ and $Y$ are two quasi-separated $A$-schemes then $\QCoh(X) \simeq \QCoh(Y)$ if and only if there is an isomorphism $X \cong Y$.
\end{corollary}

\bibliographystyle{amsalpha}
\bibliography{DeepDiveBib.bib}


\end{document}